\documentclass[]{article}
\usepackage[a4paper, margin=3cm]{geometry}
\usepackage{amsmath}
\usepackage{amssymb}
\usepackage{amsthm}
\theoremstyle{plain}
\newtheorem{theorem}{Theorem}[section]     
\newtheorem{lemma}[theorem]{Lemma}          
\newtheorem{proposition}[theorem]{Proposition}
\newtheorem{corollary}[theorem]{Corollary}
\newtheorem{assumption}[theorem]{Assumption}
\newtheorem{definition}[theorem]{Definition}
\newtheorem{example}[theorem]{Example}
\newtheorem{remark}[theorem]{Remark}
\usepackage{epstopdf}
\usepackage{makecell,multirow,diagbox,rotating,tabularx}
\usepackage{color}
\usepackage{booktabs}
\usepackage{algorithm}
\usepackage{algorithmic}
\usepackage{url}
\usepackage{hyperref} 
\usepackage{authblk} 
\usepackage{cleveref}

\newcommand{\email}[1]{\href{mailto:#1}{#1}}
\renewcommand{\Re}{{\mathbb{R}}}
\newenvironment{keywords}{\paragraph{Keywords:}}{}

\title{Alternating Gradient-Type Algorithm for \\ Bilevel Optimization with Inexact Lower-Level Solutions via Moreau Envelope-based Reformulation\thanks{This work was supported by the National Key R\&D Program of China (2023YFA1011400), the National Natural Science Foundation of China (12326605, 12222106, 12501429) and the  Shenzhen Fundamental Research Program (20250530150024003).}.
}

\author{
	Xiaoning Bai\thanks{Department of Mathematics, Southern University of Science and Technology, Shenzhen 518055, People's Republic of China. (\email{12331002@mail.sustech.edu.cn}).} \quad
	Shangzhi Zeng\thanks{National Center for Applied Mathematics Shenzhen, and Department of Mathematics, Southern University of Science and Technology, Shenzhen 518055, People's Republic of China. (\email{zengsz@sustech.edu.cn}).} \quad
	Jin Zhang\thanks{Corresponding author. Department of Mathematics, and National Center for Applied Mathematics Shenzhen, Southern University of
		Science and Technology, Shenzhen 518055, People's Republic of China. (\email{zhangj9@sustech.edu.cn}).} \quad
	Lezhi Zhang\thanks{Department of Computer Science, City University of Hong Kong, Hong Kong SAR, People's Republic of China, and Department of Mathematics, Southern University of Science and Technology, Shenzhen 518055, People's Republic of China. (\email{lezhzhang2-c@my.cityu.hk}).}
}

\date{}

\begin{document}
	
	\maketitle
	
\begin{abstract}
	In this paper, we study a class of bilevel optimization problems where the lower-level problem is a convex composite optimization model, which arises in various applications, including bilevel hyperparameter selection for regularized regression models. To solve these problems, we propose an Alternating Gradient-type algorithm with Inexact Lower-level Solutions (AGILS) based on a Moreau envelope-based reformulation of the bilevel optimization problem. The proposed algorithm does not require exact solutions of the lower-level problem at each iteration, improving computational efficiency. We prove the convergence of AGILS to stationary points and, under the Kurdyka-{\L}ojasiewicz (KL) property, establish its sequential convergence. Numerical experiments, including a toy example and a bilevel hyperparameter selection problem for the sparse group Lasso model, demonstrate the effectiveness of the proposed AGILS.
	
\end{abstract}

\begin{keywords}
	Bilevel optimization, hyperparameter selection, inexact, alternating gradient descent, convergence,  sparse group Lasso
	
\end{keywords}

\section{Introduction}
Bilevel optimization problems are a class of hierarchical optimization problems where the feasible region is implicitly defined by the solution set of a lower-level optimization problem. In this work, we focus on a bilevel optimization problem which has a convex composite lower-level problem,
\begin{equation}\label{eq1}
	\begin{aligned}
		\min\limits_{x \in X, y \in Y} &\ F(x,y)\\
		{\rm s.t.}\quad &\ y \in S(x),
	\end{aligned}
\end{equation}
where $S(x)$ denotes the set of optimal solutions to the lower-level problem:
\begin{equation}\label{LL}
	\min\limits_{y \in Y}\ \varphi(x,y) := f(x,y) + g(x,y).
\end{equation}
Here, $X \subset \Re^n$ and $Y \subset \Re^m$ are closed convex sets, $f(x, y) : \Re^n \times \Re^m \rightarrow \Re$ is smooth and convex with respect to $y$ on $Y$ for any $x \in X$, and $g(x, y) : \Re^n \times \Re^m \rightarrow \Re$ is convex but potentially nonsmooth with respect to $y$ on $Y$ for any $x \in X$. Additionally, we assume that $g(x, \cdot) : \Re^m \rightarrow \Re$ is proximal-friendly, enabling efficient evaluation of its proximal operator.

Bilevel optimization originates from economic game theory, particularly the Stackelberg game \cite{von2010market}, and has since found diverse applications in economics \cite{cecchini2013solving, garen1994executive, 29}, transportation \cite{constantin1995optimizing,  migdalas1995bilevel}, and machine learning \cite{bennett2008bilevel, franceschi2018bilevel, kunapuli2008classification, liu2021investigating,zhang2024introduction}. 
Notably, many hyperparameter optimization problems can be modeled as bilevel optimization problems with a convex composite lower-level problem as in \eqref{LL}. Bilevel optimization has been extensively studied, with monographs \cite{colson2007overview, dempe2002foundations,dempe2020bilevel,dempe2013bilevel} providing a comprehensive overview of its methodologies and applications.
For bilevel problems where all defining functions are smooth, the first-order optimality conditions are used to reformulate the problem as a mathematical program with equilibrium constraints (MPEC); see, e.g., \cite{allende2013solving, bard1998practical, luo1996mathematical,outrata1998nonsmooth}. However, this approach often encounters difficulties when applied to larger-scale problems or those involving non-smooth lower-level functions, such as those with regularization terms. Smoothing techniques have also been proposed to address nonsmooth issues in specific applications, such as hyperparameter optimization \cite{alcantara2025unified,okuno2021lp}.
Another widely used approach for reformulating bilevel optimization problems is based on the value function \cite{outrata1990numerical,ye1995optimality}. Specifically, the bilevel problem is reformulated as the following equivalent problem, 
\begin{equation}\label{reformulate_vlue}
	({\rm VP}) \quad\quad		\min\limits_{x\in X,y\in Y}\ F(x,y) \quad  
	{\rm s.t.}\quad\ \varphi (x,y)-v (x)\le 0,
\end{equation}
where $v(x):= \inf_{\theta\in Y} \varphi(x,\theta)$ is the value function of the lower-level problem. Under the partial calmness condition, Newton-type methods have been developed for this reformulation \cite{fischer2022semismooth,fliege2021gauss,jolaoso2025fresh}. For fully convex lower-level problems, (VP) can be treated as a difference-of-convex (DC) program, as studied in \cite{ye2023}, which led to applications in bilevel hyperparameter tuning problems \cite{gao2022value}. Subsequently, \cite{gao2024} extended these ideas to settings where the fully convex assumption is relaxed, introducing a Moreau envelope-based reformulation and corresponding algorithms. 

In recent years, bilevel optimization has garnered significant attention within the machine learning community, driving the development of various methods and algorithms. Among these, gradient-based approaches have emerged as particularly effective due to their simplicity and efficiency \cite{franceschi2018bilevel, liu2021investigating,zhang2024introduction}.
Many of these methods are built upon the computation of hypergradient or its approximation, especially in cases where the lower-level problem is smooth and strongly convex. The strong convexity of the lower-level problem guarantees the uniqueness of its solution, facilitating hypergradient computation \cite{bertrand2020implicit, chen2021closing, feng2018gradient, grazzi2020iteration, hong2023two, ji2021bilevel}. Among these, \cite{hong2023two} proposed a two-timescale single-loop stochastic gradient algorithm (TTSA) that efficiently tackles smooth bilevel problems. 
In addition, recent works have proposed gradient-based approaches based on the value function reformulation (VP) of bilevel problems, see, e.g., \cite{kwon2023fully, liu2022bome, lu2023slm, lu2024first, shen2023penalty}. These approaches typically require smoothness conditions for the lower-level problem. Notably, \cite{lu2024first} first proposed first-order penalty methods for finding $\epsilon$-KKT solutions of (VP) via solving a minimax optimization problem.

A key challenge in gradient-type algorithms derived from the value function reformulation is the computation of the gradient or subgradient of the value function at each iteration \cite{gao2022value, ye2023}. This involves solving the corresponding lower-level problem, even when the gradient exists, as in the case where  $\nabla v(x) = \nabla_x \varphi(x, y^*(x))$, with $y^*(x) = \arg\min_{\theta \in Y} \varphi(x, \theta)$ (\cite[Remark 4.14]{bonnans2013perturbation}). 
However, exactly solving the lower-level problem at each iteration can be computationally expensive, potentially limiting the efficiency of the algorithm. 
In cases where the lower-level problem satisfies uniformly strong convexity or a global PL condition, recent works \cite{ liu2022bome, lu2023slm} use inexact solutions to the lower-level problem, replacing the exact solution $y^*(x)$ in the gradient computation, constructing an approximation to the value function gradient.
But, when the lower-level problem lacks uniform strong convexity or a global PL condition, constructing an approximation to $\nabla v$ using inexact lower-level solutions becomes challenging. Specifically, using an inexact lower-level solution $y_k$ with arbitrarily small first-order residual (i.e., $\|\nabla_y \varphi(x_k, y_k)\| \rightarrow 0$) can still result in a fixed gap between the approximate gradient $\nabla_x \varphi(x_k,y_k)$ and the true gradient $\nabla v(x_k)$, even as $k \rightarrow \infty$. As an illustrative example, consider the following simple lower-level problem,
\begin{equation}\label{toy_example}
	\min_{y \in \mathbb{R}} \varphi(x,y) :=  xy^2,
\end{equation}
where $x, y \in \mathbb{R}$. For $x>0$, $y^*(x)=0$, so $v(x) = 0$ and $\nabla v(x) = 0$. However, for any vanishing sequence $\epsilon_{k}\rightarrow 0$, consider sequences $x_k= {\epsilon_k}/{2}$ and $ y_k=1$. Here, $y_k$ is an inexact stationary point to \eqref{toy_example} for $x = x_k$, with a small first-order  residual $\|\nabla_y \varphi(x_k, y_k)\|=\epsilon_k$, which can be arbitrarily small as $k \rightarrow \infty$. Despite this, the gradient approximation $\nabla_x \varphi(x_k,y_k)$ will still have a constant gap to the true gradient $\nabla v(x_k)$, since $\|\nabla_x  \varphi(x_k,y_k) - \nabla v(x_k) \|= 1$ even as $k \rightarrow \infty$. This example demonstrates the difficulty of constructing accurate value function gradient approximations and developing gradient-type algorithms with inexact solutions to the lower-level problem  for bilevel optimization when the lower-level problem does not have uniformly strong convexity.

To overcome this challenge, in this paper, we consider the following Moreau envelope-based reformulation of the bilevel optimization problem \eqref{eq1}, first introduced in \cite{gao2024},
\begin{equation}\label{reformualtion_Moreau}
	({\rm VP})_\gamma \quad\quad
	\min\limits_{x\in X,y\in Y}\  F(x,y) \quad		{\rm s.t.}\quad\ \varphi (x,y)-v_{\gamma} (x,y)\le 0,
\end{equation}
where $v_{\gamma}(x, y)$ is the Moreau envelope associated with the lower-level problem,
\begin{equation}\label{def_vg}
	v_{\gamma}(x,y):=\inf\limits_{\theta\in Y}\varphi(x,\theta)+\frac{1}{2\gamma}\|\theta-y\|^2,
\end{equation}
where $\gamma > 0$. When the lower-level problem is convex with respect to $y$,  this reformulation is equivalent to the original bilevel optimization problem \eqref{eq1} (see \cite[Theorem 1]{gao2024}).  Directional optimality conditions for $({\rm VP})_\gamma$ are discussed in \cite{bai2025optimality}. However, as shown in \cite{gao2024} and \cite[Proposition 7]{ye2023}, classical constraint qualification (CQ) conditions do not hold at any feasible point of $({\rm VP})_\gamma$, motivating us to study the following relaxed approximation problem,
\begin{equation}\label{reformualtion_Moreau_relax}
	({\rm VP})_\gamma^{\epsilon} \quad\quad
	\min\limits_{x\in X,y\in Y}\  F(x,y) \quad		{\rm s.t.}\quad\ \varphi (x,y)-v_{\gamma} (x,y)\le \epsilon,
\end{equation}
where $\epsilon > 0$. \cite[Proposition 6]{gao2024} shows that for any $\epsilon_0 > 0$, there exists $\epsilon > 0$ such that a local minimizer of $({\rm VP})_\gamma^{\epsilon}$ is $\epsilon_0$-close to the solution set of  the original bilevel optimization problem \eqref{eq1}. 
In \cite{gao2024}, a double-loop difference-of-convex algorithm was developed to solve $({\rm VP})_\gamma^{\epsilon}$ for general bilevel optimization problems. This algorithm requires solving two optimization subproblems at each iteration, including the proximal lower-level problem \eqref{def_vg}.  
Several subsequent works \cite{liumoreau, yao2024overcoming, yaoconstrained} have adopted
ideas akin to the Moreau envelope-based reformulation to  develop single-loop gradient-based bilevel algorithms. However, these works limit their convergence analysis to the decreasing property of certain merit functions, lacking rigorous theoretical support for the convergence of iterates.

In this paper, we propose an Alternating Gradient-type algorithm with Inexact Lower-level Solutions (AGILS), for solving the approximation bilevel  problem $({\rm VP})_\gamma^{\epsilon}$. The main contributions of this paper are summarized as follows.
\begin{itemize}
	\item We propose AGILS, an alternating gradient-type algorithm designed to solve the approximate bilevel optimization problem $({\rm VP})_\gamma^{\epsilon}$. A key feature of AGILS is the incorporation of a verifiable inexactness criterion for the proximal lower-level problem \eqref{def_vg}. Unlike the double-loop difference-of-convex algorithm developed in \cite{gao2024}, which requires an exact solution to the proximal lower-level problem at each iteration, AGILS allows for inexact solutions, enhancing both the flexibility and computational efficiency of the algorithm. Furthermore, to ensure the feasibility of the iterates, we incorporate an adaptive penalty parameter update strategy and a feasibility correction procedure into AGILS.
	
	\item We establish convergence results for the proposed AGILS under mild assumptions, along with clearly defined and estimable step size ranges. Specifically, we show that AGILS subsequentially converges to a Karush-Kuhn-Tucker (KKT) stationary point of $({\rm VP})_\gamma^{\epsilon}$. Moreover, under the Kurdyka-{\L}ojasiewicz (KL) property, we establish the sequential convergence of the algorithm. This analysis is non-trivial due to the inexactness in solving the proximal lower-level problem, the alternating update scheme, and the lack of Lipschitz continuity of $\nabla v_{\gamma}$. Through careful analysis and the introduction of a new merit function, we successfully establish sequential convergence.
	
	\item The proposed AGILS can efficiently handle high-dimensional bilevel optimization problems, benefiting from its use of only an inexact solution to the proximal lower-level problem \eqref{def_vg} and an alternating gradient-type update strategy. This strategy decouples the nonsmoothness of $g(x,y)$, enabling AGILS to handle nonsmooth terms effectively. To evaluate its performance, we conduct numerical experiments on a toy example and the sparse group Lasso bilevel hyperparameter selection problem. The numerical results demonstrate the efficiency of the AGILS compared to other commonly used approaches. 
\end{itemize}

This paper is organized as follows. In Section 2, we provide the  basic assumptions, and properties of the Moreau envelope-based reformulation. Section 3 presents the proposed AGILS for solving the approximation bilevel optimization problem $({\rm VP})_\gamma^{\epsilon}$. Section 4 provides the convergence analysis towards KKT stationary points, and Section 5 establishes the sequential convergence under the KL property. Finally, Section 6 presents numerical experiments on a toy example and the sparse group Lasso bilevel hyperparameter selection problem.

\section{Preliminaries}
\subsection{Notations}

Let $\mathbb{R}^n$ denote the $n$-dimensional Euclidean space, and let $\mathbb{B}$ denote the closed unit ball centered at the origin. The nonnegative and positive orthants in $\mathbb{R}^n$ are denoted by $\mathbb{R}^n_+$ and $\mathbb{R}^n_{++}$, respectively. The standard inner product and Euclidean norm are denoted by $\langle \cdot, \cdot \rangle$ and $\|\cdot\|$, respectively.
For a vector \( x \in \mathbb{R}^n \) and a closed convex set \( S \subset \mathbb{R}^n \), the distance from \( x \) to \( S \) is defined as \(\mathrm{dist}(x, S) = \min_{y \in S} \|x - y\|\). The indicator function \(\delta_S: \mathbb{R}^n \to \mathbb{R}\cup\{+\infty\}\) is defined by \(\delta_S(x) = 0\) if \(x\in S\) and \(\delta_S(x) = +\infty\) otherwise. Moreover, \(\mathcal{N}_S(x)\) denotes the normal cone to \(S\) at \(x\). The Euclidean projection operator onto \( S \) is denoted by \(\mathrm{Proj}_S\). The Cartesian product of two sets \( X \) and \( Y \) is denoted by \( X \times Y \).
Let \( h: \mathbb{R}^n \to \mathbb{R} \cup \{\infty\} \) be a proper closed function. The proximal mapping of \( h\) is defined as $\operatorname{Prox}_h(y):=\arg\min_{\theta \in \Re^m}\left\{h(\theta)+\|\theta-y\|^2 / 2\right\}$. The Fr\'{e}chet (regular) subdifferential \(\hat{\partial} h\) and the Mordukhovich (limiting) subdifferential \(\partial h\) at a point \( \bar{z} \in \mathrm{dom}\, h := \{z \mid h(z) < \infty\} \) are defined as:
\[
\begin{aligned}
	\hat{\partial} h(\bar{z}) &= \left\{ v \in \mathbb{R}^n \;\middle|\; h(z) \geq h(\bar{z}) + \langle v, z - \bar{z} \rangle + o(\|z - \bar{z}\|), \; \forall z \in \mathbb{R}^n \right\}, \\
	\partial h(\bar{z}) &= \left\{ v \in \mathbb{R}^n \;\middle|\; \exists z_k \xrightarrow{h} \bar{z}, \; \exists v_k \in \hat{\partial} h(z_k), \; v_k \to v \right\},
\end{aligned}
\]
where \( o(z) \) denotes a function satisfying \( o(z)/z\to 0 \) as \( z \downarrow 0 \), and \( z_k \xrightarrow{h} \bar{z} \) indicates \( z_k \to \bar{z} \) and \( h(z_k) \to h(\bar{z}) \) as \( k \to \infty \). If \( \bar{z} \notin \mathrm{dom}\, h \), we define \(\partial h(\bar{z}) = \emptyset\).

\subsection{Basic Assumptions}

This part outlines the assumptions on the problem data for the bilevel optimization problem studied in this work. First, we assume that the upper-level objective $F$ is smooth and bounded below.

\begin{assumption}\label{asup1}
	The function \(F : \mathbb{R}^n \times \mathbb{R}^m \to \mathbb{R}\) is bounded below on \(X \times Y\), i.e., \(\underline{F} := \inf_{(x, y) \in X \times Y} F(x, y) > -\infty\). \(F\) is continuously differentiable, and its gradients \(\nabla_x F\) and \(\nabla_y F\) are \(L_{F_x}\)- and \(L_{F_y}\)-Lipschitz continuous on \(X \times Y\), respectively.
\end{assumption}

Throughout this paper, we focus on cases where the lower-level problem is a convex optimization problem.

\begin{assumption}\label{asup2}
	For any \(x \in X\), the functions \(f(x, \cdot)\) and \(g(x, \cdot)\) are convex and defined on the closed convex set \(Y\). 
\end{assumption} 

Smoothness and weak convexity conditions are required for the lower-level objective functions. We recall that a function \(h : \mathbb{R}^n \times \mathbb{R}^m \to \mathbb{R}\) is \((\rho_1, \rho_2)\)-weakly convex on a convex set \(X \times Y\) if \(h(x, y) + \frac{\rho_1}{2} \|x\|^2 + \frac{\rho_2}{2} \|y\|^2\) is convex on \(X \times Y\). This notion generalizes convexity (\(\rho_1 = \rho_2 = 0\)). Weakly convex functions \cite{nurminskii1973quasigradient}, extends the class of convex functions to include many nonconvex functions. For weakly convex functions, the Fréchet (regular) subdifferential coincides with both the limiting subdifferential and the Clarke subdifferential (see, e.g., \cite[Proposition 3.1 and Theorem 3.6]{ngai2000approximate}).

\begin{assumption}\label{asup3}
	The lower-level objective functions \(f(x, y)\) and \(g(x, y)\) satisfy:
	\begin{itemize}
		\item[(1)] \(f(x, y)\) is continuously differentiable on \(X \times Y\), and its gradients \(\nabla_x f\) and \(\nabla_y f\) are \(L_{f_x}\)- and \(L_{f_y}\)-Lipschitz continuous on \(X \times Y\), respectively;
		\item[(2)] \(g(x, y)\) is continuous on \(X \times Y\), \(\nabla_x g(x, y)\) exists, and \(\nabla_x g(\cdot, y)\) is \(L_{g_1}\)-Lipschitz continuous on \(X\) for any \(y \in Y\), while \(\nabla_x g(x, \cdot)\) is \(L_{g_2}\)-Lipschitz continuous on \(Y\) for any \(x \in X\);
		\item[(3)] \(f(x, y)\) is \((\rho_{f_1}, \rho_{f_2})\)-weakly convex on \(X \times Y\), and \(g(x, y)\) is \((\rho_{g_1}, \rho_{g_2})\)-weakly convex on \(X \times Y\).
	\end{itemize}
\end{assumption}

\begin{remark}
	The bilevel problem \eqref{eq1} has also been studied in \cite{liumoreau}. However, in addition to Assumption \ref{asup3}, Assumption 3.2 in \cite{liumoreau} imposes an additional  Lipschitz-like continuity condition on the proximal operator $ \operatorname{Prox}_{{g}(x,\cdot)}$ 
	with respect to $x$. This assumption is challenging to verify in practice and is not required in this work.
\end{remark}

Assumption \ref{asup3}(1) implies \((\rho_{f_1}, \rho_{f_2})\)-weak convexity with \(\rho_{f_1} = \rho_{f_2} = L_f\) of $f(x,y)$ \cite[Lemma 5.7]{beck2017first}. Furthermore, the lower-level objective \(\varphi(x, y)\) is \((\rho_{\varphi_1}, \rho_{\varphi_2})\)-weakly convex on \(X \times Y\), where \(\rho_{\varphi_1} = \rho_{f_1} + \rho_{g_1}\) and \(\rho_{\varphi_2} = \rho_{f_2} + \rho_{g_2}\).

Below, we provide more specific examples of \(g(x, y)\) satisfying \cref{asup3}.

\begin{example}
	Consider $g(x, y) = x P(y)$, where \(P(y)\) is a convex and \(L_P\)-Lipschitz continuous function on \(Y\). Then \(g(x, y)\) is \((\rho_{g_1}, \rho_{g_2})\)-weakly convex on \(\mathbb{R}_+ \times \mathbb{R}^m\) for any \(\rho_{g_1}, \rho_{g_2} \ge 0\) satisfying \(\rho_{g_1} \rho_{g_2} \ge L_P^2\). 
	
	To verify this, recall that \(g(x, y)\) is \((\rho_{g_1}, \rho_{g_2})\)-weakly convex on \(X \times Y\) if and only if, for any points $(x_1, y_1), (x_2, y_2) \in X \times Y$ and $a \in [0,1]$, the approximate secant inequality holds: $g(ax_1 + (1-a)x_2, ay_1 + (1-a)y_2) \le a g(x_1, y_1) + (1-a)g(x_2,y_2) + a(1-a)/2 \cdot (\rho_{g_1}\|x_1 - x_2\|^2 + \rho_{g_2}\|y_1 - y_2 \|^2)$. Now, consider $g(x, y) = x P(y)$. For any points $(x_1, y_1), (x_2, y_2) \in \mathbb{R}_+ \times \mathbb{R}^m$ and $a \in [0,1]$, we have
	\[
	\begin{aligned}
		& g(ax_1 + (1-a)x_2, ay_1 + (1-a)y_2) - a g(x_1, y_1) - (1-a)g(x_2,y_2) \\
		= \,& (ax_1 + (1-a)x_2) P(ay_1 + (1-a)y_2) - ax_1P(y_1) - (1-a)x_2 P(y_2) \\
		\le \, & a(a -1)  x_1P(y_1) + a(1-a)x_1P(y_2) + a(1-a)x_2P(y_1) + a(a-1)x_2P(y_2) \\
		= \, & - a(1-a)\left( x_1 - x_2  \right) \left( P(y_1) - P(y_2) \right) \\
		\le \, & a(1-a)L_P \| x_1 - x_2 \| \left\| y_1 - y_2 \right\|,
	\end{aligned}
	\]
	where the first inequality follows from the convexity of $P(y)$. Therefore, when $ \rho_{g_1} \rho_{g_2} \geq L_P^2$, the approximate secant inequality is satisfied, proving the claim.
\end{example}

Next, we present two practical instances of this form, Lasso and group Lasso:

\begin{example} The following structured regularizers satisfy the above conditions:
	
	{\rm (1)} Let \(g(x, y) = x \|y\|_1\). Then $g(x, y) = x P(y)$ with \(P(y) = \|y\|_1\), which is \(\sqrt{m}\)-Lipschitz on \(\mathbb{R}^m\). Thus, $g(x,y)$ is \((m, 1)\)-weakly convex on \(\mathbb{R}_+ \times \mathbb{R}^m\).
	
	{\rm (2)} Let \(g(x, y) = \sum_{j=1}^J x_j \|y^{(j)}\|_2\), where \(y^{(j)}\) denotes the \(j\)-th group of \(y\). Then \(g(x, y) = \sum_{j=1}^J x_j P_j(y^{(j)}) \) with \(P_j(y^{(j)}) = \|y^{(j)}\|_2 \), which is \( 1 \)-Lipschitz. Thus, \(g(x, y)\) is \((1, 1)\)-weakly convex on \(\mathbb{R}^J_+ \times \mathbb{R}^m\).
	
\end{example}

\subsection{On the Moreau Envelope Function Reformulation}

This part explores key properties of the Moreau envelope function $v_\gamma(x,y)$ and the reformulation $({\rm VP})_\gamma$, as defined in \eqref{reformualtion_Moreau} and analyzed in \cite{gao2024}. These properties form the foundation for the convergence analysis of the proposed algorithm.

Under Assumption \ref{asup2}, the solution mapping
\begin{equation}\label{proximal_LL}
	S_{\gamma}(x, y) := \mathrm{argmin}_{\theta \in Y} \left\{ f(x, \theta) + g(x, \theta) + \frac{1}{2\gamma}\|\theta - y\|^2 \right\}
\end{equation}
is well-defined and single-valued  on $X \times \mathbb{R}^m$. We denote its unique solution by $\theta_\gamma^*(x, y)$.
Below, we summarize sensitivity results for $v_\gamma(x, y)$ from Theorems 2, 3, and 5 of \cite{gao2024}.

\begin{proposition}\label{grad_v}
	Suppose Assumptions \ref{asup2} and \ref{asup3} hold. For $\gamma \in (0, 1/(2\rho_{f_2} + 2\rho_{g_2}))$, the function $v_\gamma(x, y)$ is $(\rho_{v_1}, \rho_{v_2})$-weakly convex on $X \times \mathbb{R}^m$ with $\rho_{v_1} \geq \rho_{f_1} + \rho_{g_1}$ and $\rho_{v_2} \geq 1/\gamma$. Additionally, $v_\gamma(x, y)$ is differentiable on $X \times \mathbb{R}^m$ with gradient
	\begin{equation}\label{valuefinc}
		\nabla v_\gamma(x, y) = \Big( \nabla_x f(x, \theta_\gamma^*(x, y)) + \nabla_x g(x, \theta_\gamma^*(x, y)), \left({y - \theta_\gamma^*(x, y)}\right)/{\gamma} \Big).
	\end{equation}
\end{proposition}

In our algorithm's convergence analysis, we require only the partial weak convexity of $v_\gamma(x, y)$, which differs from the full weak convexity used in \cite{gao2024}. This relaxed condition allows for a larger range of the regularization parameter $\gamma$ and smaller weak convexity constants of $v _\gamma(x, y)$, enabling larger step sizes in the algorithm. The partial weak convexity of $v_\gamma(x, y)$ is established below.

\begin{proposition}\label{partial_grad_v}
	Suppose Assumptions \ref{asup2} and \ref{asup3} hold. For $\gamma \in (0, 1/(\rho_{f_2} + \rho_{g_2}))$, the function $v_\gamma(x, y)$ is $\rho_{v_1}$-weakly convex with respect to $x$ on $X$ for any $y \in Y$ with $\rho_{v_1} \geq \rho_{f_1} + \rho_{g_1}$. It is also convex with respect to $y$ on $\mathbb{R}^m$ for any $x \in X$. Moreover, $v_\gamma(x, y)$ is differentiable with gradient given by \eqref{valuefinc} on $X \times \mathbb{R}^m$, and $\theta_\gamma^*(x, y)$ and $\nabla v_\gamma(x, y)$ are continuous on $X \times \mathbb{R}^m$.
\end{proposition}

\begin{proof}
	We begin by showing that $v_{\gamma}(x,y) $ is continuously differentiable on $X \times \mathbb{R}^m$. By Assumption \ref{asup2}, the function $f(x, \theta) + g(x, \theta) + \frac{1}{2\gamma}\|\theta-y\|^2$ is level-bounded in $\theta$ locally uniformly for any $(\bar{x}, \bar{y}) \in X \times \mathbb{R}^m $. Specifically, for any $c \in \mathbb{R}$, there exist a compact set $D$ and a neighborhood $Z$ of $(\bar{x}, \bar{y})$ such that the level set $\{ \theta \in \mathbb{R}^m \mid f(x, \theta) + g(x, \theta) + \frac{1}{2\gamma}\|\theta-y\|^2 \le c\}$ is contained in $D$ for all $(x,y) \in Z$. Using Corollary 4.3(ii) from \cite{mordukhovich2018variational}, it follows that $v_{\gamma}(x,y) $ is locally Lipschitz continuous on $X \times \mathbb{R}^m$. Furthermore, Theorem 1.22, Theorem 4.1(ii) and Theorem 4.17 in \cite{mordukhovich2018variational} imply that $v_\gamma({x},{y})$ is differentiable on $X \times \mathbb{R}^m$ with its gradient given as in \eqref{valuefinc}. The partial weak convexity property of $v_\gamma({x},{y})$ can be established using arguments similar to those in the proof of Theorem 2 in \cite{gao2024}.
	
	Next, we show that $ \nabla v_\gamma({x},{y})$ is continuous on $X \times \mathbb{R}^m$ by showing the continuity of $\theta_{\gamma}^*(x,y)$ on $X \times \mathbb{R}^m$. Let $\{(x_k, y_k)\} \subset X \times \mathbb{R}^m$ be any sequence such that $(x_k, y_k) \rightarrow (\bar{x}, \bar{y})$. Since $f(x, \theta) + g(x, \theta) + \frac{1}{2\gamma}\|\theta-y\|^2$ is level-bounded in $\theta$ locally uniformly for $(\bar{x}, \bar{y}) \in X \times \mathbb{R}^m $, the sequence $\{\theta^*_{\gamma}(x_k,y_k)\}$ is bounded. Let $\bar{\theta}$ be an accumulation point of $\{\theta^*_{\gamma}(x_k,y_k)\}$. For any $\theta \in Y$, we have $f(x_k, \theta^*_{\gamma}(x_k,y_k)) + g(x_k, \theta^*_{\gamma}(x_k,y_k)) + \frac{1}{2\gamma}\|\theta^*_{\gamma}(x_k,y_k)-y_k\|^2 \le f(x_k, \theta) + g(x_k, \theta) + \frac{1}{2\gamma}\|\theta-y_k\|^2$. Taking the limit as $k \rightarrow \infty$, and using the continuity of $f$ and $g$, we obtain that for any $\theta \in Y$, $f(\bar{x}, \bar{\theta}) + g(\bar{x}, \bar{\theta}) + \frac{1}{2\gamma}\|\bar{\theta}-\bar{y}\|^2 \le f(\bar{x}, \theta) + g(\bar{x}, \theta) + \frac{1}{2\gamma}\|\theta-\bar{y}\|^2$. This implies $\bar{\theta} = \theta^*_{\gamma}(\bar{x}, \bar{y})$. Therefore, $\theta^*_{\gamma}(x_k,y_k) \rightarrow \theta^*_{\gamma}(\bar{x}, \bar{y})$, establishing the continuity of $\theta_{\gamma}^*(x,y)$ on $X \times \mathbb{R}^m$.
\end{proof}

The solution mapping $\theta_\gamma^*(x, y)$ of the proximal problem defining $v_\gamma(x, y)$ exhibits Lipschitz continuity with respect to $y$, as established below. 

\begin{lemma}\label{Lip_theta}
	Suppose Assumptions \ref{asup2} and \ref{asup3} hold. The mapping $\theta_\gamma^*(x, y)$ is $L_{\theta^*}$-Lipschitz continuous with respect to $y$ for any $x \in X$, with $L_{\theta^*} := 1 + \sqrt{1 - s/\gamma}$ where $s = \gamma/(\gamma L_{f_y} + 1)^2$.
\end{lemma}

\begin{proof}
	For any $y_1, y_2 \in \mathbb{R}^m$, let $\theta^*_1:= \theta_\gamma^*(x,y_1)$ and $\theta^*_2 := \theta_\gamma^*(x,y_2)$ be  optimal solutions to the problem $\min_{\theta \in Y} \, \varphi(x, \theta) + \frac{1}{2\gamma}\|\theta-y\|^2 $ with $y = y_1$ and $y_2$, respectively. From the first-order optimality conditions, we have
	\begin{equation*}
		0 \in\nabla_{y} f(x, \theta^*_i) + \partial_{y} g(x, \theta^*_i)  + (\theta^*_i - y_i)/\gamma + \mathcal{N}_Y (\theta^*_i), \quad \mathrm{for} ~ i = 1,2.
	\end{equation*}
	Since $\tilde{g}(x,y) := g(x,y) + \delta_{Y}(y)$ is convex with respect to $y$, we have
	\begin{equation}\label{lem_theta_eq1}
		\theta^*_i= \mathrm{Prox}_{s \tilde{g}(x,\cdot)} \left(\theta^*_i - s\left(\nabla_{y} f(x, \theta^*_i) + (\theta^*_i - y_i)/\gamma \right) \right), \quad \mathrm{for} ~ i = 1,2,
	\end{equation}
	for any $s > 0$.
	Using the nonexpansiveness of $ \mathrm{Prox}_{s \tilde{g}(x,\cdot)}$, we have
	\begin{equation}\label{lem_theta_eq3}
		\begin{aligned}
			&\| \theta^*_1  - \theta^*_2 \| \\ 
			\le \, & \big\|  \left(\theta^*_1 - s\left(\nabla_{y} f(x, \theta^*_1) + (\theta^*_1 - y_1)/\gamma \right) \right) - \left(\theta^*_2 - s\left(\nabla_{y} f(x, \theta^*_2) + (\theta^*_2 - y_2)/\gamma \right) \right) \big\|\\
			\le \, &\big\|  \left(\theta^*_1 - s\left(\nabla_{y} f(x, \theta^*_1) + (\theta^*_1 - y_1)/\gamma \right) \right) - \left(\theta^*_2 - s\left(\nabla_{y} f(x, \theta^*_2) + (\theta^*_2 - y_1)/\gamma \right) \right) \big\|\\
			&+\frac{s}{\gamma}  \|y_1 - y_2\|,
		\end{aligned}
	\end{equation}
	where the last inequality follows from the decomposition $s(\theta^*_2 - y_2)/\gamma = s(\theta^*_2 - y_1)/\gamma + s( y_1 - y_2)/\gamma$ and the triangle inequality. 
	
	Since $\nabla_y f(x,\cdot)$ is $L_{f_y}$-Lipschitz continuous on $Y$, it holds that
	\begin{equation*}
		\left\| \nabla_y f(x, \theta_1^*) + (\theta_1^* - y_1)/\gamma- \nabla_y f(x, \theta_2^*) - (\theta_2^* - y_1)/\gamma \right\| \le \left( L_{f_y} + 1/\gamma \right) \| \theta_1^* - \theta_2^* \|.
	\end{equation*}
	Furthermore, 
	since $f(x,\theta) + \frac{1}{2\gamma}\|\theta - y_1\|^2$ is $1/\gamma$-strongly convex with respect to $\theta$ on $Y$, we obtain
	\begin{equation*}\label{lem_theta_eq4}
		\begin{aligned}
			\left\langle  \nabla_{y} f(x, \theta^*_1) + (\theta^*_1 - y_1)/\gamma -  \nabla_{y} f(x, \theta^*_2) - (\theta^*_2 - y_1)/\gamma, \theta^*_1  - \theta^*_2 \right\rangle \ge \frac{1}{\gamma} \| \theta^*_1  - \theta^*_2\|^2.
		\end{aligned}
	\end{equation*}
	Now, let us square the first term in the last line of \eqref{lem_theta_eq3}. Applying the above bounds and choosing $ s =  \gamma/(\gamma L_{f_y} + 1)^2 $, we obtain
	\begin{align*}
		&\quad\,\left\| \left(\theta_1^* - s \left( \nabla_y f(x, \theta_1^*) + (\theta_1^* - y_1)/\gamma\right) \right) - \left( \theta_2^* - s \left( \nabla_y f(x, \theta_2^*) + (\theta_2^* - y_1)/\gamma \right) \right) \right\|^2 \\
		&= \| \theta_1^* - \theta_2^* \|^2 + s^2 \left\| \nabla_y f(x, \theta_1^*) + (\theta_1^* - y_1)/\gamma- \nabla_y f(x, \theta_2^*) - (\theta_2^* - y_1)/\gamma \right\|^2 \\
		&\quad - 2s \left\langle \nabla_y f(x, \theta_1^*) + (\theta_1^* - y_1)/\gamma- \nabla_y f(x, \theta_2^*) - (\theta_2^* - y_1)/\gamma,\ \theta_1^* - \theta_2^* \right\rangle
		\\ 	&\le \left( 1 + s^2 \left( L_{f_y} + 1/\gamma\right)^2 - 2s/\gamma  \right) \| \theta_1^* - \theta_2^* \|^2 \\
		&= \left( 1 - s/\gamma \right) \| \theta_1^* - \theta_2^* \|^2.
	\end{align*}
	Consequently, we have
	\[
	\begin{aligned}
		&\big\|  \left(\theta^*_1 - s\left(\nabla_{y} f(x, \theta^*_1) + (\theta^*_1 - y_1)/\gamma \right) \right) - \left(\theta^*_2 - s\left(\nabla_{y} f(x, \theta^*_2) + (\theta^*_2 - y_1)/\gamma \right) \right) \big\| \\
		\le \,& \sqrt{1- s/\gamma } \big\| \theta^*_1  - \theta^*_2 \big\|. 
	\end{aligned}
	\]
	Combining this with \eqref{lem_theta_eq3}, we derive
	\begin{equation*}\label{lem_theta_eq5}
		\| \theta^*_1  - \theta^*_2 \| 	\le  \sqrt{1- s/\gamma } \big\| \theta^*_1  - \theta^*_2 \big\| + s/\gamma \cdot\|y_1 - y_2\|,
	\end{equation*}
	leading to the desired conclusion.
\end{proof}

This section concludes with the concept of a stationary point for problem $({\rm VP})_{\gamma}^{\epsilon}$. This stationary point, under certain constraint qualifications detailed in \cite{bai2025optimality,gao2024}, can be considered a candidate for an optimal solution. 

\begin{definition}
	Let $(\bar{x},\bar{y})$ be a feasible solution of problem $({\rm VP})_{\gamma}^{\epsilon}$ with $\epsilon > 0$. We say that $(\bar{x},\bar{y})$ is a stationary/KKT point of problem $({\rm VP})_{\gamma}^{\epsilon}$ if there exists a multiplier $\lambda \ge 0$ such that
	\begin{equation*}
		\begin{cases}
			0\in \nabla F(\bar{x},\bar{y}) + \lambda \left(\nabla f(\bar{x},\bar{y}) + \partial g(\bar{x},\bar{y}) - \nabla v_{\gamma}(\bar{x},\bar{y})\right) + \mathcal{N}_{X\times Y}(\bar{x},\bar{y}), \\
			f(\bar{x},\bar{y}) + g(\bar{x},\bar{y}) - v_{\gamma}(\bar{x},\bar{y}) \leq \epsilon, \qquad \lambda \left(f(\bar{x},\bar{y}) + g(\bar{x},\bar{y}) - v_{\gamma}(\bar{x},\bar{y})-\epsilon\right) = 0.
		\end{cases}
	\end{equation*}
\end{definition}

\section{Algorithm design}
\label{sec3}
In this section, we propose a new gradient-type algorithm, Alternating Gradient-type algorithm with Inexact Lower-level Solutions (AGILS), to solve ${\rm (VP)}_\gamma^\epsilon$. Given that $f(x, y) + g(x,y) - v_{\gamma}(x, y)$ is always nonnegative, a natural approach is to consider the following penalized problem,
\begin{equation}\label{penprob}
	\min_{x \in X, y \in Y}\psi_{p_k}	(x, y) := \frac{1}{p_k}F(x, y) + f(x, y) + g(x,y) - v_{\gamma}(x, y),
\end{equation}
where $p_k$ is the penalty parameter at iteration $k$. Due to the nonsmoothness of $g(x,y)$ with respect to $y$ and the availability of its proximal operator, it is natural to adopt a proximal alternating linearization method \cite{bolte2014proximal} to construct the algorithm. Given the current iterate $(x^k, y^k)$, linearizing the smooth components of $\psi_{p_k}	(x, y)$ requires evaluating gradients, including $\nabla v_{\gamma}(x, y)$ as provided in \cref{grad_v}. Calculating the exact gradient  $\nabla v_{\gamma}(x^k, y^k)$ requires solving the proximal lower-level problem in \eqref{proximal_LL} to obtain $\theta_\gamma^*(x^k, y^k)$. However, this solution typically does not have a closed-form expression, and solving it requires an iterative optimization solver, which can be computationally expensive. To address this, we propose an inexact gradient approximation by substituting $\theta_{\gamma}^*$ with an inexact estimate $\theta^k$.

The proposed algorithm alternates between updating $y$ and $x$ while refining the inexact approximation of $\theta_\gamma^*(x^k, y^k)$. At each iteration $k=0,1,\ldots$, given the current iterate $(x^k, y^k)$, along with the feasibility-corrected iterate $\tilde{y}^k$ and corresponding $\tilde{\theta}^k$ (constructed from $y^k$ and  $\theta^{k}$ via the procedure detailed below), the next iterate $(x^{k+1}, y^{k+1})$ and the refined approximation $\theta^{k+1}$ are updated as follows.

{\bf Update $y$ (Fixing $x$).}
We compute the update direction for $y$,
\begin{equation}\label{dy}
	d_y^k:= \frac{1}{p_k} \nabla_y F\left(x^{k}, \tilde{y}^k\right)+\nabla_y f\left(x^{k}, \tilde{y}^k\right)-\frac{\tilde{y}^k-\tilde{\theta}^{k}}{\gamma},
\end{equation}
where $(\tilde{y}^k-\tilde{\theta}^{k})/{\gamma}$ is an inexact approximation of $\nabla_y v_{\gamma}(x^{k}, \tilde{y}^k) =( {\tilde{y}^k} -  \theta_{\gamma}^*(x^{k}, \tilde{y}^k))/ \gamma $ (see (\ref{valuefinc})). Using $	d_y^k$, the variable $y$ is updated as
\begin{equation}\label{update_y}
	y^{k+1} = \underset{y \in Y}{\mathrm{argmin}} ~ \left\langle d_y^k, y - \tilde{y}^{k} \right\rangle + g(x^k, y) + \frac{1}{2\beta_k} \left\| y - \tilde{y}^k \right\|^2, 
\end{equation}
where $\beta_k>0$ is the stepsize. 
This can also be expressed using the proximal operator
\begin{equation}\label{update_y2}
	y^{k+1} = \operatorname{Prox}_{\beta_k\tilde{g}(x^{k},\cdot)}\left(\tilde{y}^k-\beta_k d^k_y\right),
\end{equation}
where $\tilde{g}(x,y) := g(x,y) + \delta_Y(y)$. 

{\bf Find inexact approximation $\theta^{k+1/2}$.}
After updating $y^{k+1}$, we compute an inexact approximation $\theta^{k+1/2}$ for $\theta_{\gamma}^*(x^k, y^{k+1})$ by solving
\begin{equation}\label{theta_y}
	\theta^{k+1/2} \approx 	\underset{\theta \in Y}{\mathrm{argmin}} \left\{f(x^{k},\theta) + g(x^{k}, \theta) + \frac{1}{2\gamma}\left\|\theta-y^{k+1}\right\|^2\right\}, 
\end{equation}
subject to either an absolute or a relative inexactness criterion
\begin{equation}\label{inexacty1}
	\begin{aligned}
		\mathcal{G}(\theta^{k+1/2}, x^k,y^{k+1}) \le s_k, \quad \text{or} \quad \mathcal{G}(\theta^{k+1/2}, x^k,y^{k+1}) \le \tau_k \mathcal{G}(\theta^{k-1}, x^{k-1},y^{k-1}),
	\end{aligned}
\end{equation}
which can be compactly expressed as
\[
\mathcal{G}(\theta^{k+1/2}, x^{k}, y^{k+1})
\le
\max\left\{ s_k,\ \tau_k\,\mathcal{G}(\theta^{k-1}, x^{k-1}, y^{k-1}) \right\}.
\]
The parameters $s_k$ and $\tau_k$ control the solution's inexactness, and are required to satisfy $\sum_{k=0}^{\infty}s_k^2<\infty$ and $\tau_k:=\tau_0/(k+1)^{p_\tau}$ with constants $\tau_0 > 0$ and $p_\tau >0$. The prox-gradient residual $\mathcal{G}(\theta, x,y)$ quantifies the inexactness and is defined as
\[
\mathcal{G}(\theta, x,y):= 	\left\| \theta -\operatorname{Prox}_{\eta \tilde{g}\left(x, \cdot\right)}\left(\theta-\eta\left(\nabla_y f\left(x, \theta\right) + (\theta -y)/ {\gamma}\right)\right)\right\|,
\]
where $\eta > 0$. Note that $\mathcal{G}(\theta, x,y) = 0$ if and only if $\theta =\theta_{\gamma}^*(x,y)$. 

\begin{remark}
	A similar $\theta$-update based on solving the proximal subproblem \eqref{theta_y} is also employed by MEHA in \cite{liumoreau}. However, MEHA performs only a single proximal gradient step to update 
	$\theta$. In contrast, our method uses an inner loop governed by the inexactness criteria \eqref{inexacty1}. As we will show, this strategy yields the advantage of a larger and explicitly defined range for the step sizes.
\end{remark}

{\bf Update $x$ (Fixing $y$).}
We compute the update direction for $x$,
\begin{equation}\label{dx}
	d_x^k:= \frac{1}{p_k} \nabla_xF\left(x^k, y^{k+1}\right)+\nabla_x \varphi \left(x^k, y^{k+1}\right)-\nabla_x f\left(x^k, \theta^{k+1/2}\right)-\nabla_x g\left(x^k, \theta^{k+1/2}\right),
\end{equation}
where $\nabla_x f\left(x^k, \theta^{k+1/2}\right)+\nabla_x g\left(x^k, \theta^{k+1/2}\right)$ is an approximation to \( \nabla_x v_{\gamma}(x^k, y^{k+1}) = \nabla_x f(x^k, \theta^*_{\gamma}(x^k, y^{k+1}))+\nabla_x g(x^k, \theta^*_{\gamma}(x^k, y^{k+1}))\) (see \eqref{valuefinc}).
Using the direction $d_x^k$, we update $x$ with stepsize $\alpha_k>0$ as
\begin{equation}\label{update_x}
	x^{k+1}=\operatorname{Proj}_X\left(x^k-\alpha_k d_x^k\right).
\end{equation}

{\bf Find inexact approximation $\theta^{k+1}$.}
After updating $x^{k+1}$, a new inexact approximation $\theta^{k+1}$ of $\theta^*_{\gamma}(x^{k+1}, y^{k+1})$  is computed by approximately solving
\[
\theta^{k+1} \approx 	\underset{\theta \in Y}{\mathrm{argmin}} \left\{f(x^{k+1},\theta) + g(x^{k+1}, \theta) + \frac{1}{2\gamma}\left\|\theta-y^{k+1}\right\|^2\right\}, 
\]
such that $\theta^{k+1}$ satisfies either an absolute or a relative inexactness criterion
\begin{equation}\label{inexacty2}
	\begin{aligned}
		\mathcal{G}(\theta^{k+1}, x^{k+1},y^{k+1}) \le s_{k+1}, \quad \text{or} \quad  \mathcal{G}(\theta^{k+1}, x^{k+1},y^{k+1}) \le \tau_{k+1}\mathcal{G}(\theta^{k}, x^{k},y^{k}).
	\end{aligned}
\end{equation}

\begin{remark}
	The parameters $s_{k+1}$ and $\tau_{k+1}$ used in the inexact criteria \eqref{inexacty2} for determining $\theta^{k+1}$ can be chosen from sequences independent of those used for $\theta^{k+1/2}$, provided they satisfy the previously stated requirements.
\end{remark}

{\bf Penalty parameter update and feasibility correction.}
To encourage feasibility of the constraint $\varphi(x, y) - v_{\gamma}(x, y) \le \epsilon$, the penalty parameter $p_k$ is updated at each iteration.
Evaluating the constraint violation requires computing $ v_{\gamma}(x, y)$, which involves solving the proximal lower-level problem in \eqref{proximal_LL}. Instead, we use $\theta^{k+1}$ to estimate the constraint violation as
\begin{equation}\label{def_t}
	t^{k+1}=\max\left\{\varphi(x^{k+1},y^{k+1})-\varphi(x^{k+1},\theta^{k+1})-\frac{1}{2\gamma}\| \theta^{k+1} - y^{k+1}\|^2-\epsilon,0\right\}.
\end{equation}
The penalty parameter \(p_k\) is increased to $p_{k+1}=	p_k + \varrho_p$ (for some $\varrho_p > 0$) only when the change in the iterates, $\Delta_{k+1} := \left\| (x^{k+1}, y^{k+1}) - (x^k, \tilde{y}^k) \right\|$, is small relative to the penalty parameter (i.e., $\Delta_{k+1} < c_p/p_k$ for some $c_p > 0$), while the estimated constraint violation $t^{k+1}$ is large relative to this change (i.e., $t^{k+1} > \Delta_{k+1}/c_p$).
These two conditions can be compactly written as \begin{equation}\label{update_p_1}
	\Delta_{k+1} < c_p \min\left\{1/p_k,{t^{k+1}}\right\}.
\end{equation}

However,  simply increasing $p_{k}$ does not always guarantee feasibility. It can sometimes cause the algorithm to stagnate near an infeasible stationary point of the constraint function $\varphi(x, y) - v_{\gamma}(x, y)$. 
To mitigate this, we use the bilevel optimization structure, that the feasible region is the solution set of the convex lower-level problem \eqref{LL}. We introduce a {\it feasibility correction procedure} designed to help the iterates escape such stationary points.

The feasibility correction procedure aims to generate a corrected iterate $\tilde{y}^{k+1} $ from $y^{k+1}$. It is based on the observation that if $(x^{k+1},y^{k+1})$ is feasible, then $y^{k+1}$ should be a reasonably accurate solution to the lower-level problem \eqref{LL}  for $x= x^{k+1}$, implying that $\| \nabla_y v_{\gamma}(x^{k+1},y^{k+1})\| $ should be small. We use the magnitude of $( y^{k+1} - \theta^{k+1})/\gamma$ to approximate $\| \nabla_y v_{\gamma}(x^{k+1},y^{k+1})\| $. The feasibility correction  is triggered if:
\begin{equation}\label{theta_gap_condition}
	\left\| y^{k+1} - \theta^{k+1} \right\| \ge \frac{c_y \gamma}{p_k},
\end{equation}
where $c_y > 0$. Condition \eqref{theta_gap_condition} indicates that the iterate $(x^{k+1},y^{k+1})$ might be approaching an undesirable stationary point of the constraint function.

If condition \eqref{theta_gap_condition} holds, we attempt to correct $y^{k+1}$ by finding a candidate iterate $\tilde{y}^{k+1} $  as an inexact solution to the lower-level problem \eqref{LL} with $x= x^{k+1}$ satisfying
\begin{equation}\label{inexacty_y}
	\left\| \tilde{y}^{k+1} - \operatorname{Prox}_{ \tilde{g}(x^{k+1}, \cdot)} \left( \tilde{y}^{k+1} -  \nabla_y f(x^{k+1}, \tilde{y}^{k+1}) \right) \right\| \le \frac{c_{\tilde{y}} }{p_k} \Delta_{k+1}, 
\end{equation}
where $c_{\tilde{y}} > 0$. However, unconditionally accepting such a $\tilde{y}^{k+1} $ could adversely affect the algorithm's convergence. Therefore, to preserve convergence, $\tilde{y}^{k+1} $ is accepted only if it satisfies the following descent condition
\begin{equation}\label{decrease_new_y}
	\begin{aligned}
		&\frac{1}{p_k}F(x^{k+1}, \tilde{y}^{k+1}) + \varphi(x^{k+1}, \tilde{y}^{k+1}) - \left( \varphi(x^{k+1},  \tilde{\theta}^{k+1} ) + \frac{1}{2\gamma} \| \tilde{\theta}^{k+1} - \tilde{y}^{k+1} \|^2 \right)\\
		\le\; &
		\frac{1}{p_k}F(x^{k+1}, y^{k+1}) + \varphi(x^{k+1}, y^{k+1}) - \left( \varphi(x^{k+1}, \theta^{k+1}) + \frac{1}{2\gamma} \| \theta^{k+1} - y^{k+1} \|^2 \right),
	\end{aligned}
\end{equation}
where, $\tilde{\theta}^{k+1}$ is a new approximation of $\theta^*_{\gamma}(x^{k+1}, \tilde{y}^{k+1})$. We construct $\tilde{\theta}^{k+1}$ by first finding a point $\hat{\theta}^{k+1}$ that satisfies the inexactness criterion
\begin{equation}\label{inexacty3}
	\mathcal{G}(\hat{\theta}^{k+1}, x^{k+1},\tilde{y}^{k+1}) \le s_{k+1},
	\quad \text{or} \quad
	\mathcal{G}(\hat{\theta}^{k+1}, x^{k+1},\tilde{y}^{k+1})
	\le \tau_{k+1}\mathcal{G}(\theta^{k}, x^{k},y^{k}).
\end{equation}
and subsequently applying a single proximal-gradient step to $\hat{\theta}^{k+1}$
\[
\tilde{\theta}^{k+1}
=
\operatorname{Prox}_{\eta \tilde g(x^{k+1},\cdot)}
\left(
\hat{\theta}^{k+1}
-\eta\left(
\nabla_y f(x^{k+1},\hat{\theta}^{k+1})
+\frac{\hat{\theta}^{k+1}-\tilde y^{k+1}}{\gamma}
\right)
\right).
\]

If the descent condition \eqref{decrease_new_y} holds, the algorithm adopts the candidate  $\tilde{y}^{k+1} $ along with its corresponding $\tilde{\theta}^{k+1}$, and keep the penalty parameter unchanged, i.e., $p_{k+1}=	p_k$. Otherwise, the correction is discarded, meaning $\tilde{y}^{k+1} = y^{k+1}$ and $\tilde{\theta}^{k+1} = \theta^{k+1}$ and the penalty parameter is increased, i.e., $p_{k+1}=	p_k + \varrho_p$. The algorithm then proceeds with the selected pair $(\tilde{y}^{k+1}, \tilde{\theta}^{k+1})$. A detailed description of this feasibility correction procedure is presented in Algorithm \ref{alg2}.

\begin{remark}		
The additional computational cost of the feasibility correction step is modest compared with the main procedure of the algorithm. Indeed, the subproblems for $\tilde{y}^{k+1}$ and $\tilde{\theta}^{k+1}$ are instances of the lower-level problem~\eqref{LL} or its proximal variant \eqref{proximal_LL} and thus have the same structure and complexity as the $\theta^k$-update. An efficient choice for solving these subproblems is the proximal gradient method, which we employ in our numerical experiments.		
\end{remark}

With these components established, we are now ready to present the proposed AGILS algorithm in \cref{alg1}.

\begin{algorithm}[htp!]
\caption{Alternating Gradient-type algorithm with Inexact Lower-level Solutions (AGILS)}
\label{alg1}
\begin{algorithmic}[1]
	\STATE{{\bf Input}: Initial iterates $x^0,\theta^{-1}, y^0, \tilde{\theta}^0, \tilde{y}^0$, set $(x^{-1}, y^{-1}, \theta^{-1}) = (x^0, y^0, \tilde{\theta}^0)$, stepsizes $\alpha_k, \beta_k$, relaxation parameter $\epsilon$, proximal parameter $\gamma$, penalty parameters $p_{0}$,  $\varrho_p$, $c_p$, $c_y$, inexact parameters $\eta$, $s_{k}$, $\tau_k$, tolerance $tol$.}
	\FOR{$k=0,\ 1,\cdots$}
	\STATE{Construct the direction $d_y^k$ according to \eqref{dy}, and update $y^{k+1}$ as
		$$
		y^{k+1} = \operatorname{Prox}_{\beta_k\tilde{g}(x^{k,\cdot})}\left(\tilde{y}^k-\beta_k d^k_y\right).
		$$}
	
	\STATE{Find an inexact solution $\theta^{k+1/2}$ that satisfies the inexactness criterion \eqref{inexacty1}.}
	
	\STATE{Construct the direction $d_x^k$ according to \eqref{dx}, and update $x^{k+1}$ as 
		$$ x^{k+1}=\operatorname{Proj}_X\left(x^k-\alpha_k d_x^k\right). 
		$$}
	
	\STATE{Find an inexact solution $\theta^{k+1}$ that satisfies the inexactness criterion \eqref{inexacty2}.}
	
	\STATE{Calculate $\Delta_{k+1} := \left\| (x^{k+1}, y^{k+1}) - (x^k, \tilde{y}^k) \right\|$ and $t^{k+1}$ according to \eqref{def_t}. \\
		Stop when $k \ge 1$ and \(\max\{\Delta_{k+1}, s_{k}, t^{k+1} \} \le tol.\)
	}
	\STATE{Update penalty parameter and apply feasibility correction:}
	
	\STATE{\hspace*{10pt}\textbf{Case 1:}
		$\Delta_{k+1} \geq c_p \min\big\{1/p_k,\, t^{k+1} \big\}$ \\
		\vspace*{1pt}
		\qquad\qquad\quad Set \( p_{k+1} := p_k \), \( \tilde{y}^{k+1} := y^{k+1} \), \(\tilde{\theta}^{k+1} :=\theta^{k+1}\).}
	\vspace*{2pt}
	\STATE{\hspace*{10pt}\textbf{Case 2:} 
		$\Delta_{k+1} < c_p \min\big\{1/p_k,\, t^{k+1} \big\},\,\| y^{k+1} - \theta^{k+1} \| \le c_y \gamma/p_k$\\  
		\vspace*{1pt}
		\qquad\qquad\quad Set \( p_{k+1} := p_k + \varrho_p \), \( \tilde{y}^{k+1} := y^{k+1} \), \(\tilde{\theta}^{k+1} :=\theta^{k+1}\).}
	\vspace*{2pt}
	
	\STATE{\hspace*{10pt}\textbf{Case 3:} 
		$\Delta_{k+1} < c_p \min\big\{1/p_k,\, t^{k+1} \big\},\,\| y^{k+1} - \theta^{k+1} \| > c_y \gamma/p_k$\\
		\vspace*{1pt}
		\qquad\qquad\quad Apply the feasibility correction procedure (Algorithm~\ref{alg2}).}
\ENDFOR
\RETURN $(x^{k+1},\tilde{y}^{k+1})$
\end{algorithmic}
\end{algorithm}

\begin{algorithm}[htp!]
\caption{Feasibility Correction Procedure}
\label{alg2}
\begin{algorithmic}[1]

\STATE{{\bf Input}: Current iterates $x^{k+1}, y^{k+1}, \theta^{k+1}$, $\Delta_{k+1}$, proximal parameter $\gamma$, penalty parameters $p_{k}$, $\varrho_p$, inexact parameters $\eta$, $s_{k}$, $\tau_k$, constant $c_{\tilde{y}}$.}

\STATE{Compute a candidate $\tilde{y}^{k+1}$ by inexactly solving the lower-level problem \eqref{LL} with $x=x^{k+1}$  to satisfy \eqref{inexacty_y}.
}

\STATE{Find an intermediate point $\hat{\theta}^{k+1}$ that satisfies the inexactness criterion \eqref{inexacty3}, and compute $\tilde{\theta}^{k+1}=
\operatorname{Prox}_{\eta \tilde g(x^{k+1},\cdot)}
(
\hat{\theta}^{k+1}
-\eta(
\nabla_y f(x^{k+1},\hat{\theta}^{k+1})
+(\hat{\theta}^{k+1}-\tilde y^{k+1})/\gamma
)
)$.}

\IF{the descent condition \eqref{decrease_new_y} holds}
\STATE{Accept the candidate: Set $p_{k+1} := p_k$.}
\ELSE
\STATE{Reject the candidate: Set $\tilde{y}^{k+1} := y^{k+1}$, $\tilde{\theta}^{k+1} := \theta^{k+1}$ and $p_{k+1} := p_k + \varrho_p$.}
\ENDIF
\RETURN $(\tilde{y}^{k+1},\tilde{\theta}^{k+1},p_{k+1})$
\end{algorithmic}
\end{algorithm}

\vspace*{-4pt}
\section{Convergence properties}
In this section, we analyze the convergence of the proposed  AGILS for solving $({\rm VP})_{\gamma}^{\epsilon}$. Specifically, we establish the subsequential convergence of AGILS to the KKT points of the $({\rm VP})_\gamma^{\bar{\epsilon}}$ for some  $\bar{\epsilon}\le\epsilon$. 

\subsection{Preliminary results for the analysis}

We begin by considering the following lemma, which provides an estimate for the distance between the approximation $\theta$ and the exact solution $\theta^*_{\gamma}(x,y)$ of the proximal lower-level problem, based on the prox-gradient residual $\mathcal{G}(\theta, x,y)$ in the inexact criteria \eqref{inexacty1} and \eqref{inexacty2}.  This estimate is derived from the $1/\gamma$-strong convexity and $(1/\gamma + L_f)$-smoothness of the component $f(x, \theta) + \|\theta - y\|^2/(2\gamma)$ with respect to $\theta$ in the objective of the proximal lower level problem, see, e.g., \cite[Theorem 3.4 and 3.5]{drusvyatskiy2018error} or \cite[Proposition 2.2]{mordukhovich2023globally}.
\begin{lemma}\label{error_bound}
	Suppose Assumptions \ref{asup2} and \ref{asup3} hold, and let $\gamma > 0$, $\eta > 0$. Then, there exists $C_\eta > 0$ such that for any $(x,y) \in X \times \Re^m$, and $\theta \in \Re^m$, it holds that
	\[
	\|\theta - \theta^*_{\gamma}(x,y) \| \le C_\eta  \left\|\theta-\operatorname{Prox}_{\eta \tilde{g}\left(x, \cdot\right)}\left(\theta-\eta\left(\nabla_y f\left(x, \theta\right)+ (\theta-y)/{\gamma}\right)\right)\right\|,
	\]
	where $\theta^*_{\gamma}(x,y) := \arg\min_{\theta \in Y} \{ \varphi (x,\theta)+\frac{1}{2\gamma}\|\theta-y\|^2 \}$.
\end{lemma}

The next lemma provides important inequalities involving $v_\gamma$ and plays a crucial role in proving the decreasing property of the proposed algorithm. 
\begin{lemma}\label{lem9}
	Suppose Assumptions \ref{asup2} and \ref{asup3} hold. Let $\gamma \in (0, 1/(\rho_{f_2} +\rho_{g_2}) )$ and  $(\bar{x},\bar{y})\in X\times \Re^m$. Then, for any $(x,y) \in X \times \Re^m$, the following inequalities hold,
	\begin{align}
		-v_{\gamma}(x,\bar{y}) &\le -v_{\gamma}(\bar{x},\bar{y})-\langle \nabla_x v_{\gamma}(\bar{x},\bar{y}),x-\bar{x}\rangle +\frac{\rho_{f_1} + \rho_{g_1}}{2}\|\bar{x}-x\|^2,  \label{lem9_eq1} \\
		-v_{\gamma}(\bar{x},y) &\le -v_{\gamma}(\bar{x},\bar{y})-\langle \nabla_y v_{\gamma}(\bar{x},\bar{y}),y-\bar{y}\rangle. \label{lem9_eq2}
	\end{align}
\end{lemma}
\begin{proof}
	According to Proposition \ref{partial_grad_v}, for any $\bar{y} \in \Re^m$, $v_{\gamma}(x, \bar{y})$ is a $(\rho_{f_1} + \rho_{g_1})$-weakly convex function with respect to $x$ on $X$, that is, $v_{\gamma}(x,\bar{y})+\frac{\rho_{f_1} + \rho_{g_1}}{2}\|x\|^2$ is convex on $X$. Then we can obtain \eqref{lem9_eq1} immediately. The second inequality \eqref{lem9_eq2} follows from Proposition \ref{partial_grad_v} that for any $\bar{x} \in X$, $v_{\gamma}(\bar{x}, y)$ is convex with respect to $y$.
\end{proof}

We now proceed to establish a decreasing property for the proposed method based on the following merit function:
\begin{equation*}
	\tilde{\psi}_{p_{k}}(x,y) :=\frac{1}{p_k}(F(x,y)-\underline{F})+ f(x,y) + g(x,y)-v_{\gamma}(x,y).
\end{equation*}

Before establishing the main decreasing property, we first demonstrate that the descent condition \eqref{decrease_new_y}, imposed by the feasibility correction procedure (Algorithm \ref{alg2}), ensures a decrease in $\tilde{\psi}_{p_{k}}$ when moving from $y^{k+1}$ to the corrected iterate $\tilde{y}^{k+1} $.
\begin{lemma}\label{lem0}
	Suppose Assumptions \ref{asup2} and \ref{asup3} hold. Then, there exists $C_\gamma>0$,  such that the sequence $\{(x^k,y^k,\tilde{y}^k)\}$ generated by AGILS (\cref{alg1}) satisfies
	\begin{equation}\label{lem1_eq12}
		\tilde{\psi}_{p_k}(x^{k+1}, \tilde{y}^{k+1}) \le  \tilde{\psi}_{p_k}(x^{k+1}, y^{k+1}) + C_{\gamma} \zeta_{k+1}^2,
	\end{equation}
	where  $\zeta_{k} := \max\{ s_k,  ( \prod_{i=0}^{k} \tau_i ) \mathcal{G}(\theta^{0}, x^{0}, y^{0})\}$.
\end{lemma}
\begin{proof}
	At a given iteration $k$, if $\tilde{y}^{k+1} = y^{k+1}$, the inequality holds trivially.
	
	Now, consider the case where $\tilde{y}^{k+1} \neq y^{k+1}$. This implies that the feasibility correction procedure (Algorithm \ref{alg2}) is applied at iteration $k$. Accordingly, the corrected iterate  $\tilde{y}^{k+1}$ satisfies condition \eqref{decrease_new_y}, and its corresponding $\tilde{\theta}^{k+1}$ satisfies condition \eqref{inexacty3}. Condition \eqref{inexacty3} implies that $\mathcal{G}(\hat{\theta}^{k+1}, x^{k+1}, \tilde{y}^{k+1}) \le \zeta_{k+1}$.

	By convexity of the lower-level objective in \(\theta\), the function $f(x,\theta) + \tilde{g}(x,\theta) + \frac{1}{2\gamma}\|\theta-y\|^2$ is $1/\gamma$-strongly convex with respect to $\theta$ on $Y$. Thus, it satisfies a global quadratic growth condition and, by \cite[Proposition 2]{ye2021variational}, a global KL property with exponent $1/2$ on $Y$. Noting that $\tilde{g}(x,\theta) = g(x,\theta) + \delta_Y(\theta)$, we have $\tilde{g}(x,\theta) = g(x,\theta)$ for all $\theta \in Y$. This yields 
	\begin{equation}\label{KL}
	\begin{aligned}
		f(x,\theta)+g(x,\theta)+\frac{1}{2\gamma}\|\theta-y\|^2-v_\gamma(x,y)\le
		\frac{\gamma}{2}\,
		\operatorname{dist}\Bigl(
		0,\,
		\nabla_y f(x,\theta)+\frac{1}{\gamma}(\theta-y)+\partial \tilde g(x,\theta)
		\Bigr)^2.
	\end{aligned}
\end{equation}
The optimality condition of the proximal step defining $\tilde{\theta}^{k+1}$,
\[
\tilde{\theta}^{k+1}
=
\operatorname{Prox}_{\eta \tilde g(x^{k+1},\cdot)}
\left(
\hat{\theta}^{k+1}
-\eta\left(
\nabla_y f(x^{k+1},\hat{\theta}^{k+1})
+\frac{1}{\gamma}(\hat{\theta}^{k+1}-\tilde y^{k+1})
\right)
\right),
\]
yields
\[
-\nabla_y f(x^{k+1},\hat{\theta}^{k+1})
-\frac{1}{\gamma}(\hat{\theta}^{k+1}-\tilde y^{k+1})
-\frac{1}{\eta}(\tilde{\theta}^{k+1}-\hat{\theta}^{k+1})
\in
\partial \tilde g(x^{k+1},\tilde{\theta}^{k+1}).
\]
Adding $\nabla_y f(x^{k+1},\tilde{\theta}^{k+1})+(\tilde{\theta}^{k+1}-\tilde y^{k+1})/\gamma$ to both sides gives
\[
\begin{aligned}
	&\nabla_y f(x^{k+1},\tilde{\theta}^{k+1})
	+\frac{1}{\gamma}(\tilde{\theta}^{k+1}-\tilde y^{k+1}) -\nabla_y f(x^{k+1},\hat{\theta}^{k+1})
	-\frac{1}{\gamma}(\hat{\theta}^{k+1}-\tilde y^{k+1})
	-\frac{1}{\eta}(\tilde{\theta}^{k+1}-\hat{\theta}^{k+1}) \\
	\in \ &
	\nabla_y f(x^{k+1},\tilde{\theta}^{k+1})
	+\frac{1}{\gamma}(\tilde{\theta}^{k+1}-\tilde y^{k+1})
	+\partial \tilde g(x^{k+1},\tilde{\theta}^{k+1}).
\end{aligned}
\]
	By the $L_{f_y}$-Lipschitz continuity of $\nabla_y f(x,\cdot)$ and the definition of the gradient mapping $\mathcal{G}$, the norm of the left-hand element is bounded by
	\[
	\begin{aligned}
		&\quad\ 
		\Bigl\|
		\nabla_y f(x^{k+1},\tilde{\theta}^{k+1})
		-\nabla_y f(x^{k+1},\hat{\theta}^{k+1})
		+\frac{1}{\gamma}(\tilde{\theta}^{k+1}-\hat{\theta}^{k+1})
		+\frac{1}{\eta}(\hat{\theta}^{k+1}-\tilde{\theta}^{k+1})
		\Bigr\| \\
		&\le
		\Bigl(L_{f_y}+\frac{1}{\gamma}+\frac{1}{\eta}\Bigr)
		\|\tilde{\theta}^{k+1}-\hat{\theta}^{k+1}\| \\
		&=
		\Bigl(L_{f_y}+\frac{1}{\gamma}+\frac{1}{\eta}\Bigr)
		\mathcal G(\hat{\theta}^{k+1},x^{k+1},\tilde y^{k+1}),
	\end{aligned}
	\]
Therefore, by the definition of the distance from the origin to a set,
	$$
	\begin{aligned}
		&\operatorname{dist}\left(0, \nabla_y f(x^{k+1},\tilde{\theta}^{k+1}) + \frac{1}{\gamma}(\tilde{\theta}^{k+1}-\tilde{y}^{k+1}) + \partial \tilde{g}(x^{k+1},\tilde{\theta}^{k+1})\right) \\
		&\le \left(L_{f_y} + \frac{1}{\gamma} + \frac{1}{\eta}\right) \mathcal{G}(\hat{\theta}^{k+1}, x^{k+1}, \tilde{y}^{k+1}).
	\end{aligned}
	$$
	Applying the KL inequality \eqref{KL} with $(x^{k+1}, \tilde{y}^{k+1}, \tilde{\theta}^{k+1})$ and substituting the above estimate gives
	$$
	\begin{aligned}
		&f(x^{k+1},\tilde{\theta}^{k+1}) + g(x^{k+1},\tilde{\theta}^{k+1}) + \frac{1}{2\gamma}\|\tilde{\theta}^{k+1}-\tilde{y}^{k+1}\|^2 - v_\gamma(x^{k+1},\tilde{y}^{k+1}) \\
		&\le \frac{\gamma}{2} \left(L_{f_y} + \frac{1}{\gamma} + \frac{1}{\eta}\right)^2 \mathcal{G}(\hat{\theta}^{k+1}, x^{k+1}, \tilde{y}^{k+1})^2.
	\end{aligned}
	$$
	Letting $C_{\gamma} := \frac{\gamma}{2} (L_{f_y} + 1/\gamma + 1/\eta)^2$ and recalling $\varphi(x,\theta) = f(x,\theta) + g(x,\theta)$, this rearranges to:
	$$
	v_\gamma(x^{k+1},\tilde{y}^{k+1}) \ge \varphi(x^{k+1},\tilde{\theta}^{k+1}) + \frac{1}{2\gamma}\|\tilde{\theta}^{k+1}-\tilde{y}^{k+1}\|^2 - C_{\gamma}\zeta_{k+1}^2.
	$$

	By condition \eqref{decrease_new_y}, which $\tilde{y}^{k+1}$ satisfies, we have
	\begin{equation*}
		\begin{aligned}
			&\quad \tilde{\psi}_{p_k}(x^{k+1}, \tilde{y}^{k+1}) - \tilde{\psi}_{p_k}(x^{k+1}, y^{k+1}) \\
			& \le - v_\gamma(x^{k+1}, \tilde{y}^{k+1}) + \left( \varphi(x^{k+1},  \tilde{\theta}^{k+1} ) + \frac{1}{2\gamma} \| \tilde{\theta}^{k+1} - \tilde{y}^{k+1} \|^2 \right)\\
			& \quad +  v_\gamma(x^{k+1}, {y}^{k+1}) - \left( \varphi(x^{k+1}, \theta^{k+1}) + \frac{1}{2\gamma} \| \theta^{k+1} - y^{k+1} \|^2 \right) \\
			&\le C_{\gamma} \zeta_{k+1}^2,
		\end{aligned}
	\end{equation*}
	where the last inequality follows from the definition of $v_\gamma(x^{k+1}, {y}^{k+1})$ as an infimum $ \inf_{\theta\in Y} \{ \varphi(x^{k+1}, {\theta}) + \frac{1}{2\gamma} \| {\theta} - y^{k+1} \|^2 \}$. Thus, the conclusion follows.
\end{proof}

Next, we establish a key recursive inequality for  $\tilde{\psi}_{p_{k}}$, demonstrating a sufficient decrease property for the iterates generated by Algorithm \ref{alg1}.

\begin{lemma}\label{lem1}
	Suppose Assumptions \ref{asup1}, \ref{asup2} and \ref{asup3} hold. Let $\gamma \in (0, 1/(\rho_{f_2} +\rho_{g_2}) )$, $\alpha_k \in [0, 2/(L_{\psi_{x,k}} + c_\alpha)]$ and $\beta_k \in [0, 2/(L_{\psi_{y,k}} + c_\beta)]$, where $L_{\psi_{x,k}} = L_{F_x}/p_k+L_{f_x}+L_{g_1}+\rho_{f_1} + \rho_{g_1}$, $L_{\psi_{y,k}} = L_{F_y}/p_k+L_{f_y}$ and $c_\alpha, c_\beta$ are positive constants. Then, there exists $C_\eta , C_\gamma > 0$ such that the sequence $\{(x^k,y^k,\tilde{y}^k)\}$ generated by AGILS {\rm (}\cref{alg1}{\rm )} satisfies the following inequality,
	\begin{equation}\label{lem1_eq}
		\begin{aligned}
			\tilde{\psi}_{p_{k+1}}\left(x^{k+1}, \tilde{y}^{k+1}\right)\leq\,& \tilde{\psi}_{p_k}\left(x^k, \tilde{y}^k\right) +  \left( \frac{C_\eta^2}{c_\alpha}(L_{f_x} + L_{g_2})^2 + \frac{C_\eta^2}{c_\beta\gamma^2 } \right) \zeta_{k}^2 + C_{\gamma} \zeta_{k+1}^2\\
			&- \frac{c_\alpha}{4} \left\|x^{k+1}-x^k\right\|^2  -\frac{c_\beta}{4} \left\|y^{k+1}-\tilde{y}^k\right\|^2,
		\end{aligned}
	\end{equation}
	where  $\zeta_{k} := \max\{ s_k,  ( \prod_{i=0}^{k} \tau_i ) \mathcal{G}(\theta^{0}, x^{0}, y^{0})\}$ and $\{ \zeta_{k} \}$ is square summable.
\end{lemma}

\begin{proof}
	
	First, the update rule for \( y^{k+1} \) given in \eqref{update_y} defines it as the minimizer of the function $ g(x^k, y) + \langle d_y^k, y - \tilde{y}^{k} \rangle + \frac{1}{2\beta_k} \left\| y - \tilde{y}^k \right\|^2$. This function is $1/\beta_k$-strongly convex in $y$. By applying \cite[Theorem 6.39]{beck2017first}, we obtain
	\begin{equation}\label{lem1_eq5}
		g(x^{k}, y^{k+1}) + \langle d_y^k, y^{k+1}-\tilde{y}^k\rangle+\frac{1}{\beta_k}\|y^{k+1}-\tilde{y}^k\|^2 \le g(x^{k}, \tilde{y}^{k}).
	\end{equation}

	According to the assumptions, with $x^{k} \in X$, we have that $\nabla_y F(x^{k},y)$ and $\nabla_y f(x^{k},y)$ are $L_{F_y}$- and  $L_{f_y}$-Lipschitz continuous with respect to variable $y$ on $Y$, respectively, and combine this with \eqref{lem9_eq2}, we have
	\begin{equation}\label{lem1_eq6}
		\begin{aligned}
			&(\tilde{\psi}_{p_k}-g)(x^{k},y^{k+1})  \\
			\le \, & (\tilde{\psi}_{p_k}-g)	(x^{k},\tilde{y}^k) + \langle \nabla_y (\tilde{\psi}_{p_k}-g)(x^{k},\tilde{y}^k),y^{k+1}-\tilde{y}^k\rangle +\frac{L_{\psi_{y,k}}}{2}\|y^{k+1}-\tilde{y}^k\|^2,
		\end{aligned}
	\end{equation}
	where $L_{\psi_{y,k}}=L_{F_y}/p_k+L_{f_y}$. 
	Combining \eqref{lem1_eq5} and \eqref{lem1_eq6} yields
	\begin{equation}\label{lem1_eq7}
		\begin{aligned}
			\tilde{\psi}_{p_k}\left(x^{k}, y^{k+1}\right) \leq \,& 	\tilde{\psi}_{p_k}\left(x^{k}, \tilde{y}^k\right) + \langle \nabla_y (	\tilde{\psi}_{p_k}-g)(x^{k},\tilde{y}^k) - d_y^k, y^{k+1}-\tilde{y}^k\rangle  \\
			&-\left(\frac{1}{\beta_k}-\frac{L_{\psi_{y,k}}}{2}\right)\left\|y^{k+1}-\tilde{y}^k\right\|^2.
		\end{aligned}
	\end{equation}	
	Using the formula of \( \nabla_y v_\gamma \) in \eqref{valuefinc}, and the construction of $d_y^k$ in \eqref{dy}, we have
	\begin{equation}\label{lem1_eq11}
		\begin{aligned}
			\left\| \nabla_y (	\tilde{\psi}_{p_k}-g)(x^{k},\tilde{y}^k) - d_y^k \right\| & = \| \nabla_y v_\gamma(x^k, \tilde{y}^k) + (\tilde{y}^k-\theta^{k})/\gamma\| \\
			& = \frac{1}{\gamma}\left\| \theta^*_{\gamma}(x^{k},\tilde{y}^k) - \theta^{k} \right\| \le \frac{C_\eta}{\gamma} \mathcal{G}({\theta}^{k}, x^{k}, \tilde{y}^{k}) \le \frac{C_\eta}{\gamma} \zeta_{k},
		\end{aligned}
	\end{equation}
	where the first inequality follows from Lemma \ref{error_bound} and the second inequality is a consequence of the inexactness condition \eqref{inexacty2}. 
	
	Combining the above inequality with \eqref{lem1_eq7} gives us
	\begin{equation}\label{lem1_eq8}
		\tilde{\psi}_{p_k}\left(x^{k}, y^{k+1}\right) \leq \,  	\tilde{\psi}_{p_k}\left(x^{k}, \tilde{y}^k\right)  - \left(\frac{1}{\beta_k}-\frac{L_{\psi_{y,k}}}{2}- \frac{c_\beta}{4}\right)\left\|y^{k+1}-\tilde{y}^k\right\|^2 + \frac{C_\eta^2}{c_\beta\gamma^2} \zeta_{k}^2.
	\end{equation}
	Next, considering the update rule for $x^{k+1}$ given in \eqref{update_x}, we have
	\begin{equation}\label{lem1_eq1}
		\langle d_x^k, x^{k+1}-x^k\rangle+\frac{1}{\alpha_k}\|x^{k+1}-x^k\|^2 \le 0.
	\end{equation}	
	By the assumptions, with $y^{k+1} \in Y$, we have that $\nabla_x F(x, y^{k+1})$, $\nabla_x f(x, y^{k+1})$ and $\nabla_{x} g(x, y^{k+1})$ are $L_{F_x}$-, $L_{f_x}$-, and $L_{g_1}$-Lipschitz continuous with respect to variable $x$ on $X$, respectively, and combine this with \eqref{lem9_eq1}, we have
	\begin{equation}\label{lem1_eq2}
		\tilde{\psi}_{p_k}(x^{k+1}, y^{k+1})\le  \tilde{\psi}_{p_k}	(x^{k}, y^{k+1})+\langle \nabla_x 	\tilde{\psi}_{p_k}(x^{k}, y^{k+1}),x^{k+1}-x^k\rangle+\frac{L_{\psi_{x,k}}}{2}\|x^{k+1}-x^k\|^2,
	\end{equation}
	where $L_{\psi_{x,k}}=L_{F_x}/p_k+L_{f_x}+L_{g_1}+ \rho_{f_1} + \rho_{g_1} $. Combining \eqref{lem1_eq1} and \eqref{lem1_eq2} yields
	\begin{equation}\label{lem1_eq3}
		\begin{aligned}
			\tilde{\psi}_{p_k}\left(x^{k+1}, y^{k+1}\right) \leq \,& \tilde{\psi}_{p_k}\left(x^k, y^{k+1}\right) + \langle \nabla_x \tilde{\psi}_{p_k}(x^{k},y^{k+1}) - d_x^k,x^{k+1}-x^k\rangle  \\
			&-\left(\frac{1}{\alpha_k}-\frac{L_{\psi_{x,k}}}{2}\right)\left\|x^{k+1}-x^k\right\|^2.
		\end{aligned}
	\end{equation}
	Using the formula of the gradient of $v_\gamma$ given in \eqref{valuefinc} and the construction of $d_x^k$ in \eqref{dx}, and the $L_{f_x}$- and $L_{g_2}$-Lipschitz continuity of $\nabla_x f(x^k,y)$ and $\nabla_{x} g(x^k,y)$ with respect to variable $y$ on $Y$, we have
	\begin{equation}\label{lem1_eq10}
		\begin{aligned}
			\left\| \nabla_x  \tilde{\psi}_{p_k}(x^{k},y^{k+1}) - d_x^k \right\| 
			&= \left\| \nabla_{x} \varphi(x^k, \theta^*_{\gamma}(x^k,y^{k+1})) - \nabla_{x} \varphi(x^k, \theta^{k+1/2}) \right\| \\
			& \le (L_{f_x} + L_{g_2})\left\| \theta^*_{\gamma}(x^k,y^{k+1}) - \theta^{k+1/2} \right\| \le (L_{f_x} + L_{g_2})C_\eta \zeta_{k},
		\end{aligned}
	\end{equation}
	where the last inequality follows from \cref{error_bound} and the inexact criterion \eqref{inexacty1}. Combining the above inequality with \eqref{lem1_eq3} gives us
	\begin{equation}\label{lem1_eq4}
		\begin{aligned}
			\tilde{\psi}_{p_k}\left(x^{k+1}, y^{k+1}\right) \leq \,& \tilde{\psi}_{p_k}\left(x^k, y^{k+1}\right) - \left(\frac{1}{\alpha_k}-\frac{L_{\psi_{x,k}}}{2}-\frac{c_\alpha}{4}\right)\left\|x^{k+1}-x^k\right\|^2 \\ & + \frac{1}{c_\alpha}(L_{f_x} + L_{g_2})^2C_\eta^2 \zeta_{k}^2.
		\end{aligned}
	\end{equation}
	Combining  \eqref{lem1_eq8}, \eqref{lem1_eq4}, and then  \eqref{lem1_eq12} established in  Lemma \ref{lem0}, we obtain
	\begin{equation}\label{lem1_eq9}
		\begin{aligned}
			\tilde{\psi}_{p_k}\left(x^{k+1}, \tilde{y}^{k+1}\right)\leq\,& 	\tilde{\psi}_{p_k}\left(x^k, \tilde{y}^k\right)+  \frac{1}{c_\alpha}(L_{f_x} + L_{g_2})^2C_\eta^2 \zeta_{k}^2 + \frac{C_\eta^2}{c_\beta\gamma^2} \zeta_{k}^2 + C_{\gamma} \zeta_{k+1}^2\\
			&\hspace{-50pt}-\left(\frac{1}{\alpha_k}-\frac{L_{\psi_{x,k}}}{2}- \frac{c_\alpha}{4} \right)\left\|x^{k+1}-x^k\right\|^2 -\left(\frac{1}{\beta_k}-\frac{L_{\psi_{y,k}}}{2}- \frac{c_\beta}{4} \right)\left\|y^{k+1}-\tilde{y}^k\right\|^2.
		\end{aligned}
	\end{equation}
	Because  $\alpha_k \in [0, 2/(L_{\psi_{x,k}} + c_\alpha)]$ and $\beta_k \in [0, 2/(L_{\psi_{y,k}} + c_\beta)]$, and $\tilde{\psi}_{p_{k+1}}(x^{k+1}, \tilde{y}^{k+1}) \leq 	\tilde{\psi}_{p_k}(x^{k+1}, \tilde{y}^{k+1})$ following from the non-decreasing property of the penalty parameter $p_k$ and the non-negative property of $F-\underline{F}$, we can obtain \eqref{lem1_eq} from \eqref{lem1_eq9}.
	
	Finally, we show that \( \{\zeta_k\} \) is square summable. Let $ \mathcal{G}_0 = \mathcal{G}(\theta^{0}, x^{0}, y^{0})$ and $r_k :=  ( \prod_{i=0}^{k} \tau_i )\mathcal{G}_0 $. Since $\zeta_{k} := \max\{ s_k, r_k\}$, we have $\sum_{k=0}^\infty \zeta_k^2 \le \sum_{k=0}^\infty s_k^2 + \sum_{k=0}^\infty r_k^2$. By assumption, \( \sum_{k=0}^\infty s_k^2 < \infty \). Thus, we only need to show that  \( \{r_k\} \) is square summable. Recall that \( \tau_k = \tau_0/(k+1)^{p_\tau} \) with \( p_\tau > 0 \), then
	\[
	r_k = \left( \prod_{i=0}^{k} \tau_i \right) \mathcal{G}_0  = \tau_0^{k+1} \left( \prod_{i=1}^{k+1} \frac{1}{i^{p_\tau}} \right) \cdot \mathcal{G}_0  = \frac{\tau_0^{k+1}}{((k+1)!)^{p_\tau}}\mathcal{G}_0 .
	\]
	For the ratio test on  $\sum_{k=0}^\infty r_k^2$, we consider 
	\[
	\lim_{k \rightarrow \infty}\frac{r_{k+1}^2}{r_k^2} =  \lim_{k \rightarrow \infty} \frac{\tau_0^2}{(k+2)^{2p_\tau}}  = 0.
	\]
	Therefore, by the ratio test, $\sum_{k=0}^\infty r_k^2 < \infty$, implying that \( \{\zeta_k\} \) is square summable.
\end{proof}

\subsection{Analysis of the feasibility correction }\label{sec4.2}
Throughout this subsection and the remainder of this section, we assume that Assumptions \ref{asup1}, \ref{asup2}, and \ref{asup3} hold, and $\gamma \in (0, 1/(\rho_{f_2} +\rho_{g_2}) )$. The sequences $\{(x^k, y^k, \tilde{y}^k)\}$ and $\{p_k\}$ 
are generated by AGILS (\cref{alg1}), where the step sizes satisfy $\alpha_k \in [\underline{\alpha}, 2/(L_{\psi_{x,k}} + c_\alpha)]$ and $\beta_k \in [\underline{\beta}, 2/(L_{\psi_{y,k}} + c_\beta)]$, where $L_{\psi_{x,k}} = L_{F_x}/p_k+L_{f_x}+L_{g_1}+\rho_{f_1} + \rho_{g_1}$, $L_{\psi_{y,k}} = L_{F_y}/p_k+L_{f_y}$ and $\underline{\alpha}, \underline{\beta}, c_\alpha, c_\beta$ are positive constants. The inexactness parameters satisfy $\sum_{k=0}^{\infty}s_k^2 < \infty$ and 
$\tau_k := \tau_0/(k+1)^{p_\tau}$ with $\tau_0 > 0$ and $p_\tau >0$.

In this subsection, we show that the feasibility correction procedure (Algorithm \ref{alg2}) enforces the feasibility of the generated iterates.
We begin by showing that the differences between consecutive iterates are square summable.

\begin{lemma}\label{lem3}
	The sequence $\left\{(x^k,y^k,\tilde{y}^k)\right\}$ satisfies
	\begin{equation*}
		\sum\limits_{k=0}^{{\infty}}\left\|(x^{k+1},y^{k+1})-(x^k,\tilde{y}^k)\right\|^2<\infty, \quad \text{and} \quad \lim\limits_{k\to\infty}\left\|(x^{k+1},y^{k+1})-(x^k,\tilde{y}^k)\right\|=0.
	\end{equation*}
\end{lemma}
\begin{proof}
	Summing  \cref{lem1_eq} from Lemma \ref{lem1} over $k=0,1,\dots,K-1$ yields,
	\begin{equation}\label{lem3_eq1}
		\begin{aligned}
			\tilde{\psi}_{p_{K}}\left(x^{K}, \tilde{y}^{K}\right)- \tilde{\psi}_{p_{0}}\left(x^{0}, \tilde{y}^{0}\right) \leq & -\sum\limits_{k=0}^{K-1} \left(  \frac{c_\alpha}{4} \left\|x^{k+1}-x^k\right\|^2+ \frac{c_\beta}{4} \left\|y^{k+1}-\tilde{y}^k\right\|^2\right)\\& +\sum\limits_{k=0}^{K-1}\left( \frac{C_\eta^2 }{c_\alpha}(L_{f_x} + L_{g_2})^2+ \frac{C_\eta^2}{c_\beta\gamma^2} \right) \zeta_{k}^2+ \sum\limits_{k=0}^{K-1}C_{\gamma} \zeta_{k+1}^2.\\
		\end{aligned}
	\end{equation}
	Since \( \{\zeta_k\} \) is square summable from Lemma \ref{lem1}, it follows that
	\[
	\sum\limits_{k=0}^{\infty} \left( \frac{1}{c_\alpha}(L_{f_x} + L_{g_2})^2 + \frac{1}{c_\beta\gamma^2} \right) C_\eta^2\zeta_{k}^2 + \sum\limits_{k=0}^{\infty}  C_{\gamma} \zeta_{k+1}^2 < \infty.
	\]
	Additionally, because $\tilde{\psi}_{p_{K}}\left(x^{K}, \tilde{y}^{K}\right) \ge 0$, by taking $K \rightarrow \infty$ in \eqref{lem3_eq1}, we have
	\begin{equation}\label{lem3_eq2}
		\begin{aligned}
			&\quad\,\sum\limits_{k=0}^{\infty} \left(  \frac{c_\alpha}{4} \left\|x^{k+1}-x^k\right\|^2+ \frac{c_\beta}{4} \left\|y^{k+1}-\tilde{y}^k\right\|^2\right) \\
			&\le \tilde{\psi}_{p_{0}}\left(x^{0}, \tilde{y}^{0}\right) + \sum\limits_{k=0}^{\infty} \left( \frac{1}{c_\alpha}(L_{f_x} + L_{g_2})^2 + \frac{1}{c_\beta\gamma^2} \right) C_\eta^2\zeta_{k}^2 + \sum\limits_{k=0}^{\infty}  C_{\gamma} \zeta_{k+1}^2  < \infty,
		\end{aligned}
	\end{equation}
	which proves the result.
\end{proof}

The following lemma shows that any accumulation point of the corrected iterates $\tilde{y}^k$ satisfying condition \eqref{inexacty_y} lies in the solution set of the lower-level problem. 

\begin{lemma}\label{lem4}
	Let $\{ (x^{k_j},\tilde{y}^{k_j}) \} $ be any convergent subsequence of iterates $\{(x^k,\tilde{y}^k)\}$ such that $(x^{k_j},\tilde{y}^{k_j})$ satisfies condition \eqref{inexacty_y} with $k = k_j - 1$. 
	If $(x^{k_j},\tilde{y}^{k_j}) \rightarrow (\bar{x},\bar{y}) $ as $j \rightarrow \infty$ for some $(\bar{x},\bar{y}) $, then $\bar{y} \in S(\bar{x})$. 
\end{lemma}
\begin{proof}
	Define the residual term $\tilde{r}^k := \tilde{y}^{k} - \operatorname{Prox}_{ \tilde{g}(x^{k}, \cdot)} \left( \tilde{y}^{k} -  \nabla_y f(x^{k}, \tilde{y}^{k}) \right) $. By the first-order optimality condition for the proximal operator, $\tilde{r}^k$ satisfies
	\begin{equation}\label{lem4_eq1}
		\tilde{r}^k \in \nabla_y f(x^{k}, \tilde{y}^{k}) + \partial_{y}\tilde{g}(x^{k}, \tilde{y}^{k} - \tilde{r}^k).
	\end{equation}
	The sequence  $(x^{k_j},\tilde{y}^{k_j})$ satisfies condition \eqref{inexacty_y} with $k = k_j - 1$, which implies 
	\[
	\left\| \tilde{r}^{k_j} \right\| \le \frac{c_{\tilde{y}} }{p_{k_j}} \left\|(x^{k_j},y^{k_j})-(x^{k_j-1},\tilde{y}^{k_j-1})\right\|.
	\]
	From Lemma \ref{lem3}, we have $\lim_{k_j \to\infty} \|(x^{k_j},y^{k_j})-(x^{k_j-1},\tilde{y}^{k_j-1}) \|=0$. As $p_{k_j} \ge p_0$, the above inequality ensures that $\lim_{k_j \to\infty} \|\tilde{r}^{k_j} \|=0$.
	
	Taking the limit as  $k_j \rightarrow \infty$ in \eqref{lem4_eq1} with $k = k_j$, by the continuity of $ \nabla_y f$ and the outer semicontinuity of $\partial_{y}\tilde{g}$, it follows that
	\[
	0  \in \nabla_y f(\bar{x},\bar{y}) + \partial_{y}\tilde{g}(\bar{x},\bar{y}).
	\] 
	Since the lower-level problem is convex in $y$, this condition
	implies $\bar{y} \in S(\bar{x})$.
\end{proof}

We next establish the feasibility of the iterates generated by AGILS with respect to the constraint $\varphi({x}, {y}) -  v_\gamma({x},{y}) \le \epsilon$ when the penalty parameter $p_k$ is bounded.

\begin{proposition}\label{prop0}
	Let $\epsilon > 0$. If the sequence $\{p_k\}$ is bounded, then any accumulation point $(\bar{x}, \bar{y})$ of the iterate sequence $\{(x^k, {y}^k)\}$ satisfies
	\[
	\varphi(\bar{x}, \bar{y}) -  v_\gamma(\bar{x},\bar{y}) \le \epsilon.
	\]
\end{proposition}
\begin{proof}
	
	Let $(\bar{x}, \bar{y})$ be an accumulation point of $\{(x^k, {y}^k)\}$, and let \(\{k_j\}_{j \in \mathbb{N}} \subseteq \mathbb{N} \) be an index subsequence such that $(x^{k_j+1},y^{k_j+1})\to (\bar{x},\bar{y})$ as $j \rightarrow \infty$. By Lemma \ref{lem3}, we have  $\lim_{k_j \to\infty} \|(x^{k_j+1},y^{k_j+1})-(x^{k_j},\tilde{y}^{k_j}) \|=0$, which implies that $(x^{k_j},\tilde{y}^{k_j})$ also converges to $ (\bar{x},\bar{y})$ .
	
	Since the sequence $\{p_k\}$ is bounded and non-decreasing by construction in AGILS (\cref{alg1}), there exists $k_0 >0$ such that $p_k = p_{k_0}$ for all $k \ge k_0$. By passing to a subsequence if necessary, we can assume that $k_j - 1\ge k_0$ for all $j \in \mathbb{N}$.	
	
	According to the update rule of $p_k$ in AGILS (\cref{alg1}), the condition $p_{k_j - 1} = p_{k_j}$ implies that one of the following two cases must occur for each \(  k = k_j -1 \):
	
	Case I: $\left\|(x^{k_j},y^{k_j})-(x^{k_j-1}, \tilde{y}^{k_j-1})\right\| \ge c_p \min\left\{1/p_{k_0},{t^{k_j}}\right\}$ and $\tilde{y}^{k_j} = y^{k_j}$.
	
	Case II: $\left\|(x^{k_j},y^{k_j})-(x^{k_j-1}, \tilde{y}^{k_j-1})\right\| < c_p \min\left\{1/p_{k_0},{t^{k_j}}\right\}$ and the feasibility correction procedure is applied to generate $\tilde{y}^{k_j}$ satisfying condition \eqref{inexacty_y}.
	
	At least one of these two cases must hold for an infinite number of indices in the subsequence 	$\{ k_j  - 1\}$. We analyze them separately.
	
	Case 1:  Suppose Case I occurs for an infinite subsequence of indices. By passing to a further subsequence, we assume it holds for all $j \in \mathbb{N}$.
	
	The condition $\tilde{y}^{k_j} = y^{k_j}$ combined with $(x^{k_j},\tilde{y}^{k_j})\to (\bar{x},\bar{y})$ implies that $(x^{k_j}, {y}^{k_j})$\\
	$ \to (\bar{x},\bar{y})$. 
	From Lemma \ref{lem3}, we know $\lim_{k_j \to\infty} \|(x^{k_j},y^{k_j})-(x^{k_j-1},\tilde{y}^{k_j-1}) \|=0$. For the condition of Case I to hold, we must therefore have  $t^{k_j} \rightarrow 0$. This implies
	\begin{equation*}
		\limsup_{j \rightarrow \infty}	\left\{ \varphi(x^{k_j}, y^{k_j})-\varphi(x^{k_j},\theta^{k_j})-\frac{1}{2\gamma}\left\| \theta^{k_j} - y^{k_j} \right\|^2-\epsilon \right\} \le 0.
	\end{equation*}
	Since $\zeta_{k} \rightarrow 0$ and $\mathcal{G}( \theta^{k_j}, x^{k_j}, y^{k_j}) \le \zeta_{k_j}$, Lemma \ref{error_bound} ensures that $\| \theta^{k_j} - \theta^*_{\gamma}( x^{k_j}, y^{k_j} )\|\rightarrow 0$.
	By the continuity of $\theta^*_{\gamma}(x,y)$ (Proposition \ref{partial_grad_v}), it follows that $\theta^{k_j} \rightarrow \theta^*_{\gamma}(\bar{x}, \bar{y}) $.
	The continuity of $\varphi$ then implies that $\varphi(x^{k_j}, \theta^{k_j})+\frac{1}{2\gamma}\|\theta^{k_j} - y^{k_j}\|^2 \rightarrow v_\gamma(\bar{x},\bar{y})$. Then taking the limit in the inequality above as  $j \rightarrow \infty$ yields $\varphi(\bar{x},\bar{y})-v_{\gamma}(\bar{x},\bar{y})-\epsilon \le 0$. 
	
	Case 2: 
	Now, suppose Case II occurs for an infinite subsequence of indices. Again, we assume it holds for all $j \in \mathbb{N}$.
	
	In this scenario, the feasibility correction procedure is applied at each iteration $k = k_j -1$ to generate $\tilde{y}^{k_j}$ satisfying condition \eqref{inexacty_y}. Since $(x^{k_j},\tilde{y}^{k_j}) \rightarrow (\bar{x},\bar{y}) $ and satisfies condition \eqref{inexacty_y}, it follows from Lemma \ref{lem4} that $\bar{y} \in S(\bar{x})$. For any such point, $\theta^*_{\gamma}(\bar{x}, \bar{y}) = \bar{y} $ (see, e.g., \cite[Theorem 10.7]{beck2017first}), and therefore \( \varphi(\bar{x}, \bar{y}) - v_\gamma(\bar{x}, \bar{y}) = 0 \).  Since $ \epsilon > 0$, the constraint $		\varphi(\bar{x}, \bar{y}) -  v_\gamma(\bar{x},\bar{y}) \le \epsilon$ is satisfied.
	
	Since any accumulation point must be the limit of a subsequence that falls into one of these two cases, we conclude that $\varphi(\bar{x}, \bar{y}) -  v_\gamma(\bar{x},\bar{y}) \le \epsilon$.
\end{proof}

The following theorem establishes that the penalty parameter sequence $\left\{p_k\right\}$ is bounded, provided the sequence of iterates $\left\{ (x^k,y^k) \right\}$ is bounded.

\begin{theorem}\label{th2}
	Let $\epsilon > 0$. If the sequence of iterates $\left\{ (x^k,y^k) \right\}$ is bounded, then the penalty parameter sequence $\left\{p_k\right\}$ must be bounded. 
\end{theorem}

\begin{proof}
	
	The proof is by contradiction. Assume that \( p_k \to \infty \) as \( k \to \infty \). According to the penalty update rule in AGILS (\cref{alg1}), this implies the penalty parameter is increased at infinitely many iterations. Let \(\{k_j\}_{j \in \mathbb{N}} \subseteq \mathbb{N} \) be the subsequence of iteration indices where such an update occurs.
	
	Since $\{(x^k, y^k)\}$ is bounded and Lemma \ref{lem3} establishes that $\lim_{k \to\infty} \|(x^{k+1},y^{k+1})$\\
	$-(x^{k},\tilde{y}^{k}) \|=0$, the sequence $\{\tilde{y}^k\}$ is also bounded. 
	By passing to a subsequence if necessary, we can assume that for the indices $\{k_j\}$, we have $(x^{k_j+1},y^{k_j+1})\to (\bar{x},\bar{y})$ for some $(\bar{x},\bar{y}) \in X \times Y$ and $\tilde{y}^{k_j+1} \to \tilde{y}$ for some $\tilde{y} \in Y$.
	From Lemma \ref{lem3}, it also follows that $(x^{k_j},\tilde{y}^{k_j})\to (\bar{x},\bar{y})$ .
	
	At each iteration \( k = k_j \), the penalty parameter is increased, which requires that condition \eqref{update_p_1} is met
	\begin{equation}\label{th2_eq1}
		t^{k_j+1} > c_p^{-1}\left\|\left(x^{k_j+1},
		{y}^{k_j+1}\right)-\left(x^{k_j}, \tilde{y}^{k_j}\right)\right\| > 0.
	\end{equation}		
	We will show that this leads to the conclusion that $\varphi(\bar{x},\bar{y})-v_{\gamma}(\bar{x},\bar{y}) \ge \epsilon$.
	Let us assume the contrary, i.e., $\varphi(\bar{x},\bar{y})-v_{\gamma}(\bar{x},\bar{y})-\epsilon<0$.
	From Lemma \ref{error_bound}, the conditions $\mathcal{G}( \theta^{k_j+1}, x^{k_j+1}, y^{k_j+1}) \le \zeta_{k_j+1}$ and $\zeta_{k} \rightarrow 0$ ensure that $\| \theta^{k_j+1} - \theta^*_{\gamma}( x^{k_j+1}, y^{k_j+1} )\|\rightarrow 0$. 
	By the continuity of $\theta^*_{\gamma}(x,y)$ (Proposition \ref{partial_grad_v}) and $(x^{k_j+1},y^{k_j+1})\to (\bar{x},\bar{y})$, we have  $\theta^{k_j+1} \rightarrow \theta^*_{\gamma}(\bar{x}, \bar{y}) $. 
	The continuity of $\varphi$ then implies that $\varphi(x^{k_j+1}, \theta^{k_j+1})+\|\theta^{k_j+1} - y^{k_j+1}\|^2/(2\gamma) \rightarrow v_\gamma(\bar{x},\bar{y})$.
	Our assumption $\varphi(\bar{x},\bar{y})-v_{\gamma}(\bar{x},\bar{y})-\epsilon<0$, combined with the continuity of $\varphi$, means that for $j$ sufficiently large,
	\begin{equation*}
		\varphi(x^{k_j+1}, y^{k_j+1})-\varphi(x^{k_j+1},\theta^{k_j+1})-\frac{1}{2\gamma}\left\| \theta^{k_j+1} - y^{k_j+1} \right\|^2-\epsilon <0. 
	\end{equation*}
	This inequality implies that $t^{k_j+1} = 0$ which contradicts \eqref{th2_eq1}. Thus, our assumption was false, and we must have  $\varphi(\bar{x},\bar{y})-v_{\gamma}(\bar{x},\bar{y}) \ge \epsilon $. 
	
	An increase in the penalty parameter at iteration  \( k = k_j \) corresponds to one of the following two cases:
	
	Case I:  $\left\| y^{k+1} - \theta^{k+1} \right\| \le c_y \gamma/p_k$.
	
	Case II: $\left\| y^{k+1} - \theta^{k+1} \right\| > c_y \gamma/p_k$,  the feasibility correction procedure is applied, but the corrected iterate $\tilde{y}^{k+1} $ violates the descent condition \eqref{decrease_new_y}.
	
	At least one of these cases must occur for an infinite subsequence of
	$\{ k_j \}$.  We analyze both possibilities to derive a contradiction.
	
	Case 1: Assume an infinite subsequence, still denoted by $\{ k_j \}$, falls into Case I. For this subsequence, we have
	\[
	\left\| y^{k_j+1} - \theta^{k_j+1}\right\| \le \frac{\gamma}{p_{k_j+1}} \to 0.
	\]
	Since we assume $p_{k_j} \rightarrow \infty$, the right-hand side converges to $0$. As $j\to \infty$, we have $y^{k_j+1} \rightarrow \bar{y}$ and $\theta^{k_j+1} \rightarrow \theta^*_{\gamma}(\bar{x}, \bar{y}) $. Taking the limit yields $\| \bar{y} - \theta^*_{\gamma}(\bar{x}, \bar{y}) \| \le 0$, which implies $\bar{y} = \theta^*_\gamma(\bar{x},\bar{y})$. By definition, this  yields  $\varphi(\bar{x}, \bar{y}) = v_\gamma(\bar{x}, \bar{y})$, which contradicts our established result that \( \varphi(\bar{x}, \bar{y}) - v_\gamma(\bar{x}, \bar{y}) \ge \epsilon > 0 \). 
	
	Case 2: Assume an infinite subsequence, again denoted by $\{ k_j \}$, belongs to Case II.
	Here, the feasibility correction procedure is applied at each iteration  $k = k_j$, generating a corrected iterate $\tilde{y}^{k_j+1}$ satisfying condition \eqref{inexacty_y},  together with a point $\tilde{\theta}^{k_j+1}$ defined by $\tilde{\theta}^{k_j+1}=
		\operatorname{Prox}_{\eta \tilde g(x^{k_j+1},\cdot)}
		(
		\hat{\theta}^{k_j+1}
		-\eta(
		\nabla_y f(x^{k_j+1},\hat{\theta}^{k_k+1})
		+(\hat{\theta}^{k_j+1}-\tilde y^{k_j+1})/\gamma
		)
		)$ with  $\hat\theta^{k_j+1}$ satisfying the inexact criterion \eqref{inexacty3}. However, by the definition of Case II, these iterates violate the descent condition \eqref{decrease_new_y}, leading to
	\begin{equation*}
		\begin{aligned}
			&\frac{1}{p_{k_j}}F(x^{k_j+1},\tilde{y}^{k_j+1}) + \varphi(x^{k_j+1},\tilde{y}^{k_j+1}) - \varphi(x^{k_j+1},\tilde{\theta}^{k_j+1}) - \frac{1}{2\gamma} \| \tilde{\theta}^{k_j+1} - \tilde{y}^{k_j+1} \|^2 \\
			>\ 
			&\frac{1}{p_{k_j}}F(x^{k_j+1},y^{k_j+1}) + \varphi(x^{k_j+1},y^{k_j+1}) - \varphi(x^{k_j+1},\theta^{k_j+1}) - \frac{1}{2\gamma} \| \theta^{k_j+1} - y^{k_j+1} \|^2.
		\end{aligned}
	\end{equation*}
	As $j \rightarrow \infty$, we have $(x^{k_j+1},\tilde{y}^{k_j+1}) \rightarrow (\bar{x},\tilde{y}) $. Since $(x^{k_j+1},\tilde{y}^{k_j+1})$ satisfies condition \eqref{inexacty_y}, it follows from Lemma \ref{lem4} that $\tilde{y} \in S(\bar{x})$. Moreover, given that $\zeta_{k} \rightarrow 0$ and $\mathcal{G}(\hat{\theta}^{k_j+1}, x^{k_j+1},\tilde{y}^{k_j+1}) \le \zeta_{k_j+1}$, Lemma \ref{error_bound} ensures that $\| \hat{\theta}^{k_j+1} - \theta^*_{\gamma}( x^{k_j+1}, \tilde{y}^{k_j+1} )\|\rightarrow 0$. By the continuity of $\theta^*_{\gamma}(x,y)$ (Proposition \ref{partial_grad_v}), this implies $\hat{\theta}^{k_j+1} \rightarrow \theta^*_{\gamma}(\bar{x}, \tilde{y}) $. Finally, from the definition of \(\tilde{\theta}^{k_j+1}\), we have $\|\tilde{\theta}^{k_j+1}-\hat{\theta}^{k_j+1}\|=\mathcal{G}(\hat{\theta}^{k_j+1},x^{k_j+1},\tilde{y}^{k_j+1})\to0$, and therefore $\tilde\theta^{k_j+1}\to \theta^*_\gamma(\bar x,\tilde y)$.

	Because $\tilde{y} \in S(\bar{x})$, we know that $\theta^*_{\gamma}(\bar{x}, \tilde{y}) = \tilde{y} $ (see, e.g., \cite[Theorem 10.7]{beck2017first}). 
	Taking the limit of the left-hand side of the inequality above as $k_j \rightarrow \infty$, and using the continuity of $F$ and $\varphi$ along with $p_{k_j} \rightarrow \infty$, we find that the expression converges to $ \varphi(\bar{x}, \tilde{y}) - \varphi(\bar{x}, \theta^*_{\gamma}(\bar{x}, \tilde{y})) -  \frac{1}{2\gamma}  \|\theta^*_{\gamma}(\bar{x}, \tilde{y}) - \tilde{y}\|^2 = 0$. On the other hand, we have $(x^{k_j+1},y^{k_j+1})\to (\bar{x},\bar{y})$ and $\theta^{k_j+1} \rightarrow \theta^*_{\gamma}(\bar{x}, \bar{y}) $. Therefore, the right-hand side of the inequality converges to $ \varphi(\bar{x}, \bar{y}) - v_\gamma(\bar{x}, \bar{y}) \ge \epsilon > 0 $. As $k_j \rightarrow \infty$, the inequality thus implies $0 \ge \epsilon >0$, which is a contradiction.
	
	Since both possible scenarios for an infinite number of penalty updates lead to a contradiction, our initial assumption that  \( p_k \to \infty \)  must be false. Therefore, the sequence of penalty parameters $\{p_k\}$ is bounded.
\end{proof}

Next, we establish that once the penalty parameter $p_k$ becomes sufficiently large, the feasibility correction procedure will be applied only finitely many times. 

\begin{proposition}\label{th4}
	Let $\epsilon > 0$. Suppose the sequence of iterates $\left\{ (x^k,y^k) \right\}$ is bounded, and define $ M_F := \sup_k |F(x^{k},{y}^{k}) | $. If there exists $k_0 \in \mathbb{N}$	such that the penalty parameter $p_{k_0}$ is sufficiently large to satisfy $	p_{k_0} \ge 4M_F/\epsilon$,	then the feasibility correction
	procedure is applied only a finite number of times.
\end{proposition}

\begin{proof}
By Theorem~\ref{th2}, the penalty parameter sequence $\{p_k\}$ is bounded. 	Hence, $p_k = \bar{p}$ for all sufficiently large $k$, and we may assume without loss of generality that $p_k = \bar{p}$ for all $k$. Under the assumption, we have
\begin{equation}\label{th4_eq5}
\frac{M_F}{\bar{p}} \le \frac{\epsilon}{4}.
\end{equation}

From \eqref{lem1_eq12} in Lemma~\ref{lem0}, together with \eqref{lem1_eq8} and \eqref{lem1_eq4} in the proof of Lemma~\ref{lem1}, we obtain
\begin{equation}\label{th4_eq1}
\begin{aligned}
\tilde{\psi}_{\bar{p}}\left(x^{k+1}, \tilde{y}^{k+1}\right)\leq\, &  \tilde{\psi}_{\bar{p}}(x^{k+1}, y^{k+1}) + C_{\gamma} \zeta_{k+1}^2 \\
\le \, & 	\tilde{\psi}_{\bar{p}}\left(x^k, \tilde{y}^k\right)+ C_1\zeta_k^2 + C_2\zeta_{k+1}^2
- \frac{c_\alpha}{4}\|x^{k+1}-x^k\|^2
- \frac{c_\beta}{4}\|y^{k+1}-\tilde y^k\|^2,
\end{aligned}
\end{equation}
where $C_1=C_\eta^2(L_{f_x}+L_{g_2})^2/c_\alpha$ and $C_2=C_\eta^2/(c_\beta\gamma^2)+C_\gamma$.

Define
\[
A_k:=\tilde{\psi}_{\bar p}(x^{k},\tilde y^{k})
+\sum_{j=k}^{\infty}\!\big(C_1\zeta_j^2+C_2\zeta_{j+1}^2\big).
\]
Adding the tail sum to both sides of \eqref{th4_eq1} yields
\[
A_{k+1}\le A_k
-\frac{c_\alpha}{4}\|x^{k+1}-x^k\|^2
-\frac{c_\beta}{4}\|y^{k+1}-\tilde y^k\|^2.
\]
Thus, $\{A_k\}$ is a nonincreasing sequence. Since $\tilde{\psi}_{\bar{p}}\left(x^{k}, \tilde{y}^{k}\right) \ge 0$ for all $k$, we have $A_k \ge 0$, implying that ${A_k}$ is bounded below and hence convergent, i.e., $A_k \to \bar{A} \in \mathbb{R}$.

From Lemma~\ref{lem1}, the sequence $\{\zeta_k\} $ is square summable, so $\zeta_k \to 0$ and $\sum_{j=k}^{\infty} (C_1\zeta_j^2 + C_2\zeta_{j+1}^2) \to 0$ as $k \to \infty$. Consequently,
\begin{equation*}
\lim_{k \rightarrow \infty}\tilde{\psi}_{\bar p}(x^{k},\tilde y^{k}) = \lim_{k \rightarrow \infty} \left(A_k - \sum_{j=k}^{\infty}\big(C_1\zeta_j^2+C_2\zeta_{j+1}^2\big) \right) = \bar{A}.
\end{equation*}
Combining this with \eqref{th4_eq1} and the fact that $\zeta_k \to 0$ gives
\begin{equation*}
\lim_{k \rightarrow \infty}\tilde{\psi}_{\bar p}(x^{k}, y^{k}) = 	\lim_{k \rightarrow \infty}\tilde{\psi}_{\bar p}(x^{k},\tilde y^{k})   = \bar{A}.
\end{equation*}

We now prove the claim by contradiction. Suppose that the feasibility correction procedure is triggered infinitely many times. Then there exists an infinite index set ${k_j}$ such that Case~3 is executed at each iteration $k_j$. Then, we have,
for all $j > 0$, 
\[
t^{k_j+1} > 0.
\]
By the definition of $t^{k} $, we have
\[
\begin{aligned}
t^{k_j+1}  &= \varphi(x^{k_j+1},y^{k_j+1})-\varphi(x^{k_j+1},\theta^{k_j+1})-\frac{1}{2\gamma}\| \theta^{k_j+1} - y^{k_j+1}\|^2-\epsilon \\
&\le \varphi(x^{k_j+1},y^{k_j+1}) - v_\gamma(x^{k_j+1},y^{k_j+1})-\epsilon,
\end{aligned}
\]
where the inequality follows from 	$v_\gamma(x^{k_j+1},y^{k_j+1})=\min_{\theta\in Y}\{\varphi(x^{k_j+1},\theta)+\tfrac{1}{2\gamma}\|\theta-y^{k_j+1}\|^2\}$. 
Since $t^{k_j+1} > 0$, we obtain
\begin{equation}\label{th4_eq2}
\varphi(x^{k_j+1},y^{k_j+1}) - v_\gamma(x^{k_j+1},y^{k_j+1}) \ge \epsilon.
\end{equation}

As the feasibility correction procedure is performed at iteration $k_j$, a corrected iterate $\tilde{y}^{k_j+1}$ satisfying condition~\eqref{inexacty_y} is generated.
Because $\{(x^k, y^k)\}$ is bounded and Lemma \ref{lem3} shows that $\lim_{k \to\infty} \|(x^{k+1},y^{k+1})-(x^{k},\tilde{y}^{k}) \|=0$, the sequence $\{\tilde{y}^k\}$ is also bounded. 
By passing to a subsequence if necessary, we may assume that $(x^{k_j+1},\tilde{y}^{k_j+1})\to (\bar{x},\bar{y})$ for some $(\bar{x},\bar{y}) \in X \times Y$. Then Lemma \ref{lem4} implies that $\bar{y}\in S(\bar x)$, and thus $	\varphi(\bar x,\bar{y})-v_\gamma(\bar x,\bar{y})=0$. By continuity of $\varphi$ and $v_\gamma$, there exists $j_0 \in \mathbb{N}$ such that for all $j\ge j_0$,
\begin{equation}\label{th4_eq3}
\varphi(x^{k_j+1},\tilde{y}^{k_j+1})-v_\gamma(x^{k_j+1},\tilde{y}^{k_j+1})\le \frac{\epsilon}{10}.
\end{equation}

Furthermore, since $\{(x^k, y^k)\}$ is bounded and $\lim_{k \to\infty} \|(x^{k+1},y^{k+1})-(x^{k},\tilde{y}^{k}) \|=0$ from  Lemma \ref{lem3}, by the continuity of $F$ on $X \times Y$, we have 
\begin{equation}\label{th4_eq4}
\sup_k |F(x^{k},\tilde{y}^{k}) | = \sup_k |F(x^{k},{y}^{k}) | = M_F.
\end{equation}

We now consider the difference $\tilde{\psi}_{\bar p}(x^{k_j+1},y^{k_j+1})-\tilde{\psi}_{\bar p}(x^{k_j+1},\tilde{y}^{k_j+1})$ when  $j\ge j_0$,
\[
\begin{aligned}
&\tilde{\psi}_{\bar p}(x^{k_j+1},y^{k_j+1})-\tilde{\psi}_{\bar p}(x^{k_j+1},\tilde{y}^{k_j+1}) \\
= \, & 	\frac{1}{\bar p}\big( F(x^{k_j+1},y^{k_j+1}) - F(x^{k_j+1},\tilde{y}^{k_j+1}) \big) \\
&+(\varphi-v_\gamma)(x^{k_j+1},y^{k_j+1}) 
-(\varphi-v_\gamma)(x^{k_j+1},\tilde{y}^{k_j+1})\\
\ge \, & - \frac{2M_F}{\bar{p}} + \epsilon - \frac{\epsilon}{10} \\
\ge\, & - \frac{\epsilon}{2} +  \epsilon - \frac{\epsilon}{10} \\
= \, & \frac{2}{5}\epsilon > 0,
\end{aligned}
\]
where the first inequality follows from \eqref{th4_eq2}, \eqref{th4_eq3} and \eqref{th4_eq4} and the second follows from \eqref{th4_eq5}. This contradicts the fact that $	\lim_{k \rightarrow \infty}\tilde{\psi}_{\bar p}(x^{k}, y^{k}) = 	\lim_{k \rightarrow \infty}\tilde{\psi}_{\bar p}(x^{k},\tilde y^{k})   = \bar{A}$. Hence, the feasibility correction procedure can occur only finitely many times.
\end{proof}

\subsection{Convergence to stationary points}
\label{sec4.3}

In this part, we establish the  sequential convergence  properties of the AGILS (\cref{alg1}).  We demonstrates that any accumulation point of sequence $\{(x^k, y^k)\}$ is a KKT point of the problem $({\rm VP})_\gamma^{\bar{\epsilon}}$ for some  $\bar{\epsilon}\le\epsilon$. The problem setting and algorithmic assumptions are identical to those in the previous subsection.

\begin{theorem}\label{th1}
	Let $\epsilon > 0$. If the sequence $\{p_k\}$ is bounded, then any accumulation point $(\bar{x}, \bar{y})$ of $\{(x^k,y^k)\}$ is a KKT point of $({\rm VP})_\gamma^{\bar{\epsilon}}$ for some  $\bar{\epsilon}\le\epsilon$. 
\end{theorem}
\begin{proof}
	The proof builds upon techniques used in the proof of Theorem 1 in \cite{ye2023}. 
	Let $(\bar{x}, \bar{y})$ be an accumulation point of the sequence $\{(x^k,y^k)\}$, and let $\{(x^{k_j+1},y^{k_j+1})\}$ be a convergent subsequence such that $(x^{k_j+1},y^{k_j+1}) \rightarrow (\bar{x}, \bar{y})$ as $k_j \rightarrow \infty$. 
	By Lemma \ref{lem3}, we have  $\lim_{k_j \to\infty} \|(x^{k_j+1},y^{k_j+1})-(x^{k_j},\tilde{y}^{k_j}) \|=0$, which implies that $(x^{k_j},\tilde{y}^{k_j}) \to (\bar{x},\bar{y})$ .
	Since $\{p_k\}$ is bounded, there exists $k_0 > 0$ such that $p_k = \bar{p} := p_{k_0}$ for any $k \ge k_0$.

	From the update rules \eqref{update_x} and \eqref{update_y}, and the expressions for $d_x^k$ and $d_y^k$, we have    
	\begin{equation}\label{th1_eq2}
		\begin{aligned}
			\xi_x^k & \in \nabla_x F(x^{k+1},y^{k+1})+p_k \left(\nabla_x \varphi(x^{k+1},y^{k+1}) - \nabla_x v_\gamma(x^{k+1},y^{k+1})\right) +\mathcal{N}_{X}(x^{k+1}),\\
			\xi_y^k & \in \nabla_y F(x^{k},y^{k+1})+p_k \left(\partial_y \varphi(x^{k},y^{k+1}) - \nabla_y v_\gamma(x^{k},y^{k+1})\right) +\mathcal{N}_{Y}(y^{k+1}),
		\end{aligned}
	\end{equation}
	where $\xi_x^k \in \Re^n$ and $\xi_y^k \in \Re^m$ are given by
	\begin{equation}\label{th1_eq4}
		\begin{aligned}
			\xi_x^k &:=  p_k \left(\nabla_x \tilde{\psi}_{p_{k}}\left(x^{k+1}, y^{k+1} \right) - d_x^k \right)  - \frac{p_k}{\alpha_k}\left(x^{k+1}-x^k\right), \\
			\xi_y^k &:=  p_k \left( \nabla_y\left(\tilde{\psi}_{p_k}-g\right)\left(x^{k}, y^{k+1}\right) - d_y^k \right)-\frac{p_k}{\beta_k}\left(y^{k+1}-\tilde{y}^k\right).
		\end{aligned} 
	\end{equation} 
	We now show that $\xi_x^{k_j}, \xi_y^{k_j} \rightarrow 0$ as $k_j \rightarrow \infty$. First, we have 
	\begin{equation}
		\begin{aligned}
			& \| \nabla_x \tilde{\psi}_{p_{k}}\left(x^{k+1}, y^{k+1} \right) - d_x^k\| \\ \le \, & \|  \nabla_x \tilde{\psi}_{p_{k}}\left(x^{k+1}, y^{k+1} \right) - \nabla_x \tilde{\psi}_{p_{k}}\left(x^{k}, y^{k+1} \right) \| + \| \nabla_x \tilde{\psi}_{p_{k}} \left(x^{k}, y^{k+1} \right) - d_x^k \| \\
			\le \, & \|  \nabla_x \tilde{\psi}_{p_{k}}\left(x^{k+1}, y^{k+1} \right) - \nabla_x \tilde{\psi}_{p_{k}}\left(x^{k}, y^{k+1} \right) \| + (L_{f_x} + L_{g_2})C_\eta \zeta_{k}, 
		\end{aligned}
	\end{equation}
	where the last inequality follows from \eqref{lem1_eq10}. Since $p_k = \bar{p}$ for all $k \ge k_0$, and assuming that $\nabla_{x} F, \nabla_{x} f$ and $\nabla_{x} g$ are continuous on $X \times Y$, and $\nabla_{x} v_\gamma$ is continuous on $X \times Y$ (Proposition \ref{partial_grad_v}), we conclude that $ \nabla_x \tilde{\psi}_{p_{k}} =  \nabla_x \tilde{\psi}_{\bar{p}}$ is continuous on $X \times Y$. 
	As $\lim_{k_j\to\infty} (x^{k_j+1},y^{k_j+1}) = \lim_{k_j\to\infty} (x^{k_j},\tilde{y}^{k_j}) = (\bar{x}, \bar{y})$ and $\zeta_{k} \rightarrow 0$, it follows that $\nabla_x \tilde{\psi}_{p_{k_j}} (x^{k_j+1}, y^{k_j+1} ) - d_x^{k_j} \rightarrow 0$ as $k_j \rightarrow \infty$. Combining this with $ \|x^{k_j+1}-x^{k_j} \| \to 0$ and $\alpha_k \ge \underline{\alpha} > 0$, we obtain from \eqref{th1_eq4} that $\xi_x^{k_j} \rightarrow 0$ as $k_j \rightarrow \infty$.
	
	A similar argument can be applied to $\xi_y^k$. From \eqref{lem1_eq11}, we have 
	\begin{equation}\label{th1_eq5}
		\begin{aligned}
			& \|\nabla_y\left(\tilde{\psi}_{p_k}-g\right)\left(x^{k}, y^{k+1}\right) - d_y^k \| \\ \le \, & \| \nabla_y\left(\tilde{\psi}_{p_k}-g\right)\left(x^{k}, y^{k+1} \right) - \nabla_y\left(\tilde{\psi}_{p_k}-g\right)\left(x^{k}, \tilde{y}^k \right) \| + {C_\eta}\zeta_{k}/\gamma.
		\end{aligned}
	\end{equation}
	Since $p_k = \bar{p}$ for all $k \ge k_0$, and under the assumption that $\nabla_{y} F$ and $\nabla_{y} f$ are continuous on $X \times Y$, and $\nabla_{y} v_\gamma$ is continuous on $X \times Y$ (Proposition \ref{partial_grad_v}), we have that $ \nabla_y (\tilde{\psi}_{p_{k}} - g) = \nabla_y (\tilde{\psi}_{\bar{p}} - g)$ is continuous on $X \times Y$. As $\lim_{k_j\to\infty} (x^{k_j},y^{k_j+1}) = \lim_{k_j\to\infty} (x^{k_j},\tilde{y}^{k_j}) = (\bar{x}, \bar{y})$ and $\zeta_{k} \rightarrow 0$, it follows that $\nabla_y (\tilde{\psi}_{p_{k_j}} - g)(x^{k_j}, y^{k_j+1} ) - d_y^{k_j} \rightarrow 0$ as $k_j \rightarrow \infty$. Combining this with $ \|y^{k_j+1}-\tilde{y}^{k_j} \| \to 0$ and $\beta_k \ge \underline{\beta} > 0$, we conclude from \eqref{th1_eq5} that $\xi_y^{k_j} \rightarrow 0$ as $k_j \rightarrow \infty$.
	
	Now, taking the limit as $k = k_j \rightarrow \infty$ in \eqref{th1_eq2}, since $\xi_x^{k_j}, \xi_y^{k_j} \rightarrow 0$, $p_k = \bar{p}$ for all $k \ge k_0$, $\lim_{k_j\to\infty} (x^{k_j+1},y^{k_j+1}) = \lim_{k_j\to\infty} (x^{k_j},y^{k_j+1}) = \lim_{k_j\to\infty} (x^{k_j},\tilde{y}^{k_j}) = (\bar{x}, \bar{y})$, $\nabla F$, $\nabla v_\gamma$, and $\nabla_x \varphi$ are continuous on $X \times Y$, $\partial_y \varphi$ and $\mathcal{N}_{X\times Y}$ are outer semicontinuous on $X \times Y$, and $\partial \varphi = \{\nabla_x \varphi\} \times \partial_y \varphi$ on $X \times Y$, we obtain 
	\begin{equation}\label{th1_eq6}
		0\in\nabla F(\bar{x}, \bar{y})+\bar{p}\left(\partial \varphi(\bar{x},\bar{y})-\nabla v_\gamma(\bar{x}, \bar{y})\right)+\mathcal{N}_{X\times Y}(\bar{x}, \bar{y}).
	\end{equation}
	
	Finally, from Proposition \ref{prop0}, we know that the limit point $(\bar{x}, \bar{y})$ satisfies $\varphi(\bar{x}, \bar{y}) -  v_\gamma(\bar{x},\bar{y}) \le \epsilon $. Let $\bar{\epsilon}:= \varphi(\bar{x}, \bar{y}) -  v_\gamma(\bar{x},\bar{y}) \in [0, \epsilon]$. Combining this with \eqref{th1_eq6}, we have
	\begin{equation}
		\left\{
		\begin{aligned}
			&0\in\nabla F(\bar{x}, \bar{y}) + \bar{p}\left( \partial \varphi(\bar{x}, \bar{y}) - \nabla v_\gamma(\bar{x}, \bar{y})\right)+\mathcal{N}_{X\times Y}(\bar{x}, \bar{y}),\\
			&\varphi(\bar{x}, \bar{y}) - v_{\gamma}(\bar{x}, \bar{y})- \bar{\epsilon} \le 0, \qquad  \bar{p}\left( \varphi(\bar{x}, \bar{y}) - v_{\gamma}(\bar{x}, \bar{y})- \bar{\epsilon} \right)=0,
		\end{aligned}\right.
	\end{equation}
	which implies that $(\bar{x},\bar{y})$ is a KKT point of problem $({\rm VP})_\gamma^{\bar{\epsilon}}$ with $\bar{\epsilon}\le \epsilon$.
\end{proof}

As established in Lemma \ref{lem3}, the Ostrowski condition \(\lim_{k\to\infty}\|(x^{k+1},y^{k+1})-(x^k,\tilde{y}^k)\|=0\) is satisfied. By combining this result with Theorem \ref{th1} and \cite[Proposition 8.3.10]{facchinei2007finite}, we derive the following convergence result:

\begin{corollary}
	Let $\epsilon > 0$. Suppose the sequence ${p_k}$ is bounded, and the sequence $\{(x^k, y^k)\}$ has an isolated accumulation point $(\bar{x}, \bar{y})$, meaning there exists a neighborhood $Z$ of $(\bar{x}, \bar{y})$ such that $(\bar{x}, \bar{y})$ is the only accumulation point of $\{(x^k, y^k)\}$ in $Z$. Then the sequence $\{(x^k, y^k)\}$ converges to $(\bar{x}, \bar{y})$, which is a KKT point of $({\rm VP})_\gamma^{\bar{\epsilon}}$ for some $\bar{\epsilon} \leq \epsilon$.
\end{corollary}

\begin{remark}
	Our algorithm AGILS offers significant advantages over the MEHA in \cite{liumoreau}, particularly regarding step-size selection and convergence guarantees.
	First, AGILS allows a larger, explicit, and easily computable range of step sizes. In contrast, the step sizes for MEHA must be small and lack a clear, practical selection range. Second, AGILS ensures that any accumulation point of the iterates is a feasible KKT point of $({\rm VP})_\gamma^{\bar{\epsilon}}$ for some  $\bar{\epsilon}\le\epsilon$. Conversely, MEHA's convergence analysis only ensures that the stationary residual of penalized problem \eqref{penprob} converges to zero, without guaranteeing the feasibility of the iterates.
	
\end{remark}

\section{Sequential convergence under KL property}

This section establishes the sequential convergence of AGILS (Algorithm \ref{alg1}) under the Kurdyka-{\L}ojasiewicz (KL) property. Our analysis focuses on the scenario where the feasibility correction procedure is applied only a finite number of times. (see Proposition~\ref{th4} for a sufficient condition ensuring this).
We begin by recalling the definition of the KL property. Let $\zeta \in[0,+\infty)$ and $\Phi_{\zeta}$ represent the set of all concave and continuous functions $\phi\ :\ [0, \zeta)\to[0,+\infty)$ that meet the following conditions: (a) $\phi(0) = 0$, (b) $\phi$ is $C^1$ on $(0, \zeta)$ and continuous at 0, (c) $\phi'(s)>0$ for all $s\in(0,\zeta)$.
\begin{definition}[Kurdyka-{\L}ojasiewicz property]
	Consider $h:\Re^d\to (-\infty,+\infty]$ as proper and lower semicontinuous. The function $h$ is said to have the Kurdyka-{\L}ojasiewicz (KL)
	property at $\bar{x}\in{\rm dom} \, \partial h := \left\{x\in\Re^d\ |\ \partial  h(x)\neq\emptyset\right\}$, if there exist $\zeta\in(0,+\infty]$, a neighborhood $\mathcal{U}$ of $\bar{x}$ and a function $\phi \in \Phi_{\zeta}$, such that for all $x\in \mathcal{U}\bigcap\left\{x\in\Re^d\ |\ h(\bar{x})< h (x)< h (\bar{x})+\zeta\right\}$,
	the following inequality holds
	\begin{equation*}
		\phi'(h(x)-h(\bar{x})){\rm dist}(0,\partial h(x))\ge1.
	\end{equation*}
	If $h$ satisfies the KL property at each point of ${\rm dom}\,\partial h$ then $h$ is referred to as a KL function.
\end{definition}

In addition, when the KL property holds for all points in a compact set, the uniformized KL property is applicable, refer to Lemma 6 in \cite{bolte2014proximal} for further details.
\begin{lemma}[Uniformized KL property]\label{lem2}
	Given a compact set $D$ and a proper and lower semicontinuous function $h\ :\ \Re^d \to (-\infty,+\infty]$, suppose that $h$ is constant on $D$ and satisfies the KL property at each point of $D$. Then, there exist $\epsilon$, $\zeta$ and $\phi \in\Phi_{\zeta}$ such that for all $\bar{x}\in D$ and $x\in\left\{ x\in\Re^d\,|\,{\rm dist}(x,D)<\epsilon,\ h(\bar{x})<h(x)<h(\bar{x})+\zeta\right\}$,
	\begin{equation*}
		\phi'(h(x)-h(\bar{x})){\rm dist}(0,\partial h(x))\ge 1.
	\end{equation*}
\end{lemma}

To establish the sequential convergence of AGILS (Algorithm \ref{alg1}), we further assume that $\partial_y g(x,y)$ admits a Lipschitz continuity with respect to $x$ on $X$.
\begin{assumption}\label{asup5}
	$\partial_y g(x,y)$ is $L_{g_y}$-Lipschitz  continuous with respect to $x$ on $X$ for any fixed $y \in Y$. That is, for any $x, x' \in X$,
	\[
	\partial_y g(x',y) \subseteq \partial_y g(x,y) + L_{g_y}\|x - x'\|\mathbb{B}.
	\]
\end{assumption}

Assumption \ref{asup5} is not restrictive. For example, if $g(x,y) = \sum_{i = 1}^p Q_i(x)P_i(y)$, where $Q_i(x): X \rightarrow \mathbb{R}$ is Lipschitz continuous and strictly positive on $X$, and $P_i(y) : Y \rightarrow \mathbb{R}$, then $g(x,y)$ satisfies Assumption \ref{asup5}. Specific examples include $g(x,y) = x\| y\|_1$ on $\mathbb{R}_{++} \times \mathbb{R}^m$ and $g(x,y) = \sum_{j = 1}^J x_j\| y^{(j)}\|_2$ on $\mathbb{R}^J_{++} \times \mathbb{R}^m$.

Consider the sequences  $\{(x^k, y^k)\}$ and $\{p_k\}$  generated by AGILS (Algorithm \ref{alg1}) under the step size conditions specified in Section \ref{sec4.2} and \ref{sec4.3}, where $\gamma \in (0, 1/(\rho_{f_2} +\rho_{g_2}) )$, the step sizes satisfy $\alpha_k \in [\underline{\alpha}, 2/(L_{\psi_{x,k}} + c_\alpha)]$ and $\beta_k \in [\underline{\beta}, 2/(L_{\psi_{y,k}} + c_\beta)]$, where $L_{\psi_{x,k}} = L_{F_x}/p_k+L_{f_x}+L_{g_1}+\rho_{f_1} + \rho_{g_1}$, $L_{\psi_{y,k}} = L_{F_y}/p_k+L_{f_y}$ and $\underline{\alpha}, \underline{\beta}, c_\alpha, c_\beta$ are positive constants. 
To analyze sequential convergence, we define the following merit function $E_p(z, u)$, where $z := (x,y)$, 
\begin{equation}\label{def_merit}
	E_{p}(z,u) := G_p(z) - \langle u,x \rangle + H(u, y), \qquad \text{with}
\end{equation}
\[
G_p(z) := \frac{1}{p} \left( F(z) - \underline{F} \right) + \varphi(z) + \frac{\rho_{v_1}}{2}\|x\|^2, \quad  H(u, y) := \sup_x \{ \langle u,x \rangle - \tilde{v}_\gamma(x,y)\},
\] 
where $\rho_{v_1} = \rho_{f_1} + \rho_{g_1} + c_\alpha/4$ and $\tilde{v}_\gamma(x,y) := {v}_\gamma(x,y) + (\rho_{v_1}/2) \|x\|^2$. By Proposition \ref{partial_grad_v} and under Assumption \ref{asup3}, $\tilde{v}_\gamma(x,y)$ is strongly convex with respect to $x$.  $H(u,y)$ is the Fenchel conjugate of $\tilde{v}_\gamma(x,y)$ with respect to variable $x$  and is differentiable due to the differentiability  of ${v}_\gamma(x,y)$ and the strong convexity of ${v}_\gamma(x,y)$  with respect to $x$.
For the sequence $\{(x^k, y^k)\}$  generated by AGILS (\cref{alg1}),  let $u ^k = \nabla_x \tilde{v}_\gamma(x^k,y^{k+1})$. We have $\nabla_{u} H(u^k, y^{k+1}) = x^k$ and $\nabla_{y} H(u^k, y^{k+1}) = \nabla_y \tilde{v}_\gamma(x^k,y^{k+1})$.

The following lemma establishes a relaxed sufficient decrease property and relative error condition for the merit function $E_{p}(z,u)$ at $(z^{k+1}, u^k)$.

\begin{lemma}\label{lem5}
	Under Assumptions \ref{asup1}, \ref{asup2}, \ref{asup3} and  \ref{asup5}. Let $\epsilon > 0$. Let $\{(x^k, y^k)\}$ and $\{p_k\}$ be the sequences generated by AGILS under the same conditions as in Section \ref{sec4.2}, and let $E_{p}(z,u)$ be defined as in \eqref{def_merit}, In the case where $\tilde{y}^k = y^k$, 
	there exist constants $a, b >0$ such that
	\begin{equation}\label{lem5_eq_dec}
		\begin{aligned}
			E_{p_k}(z^{k+1}, u^k) +  a\left\|z^{k+1}-z^k\right\|^2 \le E_{p_k}(z^{k}, u^{k-1}) + \nu_{k},
		\end{aligned}
	\end{equation}
	where  $\nu_k =  \left ( \frac{(L_{f_x} + L_{g_2})^2}{ c_\alpha} + \frac{1}{c_\beta\gamma^2} \right)C_\eta^2\zeta_{k}^2$ with
	$\zeta_{k} := \max\{ s_k,  ( \prod_{i=0}^{k} \tau_i ) \mathcal{G}(\theta^{0}, x^{0}, y^{0})\}$, and
	\begin{equation}\label{lem5_eq_opt_err}
		\begin{aligned}
			{\rm dist}\left(0, \partial E_{p_k}\left( z^{k+1}, u^k \right) \right) \le b\|z^{k+1} - z^k\| + \tilde{\nu}_{k},
		\end{aligned}
	\end{equation}
	where $ \tilde{\nu}_{k} = (L_{f_x} + L_{g_2} + 1/\gamma)C_\eta \zeta_{k}$.
\end{lemma}
\begin{proof}
	We begin by deriving the first inequality. In this case, \( \tilde{y}^k = y^k \), the inequality~\eqref{lem1_eq8} simplifies to $\tilde{\psi}_{p_k}\left(x^{k}, y^{k+1}\right) \leq \tilde{\psi}_{p_k}\left(x^{k}, y^k\right)  -\left(1/\beta_k - L_{\psi_{y,k}}/2 - c_\beta/4\right) $\\
	$\left\|y^{k+1} - y^k\right\|^2 + C_\eta^2/(c_\beta\gamma^2) \zeta_k^2
	$, which implies 
	\begin{equation}\label{lem5_eq3}
		\begin{aligned}
			& G_{p_k}(x^k, y^{k+1}) -  {\tilde{v}}_\gamma(x^k, y^{k+1})  \\ \leq \, & G_{p_k}(x^k, {y}^{k}) -  {\tilde{v}}_\gamma(x^k, {y}^{k})  - \left(\frac{1}{\beta_k}-\frac{L_{\psi_{y,k}}}{2}- \frac{c_\beta}{4}\right)\left\|y^{k+1}-y^k\right\|^2 + \frac{C_\eta^2}{c_\beta\gamma^2} \zeta_{k}^2. 
		\end{aligned}
	\end{equation}
	Next, by Lemma~\ref{lem1}, \( \tilde{\psi}_{p_k}(x^{k+1}, y^{k+1})\le \tilde{\psi}_{p_k}(x^k, y^{k+1}) 
	+ \langle \nabla_x \tilde{\psi}_{p_k}(x^k, y^{k+1}) - d_x^k,\, x^{k+1} - x^k \rangle 
	- \left(1/\alpha_k - L_{\psi_{x,k}}/2\right)\|x^{k+1} - x^k\|^2 \). Moreover, since the gradients \( \nabla_x F(x, y^{k+1}) \), \( \nabla_x f(x, y^{k+1}) \), and \( \nabla_x g(x, y^{k+1}) \) are $L_{F_x}$, $L_{f_x}$ and $L_{g_1}$-Lipschitz continuous with respect to variable $x$ on $X$, respectively, we have
	\begin{equation}
		\begin{aligned}
			& G_{p_k}(x^{k+1}, y^{k+1}) + \left(\frac{1}{\alpha_k}-\frac{ L_{F_x}/p_k+L_{f_x}+L_{g_1} + \rho_{v_1} }{2}\right)\left\|x^{k+1}-x^k\right\|^2\\ \le\, & G_{p_k}(x^{k}, y^{k+1}) + \left\langle \nabla_{x} G_{p_k} (x^{k}, y^{k+1} ) - d_{x}^k, x^{k+1} - x^k \right \rangle.
		\end{aligned}
	\end{equation}
	Using the definitions of $G_{p}(z)$ and $\tilde{v}_\gamma(z)$,  along with the formulas for $\nabla_x{v}_\gamma(z)$ and $d_x^k$, as well as  \eqref{lem1_eq10} and applying the Cauchy-Schwarz inequality, we further derive
	\begin{equation}
		\begin{aligned}
			& G_{p_k}(z^{k+1}) - \langle \nabla_{x} \tilde{v}_\gamma(x^k, y^{k+1}), x^{k+1} - x^k \rangle + \left(\frac{1}{\alpha_k}-\frac{ L_{E_{x,k}} }{2} - \frac{c_\alpha}{4}\right)\left\|x^{k+1}-x^k\right\|^2\\ \le\, & G_{p_k}(x^{k}, y^{k+1}) + \frac{1}{c_\alpha}(L_{f_x} + L_{g_2})^2C_\eta^2 \zeta_{k}^2,
		\end{aligned}
	\end{equation}
	where $ L_{E_{x,k}}=L_{F_x}/p_k+L_{f_x}+L_{g_1}+\rho_{v_1}$.
	Combining this result with \eqref{lem5_eq3}, we have
	\begin{equation}\label{lem5_eq4}
		\begin{aligned}
			& G_{p_k}(z^{k+1}) - \tilde{v}_\gamma(x^k, y^{k+1}) - \langle \nabla_{x} \tilde{v}_\gamma(x^k, y^{k+1}), x^{k+1} - x^k \rangle \\ \le\, & G_{p_k}(z^k) - \tilde{v}_\gamma(z^{k}) + \left( \frac{1}{ c_\alpha}(L_{f_x} + L_{g_2})^2 + \frac{1}{c_\beta\gamma^2} \right) C_\eta^2\zeta_{k}^2 \\ &  - \left(\frac{1}{\alpha_k}- \frac{ L_{E_{x,k}} }{2} -\frac{c_\alpha}{4}\right)\left\|x^{k+1}-x^k\right\|^2 - \left(\frac{1}{\beta_k}-\frac{L_{\psi_{y,k}}}{2}-\frac{c_\beta}{4}\right)\left\|y^{k+1}-y^k\right\|^2 .
		\end{aligned}
	\end{equation} 
	By the convexity of $\tilde{v}_\gamma$ with respect to $x$ and $u^k = \nabla_x \tilde{v}_\gamma(x^k, y^{k+1})$, we have the identity $ - \tilde{v}_\gamma(x^k, y^{k+1}) + \langle \nabla_{x} \tilde{v}_\gamma(x^k, y^{k+1}),  x^k \rangle = H(u^k, y^{k+1}) $. Thus, the left-hand side of inequality~\eqref{lem5_eq4} equals $
	G_{p_k}(z^{k+1}) - \langle u^k, x^{k+1} \rangle + H(u^k, y^{k+1})$.
	Furthermore, since $ - \tilde{v}_\gamma(x^k, y^k) \le - \langle u^{k-1}, x^k \rangle + H(u^{k-1}, y^k)$, the term $G_{p_k}(z^k) - \tilde{v}_\gamma(z^{k})$ on the right-hand side of  \eqref{lem5_eq4} equals $
	G_{p_k}(z^k) -\langle u^{k-1}, x^k \rangle + H(u^{k-1}, y^k)$.
	
	Substituting these into \eqref{lem5_eq4} and recalling the definition of the merit function $E_{p_k}(z, u)$ directly yields the following inequality
	\begin{equation*} 
		\begin{aligned}
			&E_{p_k}(z^{k+1}, u^k) - E_{p_k}(z^k, u^{k-1}) - \left( \frac{1}{ c_\alpha}(L_{f_x} + L_{g_2})^2 + \frac{1}{c_\beta\gamma^2} \right) C_\eta^2\zeta_{k}^2\\ \le\, &    \left(\frac{1}{\alpha_k}- \frac{ L_{E_{x,k}} }{2} -\frac{c_\alpha}{4}\right)\left\|x^{k+1}-x^k\right\|^2 + \left(\frac{1}{\beta_k}-\frac{L_{\psi_{y,k}}}{2}-\frac{c_\beta}{4}\right)\left\|y^{k+1}-y^k\right\|^2.
		\end{aligned}
	\end{equation*}
	For $\alpha_k \in [\underline{\alpha},\ 2/(L_{\psi_{x,k}} + c_\alpha)]$ and recalling $L_{E_{x,k}}= L_{\psi_{x,k}} + c_\alpha/4$, we have
	\[
	\frac{1}{\alpha_k} - \frac{L_{E_{x,k}}}{2} - \frac{c_\alpha}{4} \ge \frac{L_{\psi_{x,k}}}{2} + \frac{c_\alpha}{2} - \frac{L_{\psi_{x,k}}}{2} - \frac{c_\alpha}{8} - \frac{c_\alpha}{4} = \frac{c_\alpha}{8}.
	\]
	Similarly, for  $\beta_k \in [\underline{\beta},\ 2/(L_{\psi_{y,k}} + c_\beta)]$,  the second coefficient satisfies
	\[
	\frac{1}{\beta_k} - \frac{L_{\psi_{y,k}}}{2} - \frac{c_\beta}{4}
	\ge \frac{L_{\psi_{y,k}}}{2} + \frac{c_\beta}{2}  - \frac{L_{\psi_{y,k}}}{2} - \frac{c_\beta}{4} =  \frac{c_\beta}{4}.
	\]
	Inserting these bounds into the previous descent estimate yields the inequality
	\begin{equation}\label{lem5_add1}
		\begin{aligned}
			E_{p_k}(z^{k+1}, u^k) +  c_\alpha/8 \cdot \left\|x^{k+1}-x^k\right\|^2 + c_\beta/4 \cdot \left\|y^{k+1}-y^k\right\|^2  \le E_{p_k}(z^{k}, u^{k-1}) + \nu_{k},
		\end{aligned}
	\end{equation}
	where $\nu_k =  \left( \frac{(L_{f_x} + L_{g_2})^2}{ c_\alpha} + \frac{1}{c_\beta\gamma^2}\right)C_\eta^2 \zeta_{k}^2$, and \eqref{lem5_eq_dec} follows.
	
	We now establish the relative error condition of $E_{p}(z,u)$. To begin, we provide the subdifferential characterization of $E_{p}(z,u)$,
	\begin{equation}\label{partial_E}
		\begin{aligned}
			\partial 	E_{p}(z, u) = \begin{pmatrix}
				\nabla_x G_p(z) - u + \mathcal{N}_{X}(x)\\ \partial_y G_p(z) + \nabla_{y}H(u,y) + \mathcal{N}_{Y}(y) \\ \nabla_u H(u,y) - x
			\end{pmatrix}.
		\end{aligned}
	\end{equation}
	Using the update rules of $x^{k+1}$ and $y^{k+1}$ in \eqref{update_x} and \eqref{update_y} together with the assumption $\tilde{y}^k = y^k$, and the formulas for $d_x^k$ and $d^k_y$ in \eqref{dx} and \eqref{dy}, we have 
	\begin{equation*}
		\begin{aligned}
			0 &\in \frac{1}{p_k}\nabla_x F(x^{k}, y^{k+1}) + \nabla_x \varphi(x^{k}, y^{k+1}) - \nabla_x \varphi(x^{k}, \theta^{k+1/2}) + \frac{x^{k+1} - x^k}{\alpha_k}+\mathcal{N}_{X}(x^{k+1}), \\
			0 &\in \frac{1}{p_k}\nabla_y F(z^{k}) + \nabla_y f(z^{k}) + \partial_y g(x^k,y^{k+1}) - \frac{\tilde{y}^k-\theta^{k}}{\gamma} + \frac{y^{k+1} - \tilde{y}^k}{\beta_k} + \mathcal{N}_{Y}(y^{k+1}).
		\end{aligned}
	\end{equation*}
	From this, we have $	\tilde{\xi}_x^k \in \nabla_x G_{p_k}(z^{k+1}) -u^k + \mathcal{N}_{X}(x^{k+1})$ with
	\begin{equation}
		\begin{aligned}
			\tilde{\xi}_x^k := \,&\frac{1}{p_k}\left( \nabla_x F(z^{k+1}) - \nabla_x F(x^{k}, y^{k+1}) \right) + \nabla_x \varphi(z^{k+1}) - \nabla_x \varphi(x^{k}, y^{k+1})\\
			& - \nabla_{x} v_\gamma(x^k, y^{k+1}) + \nabla_x \varphi(x^{k}, \theta^{k+1/2}) - \left( 1/\alpha_k -  \rho_{v_1} \right)(x^{k+1} - x^k).
		\end{aligned}
	\end{equation}
	By Assumption \ref{asup5}, for any $\xi_g^k \in \partial_y g(x^k,y^{k+1})$, there exists  $\tilde{\xi}_g^k \in \mathbb{R}^m$ satisfying $\|\tilde{\xi}_g^k\| \le L_{g_y} \| x^{k+1} - x^k\|$ such that ${\xi}_g^k + \tilde{\xi}_g^k \in \partial_y g(x^{k+1},y^{k+1}) $. Consequently, we have $\tilde{\xi}_y^k \in \partial_y G_{p_k}(z^{k+1}) + \nabla_{y}H(u^k,y^{k+1}) + \mathcal{N}_{Y}(y^{k+1})$, where 
	\begin{equation}
		\begin{aligned}
			\tilde{\xi}_y^k :=  \,&\frac{1}{p_k}\left( \nabla_y F(z^{k+1}) - \nabla_y F(z^{k}) \right) + \nabla_y f(z^{k+1}) - \nabla_y f(z^{k}) + \tilde{\xi}_g^k \\
			& - \nabla_{y} v_\gamma(x^k, y^{k+1}) + (\tilde{y}^k-\theta^{k})/\gamma - (y^{k+1} - \tilde{y}^k)/\beta_k.
		\end{aligned}
	\end{equation}
	By the $L_{F_x}$-, $L_{f_x}$- and $L_{g_1}$-Lipschitz continuity of $\nabla_x F(x, y^{k+1})$, $\nabla_x f(x, y^{k+1})$ and $\nabla_{x} g(x, y^{k+1})$  with respect to variable $x$ on $X$, along with $\alpha_k \ge \underline{\alpha}$, we have
	\begin{equation}\label{lem5_eq5}
		\begin{aligned}
			\| 	\tilde{\xi}_x^k \| \le \, & \left(  L_{F_x}/p_k+L_{f_x}+L_{g_1} + 1/\underline{\alpha} + \rho_{v_1} \right)\|x^{k+1} - x^k\| \\ &  + \| \nabla_{x} v_\gamma(x^k, y^{k+1}) - \nabla_x \varphi(x^{k}, \theta^{k+1/2}) \| \\
			\le \, & \left(  L_{F_x}/p_k+L_{f_x}+L_{g_1} + 1/\underline{\alpha} + \rho_{v_1} \right)\|x^{k+1} - x^k\| + (L_{f_x} + L_{g_2})C_\eta \zeta_{k} 
		\end{aligned}
	\end{equation}
	where the second inequality follows from \eqref{lem1_eq10}.
	For $\tilde{\xi}_y^k$, as $\nabla_y F(x^{k},y)$ and $\nabla_y f(x^{k},y)$ are $L_{F_y}$- and  $L_{f_y}$-Lipschitz continuous on $X \times Y$, and $\beta_k \ge \underline{\beta}$, we have
	\begin{equation}\label{lem5_eq6}
		\begin{aligned}
			\| 	\tilde{\xi}_y^k\| \le \, & \left(  L_{F_y}/p_k+L_{f_y} + 1/\underline{\beta} \right)\|z^{k+1} - z^k\| + \|\tilde{\xi}^k_g\| + \| \nabla_{y} v_\gamma(x^k, y^{k+1}) - (y^k-\theta_y^{k})/\gamma \| \\
			\le \, & \left(  L_{F_y}/p_k+L_{f_y} + 1/\underline{\beta}\right)\|z^{k+1} - z^k\| +  L_{g_y} \| x^{k+1} - x^k\| \\ 
			&+ \| \nabla_{y} v_\gamma(x^k, y^{k+1}) -\nabla_{y} v_\gamma(x^k, y^{k}) \|+ \| \nabla_{y} v_\gamma(x^k, y^{k}) - (y^k-\theta_y^{k})/\gamma \| \\
			\le \, & \left(  L_{F_y}/p_k+L_{f_y} +  L_{g_y} + L_{\theta^*}/\gamma + 1/\underline{\beta} + 1/\gamma \right)\|z^{k+1} - z^k\|   +  {C_\eta} \zeta_{k}/\gamma,
		\end{aligned}
	\end{equation}
	where the last inequality follows from \eqref{lem1_eq11} and the $L_{\theta^*}$-Lipschitz continuity of $\theta^*_{\gamma}(x,y)$ with respect to $y$ (from Lemma \ref{Lip_theta}). For the last component of $\partial	E_{p_k}(z^{k+1}, u^k)$, note that since $u^k = \nabla_x \tilde{v}_\gamma(x^k, y^{k+1})$, we have $ \nabla_u H(u^k,y^{k+1}) = x^k$. Thus
	\begin{equation}\label{lem5_eq7}
		\|\nabla_u H(u^k,y^{k+1}) - x^{k+1} \| = \| x^k - x^{k+1}\|.
	\end{equation}
	Finally, combining \eqref{lem5_eq5}, and \eqref{lem5_eq6}, and using the fact that $(\tilde{\xi}_x^k, \tilde{\xi}_y^k, x^k - x^{k+1}) \in \partial	E_{p_k}(z^{k+1}, u^k)$, the desired result \eqref{lem5_eq_opt_err} follows.
\end{proof}

To establish sequential convergence under the relaxed conditions \eqref{lem5_eq_dec} and \eqref{lem5_eq_opt_err}, we impose a stronger requirement on the inexact parameter $s_k$.  Specifically, $s_k$ satisfies
\begin{equation}\label{sk_assump}
	\sum_{k = 1}^\infty s_k< \infty \quad \text{and} \quad   \sum_{k = 1}^\infty k s_k^{p_s} < \infty, 
\end{equation} 
for some $p_s < 2$. For example, $s_k$ can be chosen as $s_k = 1/k^{1.1}$.
Under these conditions, we can choose a sufficiently large  $q > 0$ such that $2(1-1/q) > p_s$, ensuring that $\sum_{k = 1}^\infty k s_k^{2(1-1/q)} < \infty$. Furthermore, by applying the ratio test, as in the proof of Lemma \ref{lem1}, we can show that $\sum_{k = 1}^\infty k ((\prod_{i=0}^{k} \tau_i )\mathcal{G}(\theta^{0}, x^{0}, y^{0}))^{2(1-1/q)} < \infty$. Since $\zeta_{k} := \max\{ s_k,  (\prod_{i=0}^{k} \tau_i )\mathcal{G}(\theta^{0}, x^{0}, y^{0})\}$, it follows that $\sum_{k = 1}^\infty k \zeta_k^{2(1-1/q)} < \infty$.

The main difficulty in the convergence analysis arises from the inequality \eqref{lem5_eq_dec}. The presence of the error term $\nu_k$ means this condition does not enforce a sufficient decrease in the merit function $E_{p}$ at each iteration. To address this challenge, we introduce the following auxiliary merit function for $E_{p}(z,u)$,
\[
\tilde{E}_p(z, u, w) = E_{p}(z,u) + w^q, 
\]
where $	E_{p}(z,u) $ is as defined in \eqref{def_merit}.
Since $w^q$ is a KL function, it follows from \cite[Theorem 3.6]{wang2023calculus} (although this result considers the generalized concave KL property, it can be directly extended to the KL property using the same proof) that if $E_{p}(z,u)$ is a KL function, then $\tilde{E}_p(z, u, w)$ is also a KL function. 

We now present the global convergence theorem.  

\begin{theorem}\label{th3}
	Under Assumptions \ref{asup1}, \ref{asup2}, \ref{asup3} and \ref{asup5}, let $\epsilon > 0$. Consider the sequences $\{(x^k, y^k)\}$ and $\{p_k\}$ generated by AGILS (\cref{alg1}) under the same conditions as in Section \ref{sec4.2}, with $s_k$ satisfying condition \eqref{sk_assump}. Suppose the sequence  $\{(x^k,y^k)\}$ is bounded, and for sufficiently large $k$, $p_k = \bar{p}$
	and the feasibility correction procedure (Algorithm \ref{alg2}) is not applied ( i.e., $\tilde{y}^k = y^k$).
	If $E_{\bar{p}}(z,u)$ is a KL function, then the sequence $\left\{(x^k,y^k)\right\}$ converges to a	KKT point of $({\rm VP})_\gamma^{\tilde{\epsilon}}$ for some  $\tilde{\epsilon}\le\epsilon$. 
\end{theorem}

\begin{proof}
	By assumption, the penalty parameter \(p_k\) is eventually constant and the feasibility correction procedure is eventually inactive. Therefore, there exists \( k_0 \in \mathbb{N} \) such that for all \( k \ge k_0 \), we have \(p_k = \bar{p} := p_{k_0}\), and \(\tilde{y}^k = y^k\). From Lemma \ref{lem3}, we know that $\lim_{k \rightarrow \infty} \|z^{k+1} - z^k\|^2 = 0$.
	Let $\Omega$ denote the set of all limit points of the sequence $\{(z^k, u^{k-1})\}$. 
	From Proposition \ref{partial_grad_v}, we know that $\nabla v_\gamma$ is continuous. Given that the sequence $\{z^k\}$ is bounded, and $u^{k} = \nabla_x v_\gamma (x^k, y^{k+1}) + \rho_{v_1}x^k$, we conclude that the sequence $\{u^k\}$ is also bounded. Consequently, $\Omega$ is a bounded set.
	Since limit points form a closed set, it follows that $\Omega$ is compact and $\lim_{k \rightarrow \infty} \mathrm{dist}((z^k, u^{k-1}), \Omega) = 0$. 
	Now, define $w^k = (\sum_{i = k}^{\infty}\nu_i)^{1/q} $, since $\sum_{k}\nu_k < \infty$, we have $w^k \rightarrow 0$. We can also have $\tilde{E}_{\bar{p}}(z^k, u^{k-1}, w^k) = E_{\bar{p}}(z^k,u^{k-1}) + \sum_{i = k}^{\infty}\nu_i$. By Lemma \ref{lem5}, there exist constants $a, b > 0$ such that
	\begin{gather}
		\tilde{E}_{\bar{p}}(z^{k+1}, u^{k}, w^{k+1}) +  a\left\|z^{k+1}-z^k\right\|^2 \le \tilde{E}_{\bar{p}}(z^k, u^{k-1}, w^k), \quad \text{and} \label{thm3_eq_dec} \\
		{\rm dist}\left(0, \partial \tilde{E}_{\bar{p}}\left( z^{k+1}, u^k, w^{k+1} \right) \right) \le b\|z^{k+1} - z^k\| + \tilde{\nu}_{k} + q|w^{k+1}|^{q-1}. \label{thm3_eq_opt_err}
	\end{gather}
	Since $\tilde{E}_{\bar{p}}(z,u,w)$ is bounded below and continuous on $X \times Y \times \Re_+$, $\lim_{k \rightarrow \infty} \|z^{k+1} - z^k\| = 0$,  $\tilde{E}_{\bar{p}}(z^k,u^{k-1},w^k)$ is nonincreasing as $k$ increases, therefore, there exists a constant $\bar{E}$ such that for any subsequence $\{(z^{k_j}, u^{k_j-1}, w^{k_j})\}$ of sequence $\{(z^k,u^{k-1},w^k)\}$, we have $
	\bar{E} = \lim_{k_j \rightarrow \infty} \tilde{E}_{\bar{p}}(z^{k_j},u^{k_j-1},$\\
	$  w^{k_j}) = \lim_{k \rightarrow \infty} \tilde{E}_{\bar{p}}(z^k,u^{k-1},w^k)$.
	This implies that the function $\tilde{E}_{\bar{p}}(z,u,w)$ is constant on $\Omega \times \{0\}$.
	We can assume that $ \tilde{E}_{\bar{p}}(z^k,u^{k-1},w^k) > \bar{E}$ for all $k$. If there exists some $k > 0$ such that $\tilde{E}_{\bar{p}}(z^k,u^{k-1},w^k) = \bar{E}$,  then from \eqref{thm3_eq_dec}, we would have $\|z^{k+1} - z^k\| = 0$ for sufficiently large $k$, implying that the sequence $\{z^k\}$ converges and the conclusion follows immediately. 
	
	Since $\lim_{k \rightarrow \infty} \mathrm{dist}((z^k,u^{k-1},w^{k}), \Omega \times \{0\}) = 0$ and $\lim_{k \rightarrow \infty} \tilde{E}_{\bar{p}}(z^k,u^{k-1},w^{k}) = \bar{E}$, for any $\epsilon, \eta > 0$, there exists $k_0$ such that $\mathrm{dist}((z^k,u^{k-1},w^{k}), \Omega \times \{0\}) < \epsilon$ and $\bar{E} < \tilde{E}_{\bar{p}}(z^k,u^{k-1},w^{k})  < \bar{E} + \eta$ for all $k \ge k_0$.	
	Since $\tilde{E}_{\bar{p}}(z,u,w)$ satisfies the Kurdyka-\L{}ojasiewicz property at each point in $\Omega \times \{0\} $, and $\tilde{E}_{\bar{p}}(z,u,w)$ is equal to a finite constant on $\Omega \times \{0\}$, we can apply Lemma \ref{lem2} to obtain a continuous concave function $\phi$ such that for all $k \ge k_0$,
	\[
	\phi' \left(\tilde{E}_{\bar{p}}(z^k,u^{k-1},w^{k}) - \bar{E} \right) \mathrm{dist} \left( 0, \partial \tilde{E}_{\bar{p}}(z^k,u^{k-1},w^{k}) \right) 	\ge 1.
	\]
	For notational simplicity, we define $\tilde{E}_k := \tilde{E}_{\bar{p}}(z^k,u^{k-1},w^{k}) - \bar{E}$.
	By combining the above inequality with \eqref{thm3_eq_opt_err}, we obtain the following inequality,
	\[
	\phi' \left ( \tilde{E}_k \right ) 
	\cdot\,  \left(b\|z^{k} - z^{k-1}\| + \tilde{\nu}_{k-1} + q|w^{k}|^{q-1} \right)\ge 1.
	\]
	Next, using the concavity of $\phi$ and \eqref{thm3_eq_dec}, we obtain 
	\[
	\phi'(  \tilde{E}_k ) \cdot a \|z^{k+1} - z^k\|^2 \le \phi'( \tilde{E}_k ) \cdot \left( \tilde{E}_k - \tilde{E}_{k+1} \right) \le \phi\left(\tilde{E}_k  \right) -  \phi\left( \tilde{E}_{k+1}  \right).
	\]
	Combining the above two inequalities, we get
	\[
	\begin{aligned}
		&  \left(b\|z^{k} - z^{k-1}\| + \tilde{\nu}_{k-1} + q|w^{k}|^{q-1} \right) \left[ \phi\left(\tilde{E}_k  \right) -  \phi\left( \tilde{E}_{k+1}  \right) \right] \\
		\ge \, & \left(b\|z^{k} - z^{k-1}\| + \tilde{\nu}_{k-1} + q|w^{k}|^{q-1} \right) \cdot\phi'( \tilde{E}_k  ) \cdot a \|z^{k+1} - z^k\|^2 \ge a\|z^{k+1} - z^k\|^2.
	\end{aligned}
	\]
	Multiplying both sides of this inequality by $4/a$ and then taking the square root, we apply the inequality $2cd \le c^2 + d^2$ to obtain
	\[
	2\|z^{k+1} - z^k\| \le \|z^{k} - z^{k-1}\| +  \tilde{\nu}_{k-1}/b + q|w^{k}|^{q-1}/b+b/a \cdot \left[ \phi\left(\tilde{E}_k  \right) -  \phi\left( \tilde{E}_{k+1}  \right) \right].
	\]
	Summing up the above inequality for $k = k_0, \ldots, K$, since $\phi \ge 0$, we have
	\[
	\sum_{k = k_0}^K \|z^{k+1} - z^k\| \le  \|z^{k_0} - z^{k_0-1}\| + \frac{b}{a} \phi\left(  \tilde{E}_{k_0}\right)  + \sum_{k = k_0}^K \frac{\tilde{\nu}_{k-1}}{b} + \sum_{k = k_0}^K \frac{q|w^{k}|^{q-1}}{b}.
	\]
	Taking $K \rightarrow \infty$ in the above inequality, and using the fact that $ \sum_{k = 1}^\infty \tilde{\nu}_{k-1} < \infty$ and $ \sum_{k = 1}^\infty |w^{k}|^{q-1} = \sum_{k = 1}^\infty |(\sum_{i = k}^{\infty}\nu_i)^{1/q}|^{q-1} \le \sum_{k = 1}^\infty \sum_{i = k}^{\infty} \nu_i^{1-1/q} = \sum_{k = 1}^\infty k \nu_k^{1-1/q} \le C_\nu(\sum_{k = 1}^\infty k \zeta_k^{2(1-1/q)}) < \infty$ for some $C_\nu>0$, we obtain $\sum_{k = 1}^\infty \|z^{k+1} - z^k\| < \infty$.
	Thus, the sequence $\{z^k\}$ is a Cauchy sequence. Hence the sequence $\{z^k\}$ converges and we get the conclusion from Theorem \ref{th1}.
\end{proof}

	\begin{remark}
	To further substantiate the applicability of our sequential convergence theory in applications, we now discuss the KL property imposed on the merit function $E_{\bar{p}}(z,u)$ in Theorem \ref{th3}.
	A wide range of functions automatically satisfy the KL property, see, e.g., \cite{ attouch2010proximal, attouch2013convergence,bolte2010characterizations, bolte2014proximal}. Notably, as shown in \cite{bolte2007lojasiewicz, bolte2007clarke}, semi-algebraic function is an important subclass of KL functions.
	More importantly, the class of semi-algebraic functions is closed under various operations, see, e.g., \cite{attouch2010proximal,attouch2013convergence,bolte2014proximal}. Specifically, the indicator functions of semi-algebraic sets, finite sum and product of semi-algebraic functions, composition of semi-algebraic functions and partial minimization of semi-algebraic functions over semi-algebraic sets	are all semi-algebraic functions. 
	
	Based on these properties, if the objective functions $f$ and $g$ and the constraint set $Y$ in the lower-level problem are semi-algebraic, the Moreau envelope function $v_\gamma$ is also semi-algebraic. Additionally, if $F$ and the constraint set $X$ in the upper-level are semi-algebraic, the merit function $E_{\bar{p}}(z,u)$ defined in \eqref{def_merit}  is semi-algebraic and thus satisfies the KL property. Therefore, when $F$, $f$, $g$, $X$ and $Y$ in \eqref{eq1} are all semi-algebraic, the KL property of $E_{\bar{p}}(z,u)$ holds, ensuring the sequential convergence described in Theorem \ref{th3}.
\end{remark}

\section{Numerical experiments}
\label{sec:experiments}

In this section, we evaluate the performance of our proposed algorithm, AGILS, by comparing it with several existing methods on two problems: a toy example and the sparse group Lasso bilevel hyperparameter selection problem. AGILS utilizes an inexactness criterion that can be either absolute or relative. To explore this further, we also examine two variants: AGILS\_A, which exclusively employs an absolute inexactness criterion for approximating $\theta$, and AGILS\_R, which uses only a relative inexactness criterion for the same purpose. Both problems are specific instances of bilevel problem \eqref{eq1} and satisfy Assumptions \ref{asup1}, \ref{asup2}, \ref{asup3}, and \ref{asup5}.
The baseline methods for comparison include:
grid search, random search, TPE \cite{bergstra2013making}\footnote{\url{https://github.com/hyperopt/hyperopt}} (a Bayesian optimization method), IGJO \cite{feng2018gradient}\footnote{\url{https://github.com/jjfeng/nonsmooth-joint-opt}} (an implicit differentiation method), MEHA \cite{liumoreau} (a single-loop, gradient-based method), VF-iDCA \cite{gao2022value}\footnote{\url{https://github.com/SUSTech-Optimization/VF-iDCA}} (a method based on a difference-of-convex algorithm), 
and the MPCC approach (which reformulates the bilevel problem as a mathematical program with complementarity constraints).
Methods that require smoothness assumptions, such as those in \cite{hong2023two,lu2024first}, were excluded because the tested problems do not satisfy the required smoothness conditions.
For grid search, random search, and TPE, each subproblem was solved using the CVXPY package with the CLARABEL and SCS solvers. All experiments were implemented in Python and conducted on a laptop with an Intel  i7-1260P CPU (2.10 GHz) and 32 GB RAM.

\subsection{Toy example}
We consider the following toy example:
\begin{equation}\label{toy}
	\begin{aligned}
		\min\limits_{x\in \Re^n, y\in\Re^n,0\le x \le 1}  \sum\limits_{i=1}^n y_i  \quad	{\rm s.t. } ~~ y\in \underset{y \in \Re^n}{\mathrm{argmin}}\ \sum\limits_{i=1}^n \sqrt{(y_i - a_i)^2 + 1/{n^2}} + \sum\limits_{i=1}^n x_i|y_i|, 	
	\end{aligned}
\end{equation}
where $a \in \Re^n$ satisfies $a_i = -2 / n^{2/3}$ for $i \in 1,\dots, n/2$, and $a_i = 2 / n^{2/3}$ for $i \in n / 2 + 1, \ldots, n$. The optimal solution is $x^*_i = 0$, $y^*_i = a_i$ for $i = 1,\dots, n / 2$ and $x^*_i \in  [a_i / \sqrt{a_i^2 + 1/ n^2}, 1]$, $y^*_i = 0$ for $i=n / 2 + 1,\dots,n$. 

We compare the performance of AGILS and its two variants with several methods, including grid search, random search, TPE, IGJO, VF-iDCA, MEHA, and the MPCC approach. The performance of each method is evaluated using metric Error $:= {\rm dist}(z^{k},S^*) / \sqrt{1 + \min_{z\in S^*}\lVert z^*  \rVert^2}$, where $S^*$ is the optimal set of \eqref{toy}, and $z :=(x, y)$. For grid search, we set \( x_1 = x_2 = \dots = x_n = \rho \), and perform the search over $\rho $ on a grid of 100 uniformly spaced values within the interval \([0, 1]\). Random search is implemented by uniformly  sampling 100 points  from \([0, 1]^n\). For TPE, the search space is modeled as a uniform distribution over \([0, 1]^n\). In the MPCC approach, we first reformulate the toy example by introducing non-negative vectors \( u \) and \( v \) to represent \( y = u - v \), and transform \( \|y\|_1 \) into \( \sum_{i=1}^n (u_i + v_i) \). The resulting MPCC is solved using IPOPT \cite{wachter2006implementation}, with a relaxation parameter and tolerance of 0.1. VF-iDCA, MEHA, and IGJO use the same initial iterates as AGILS. For VF-iDCA, we set \( \beta_0 = 0.5 \), \( \rho = 0.5 \), \( c = 5 \), and \( \delta = 5 \). Since VF-iDCA and IGJO may terminate without converging to a true solution, we impose additional stopping rules. VF-iDCA stops when \( \lVert z^{k+1} - z^k \rVert / \sqrt{1 + \lVert z^k \rVert^2} < 10^{-4} \) or \( \mathrm{Error} < 1/n \), while IGJO stops when \( \lVert x^{k+1} - x^k \rVert / \sqrt{1 + \lVert x^k \rVert^2} < 10^{-4} \) or until \( \mathrm{Error} < 1/n \). For MEHA, and the parameters are searched over \( \tilde{c}_0 \in \{0.1, 0.5, 1, 5\} \), \( p \in \{0.1, 0.2, 0.3, 0.4\} \) and \( \tilde{\alpha} \), \( \tilde{\beta} \), \( \tilde{\eta} \), and \( \tilde{\gamma} \) scaled proportionally from the AGILS parameters \( {\alpha} \), \( {\beta} \), \( {\eta} \), and \( {\gamma} \) using scales  \(\{1, 1/2, 1/3, 1/4\}\). The optimal parameters that minimize the Error are selected as \( \tilde{c}_0 = 1 \), \( p = 0.2 \), \( \tilde{\gamma} = \gamma/3\), \( \tilde{\alpha} = \alpha/4 \), \( \tilde{\beta} = \beta \), and \( \tilde{\eta} = \eta \). MEHA terminates when $\mathrm{Error} < 1/n$. For AGILS and its variants, the initial iterates are $x_{0}=[0,\dots,0]$, $y_{0} =\theta_0 = a$, the parameters are set as $\epsilon = 10^{-6}$, $c_{\tilde{y}}=50\sqrt{n}$ and $p_0 = 0.5,\ \varrho_p = 0.02,\ c_p = 1$. The step sizes $\alpha, \beta, \gamma$ are chosen as $\alpha_k = 1/(L_{\psi_{x,k}} + c_\alpha)$, $\beta_k = 1/(L_{\psi_{y,k}} + c_\beta)$, and $\gamma = 1/(\rho_{f_2} +\rho_{g_2})$ where $L_{\psi_{x,k}} = L_{F_x}/p_k+L_{f_x}+L_{g_1}+\rho_{f_1} + \rho_{g_1}$, $L_{\psi_{y,k}} = L_{F_y}/p_k+L_{f_y}$  with the following values for the constants $L_{F_x} = L_{F_y} = L_{f_x} = \rho_{f_1} = \rho_{f_2} =L_{g_1}=0$, $\rho_{g_1} = \rho_{g_2} = 1$, $L_{f_y} = n$, $c_{\alpha} = c_{\beta} = 0.1$ and  $\eta = 1 / ( L_{f_y} + 1 / \gamma)$. AGILS and its two variants terminate when $\text{Error}<1/n$. The approximations $\theta$ are computed using the proximal gradient method to satisfy inexact criteria  \eqref{inexacty1} and \eqref{inexacty2}. The inexactness parameters: \( s_k = 0.05/(k+1)^{1.05} \), \( \tau_k = 20/(k+1)^{0.7} \). The $\tilde{y}$ and $\hat{\theta}$ in the feasibility correction procedure are also computed using proximal gradient method. 

Numerical results for two different problem dimensions, $n = 200$ and $n = 600$, are reported in Table \ref{tab-num-1}. The results demonstrate that AGILS and its variants perform well in both computational efficiency and solution accuracy, consistently achieving the lowest Error with the shortest computational time. MEHA performs second, but it requires tuning of parameters from a set of candidate values, which can be a limitation. This limitation stems from MEHA's stricter and implicit step size rules.

\begin{table}[ht]
\caption{Comparisons of different methods on the toy example for dimensions $n = 200$ and $n = 600$. AGILS(Orig) refers to the original AGILS method. Its variants, AGILS\_A and AGILS\_R, are denoted as AGILS(Abs) and AGILS(Rel), respectively.}
\label{tab-num-1}
\centering
{\small 
	\setlength{\tabcolsep}{1.7pt}  
	\renewcommand{\arraystretch}{1.0} 
	\begin{tabular}{c@{\hskip 6pt}ccccccc>{\centering\arraybackslash}m{0.8cm}>{\centering\arraybackslash}m{0.8cm}>{\centering\arraybackslash}m{0.8cm}}
		\hline
		\multicolumn{11}{c}{{\normalsize \rule{0pt}{8pt}Dimension $n = 200$}} \\
		\hline
		\multirow{2}{*}{\textbf{Method}} & \multirow{2}{*}{\textbf{Grid}} & \multirow{2}{*}{\textbf{Random}} & \multirow{2}{*}{\textbf{TPE}} & \multirow{2}{*}{\textbf{MPCC}} & \multirow{2}{*}{\textbf{IGJO}} & \multirow{2}{*}{\textbf{VF-iDCA}} & \multirow{2}{*}{\textbf{MEHA}} & \multicolumn{3}{c}{\textbf{AGILS}} \\
		\cline{9-11}
		& & & & & & & & \textbf{Orig}  & \textbf{Abs} & \textbf{Rel} \\
		\hline
		Time(s) & 0.59 & 5.89 & 29.61 & 20.99 & 6.06 & 0.52 & 0.09 & 0.07 &  0.08 & 0.08  \\
		Error & 0.70 & 0.75 & 0.72 & 0.02 & 0.82 & 0.21 & 0.00 &  0.00 & 0.00 & 0.00\\
		\hline
		\multicolumn{11}{c}{{\normalsize \rule{0pt}{8pt}Dimension $n = 600$}} \\
		\hline
		\multirow{2}{*}{\textbf{Method}} & \multirow{2}{*}{\textbf{Grid}} & \multirow{2}{*}{\textbf{Random}} & \multirow{2}{*}{\textbf{TPE}} & \multirow{2}{*}{\textbf{MPCC}} & \multirow{2}{*}{\textbf{IGJO}} & \multirow{2}{*}{\textbf{VF-iDCA}} & \multirow{2}{*}{\textbf{MEHA}} & \multicolumn{3}{c}{\textbf{AGILS}} \\
		\cline{9-11}
		& & & & & & & & \textbf{Orig} & \textbf{Abs} & \textbf{Rel} \\
		\hline
		Time(s) & 2.55 & 27.97 & 107.37 & 179.12  & 55.19 & 2.27 & 0.48 & 0.14  & 0.15 & 0.39\\
		Error & 0.71 & 0.79 & 0.76 & 1.00 & 0.69 & 0.10 & 0.00 & 0.00 & 0.00 & 0.00\\	
		\hline
	\end{tabular}
}
\end{table}

We report the average inner-to-outer iteration ratio, defined as \(\sum_{k=1}^K \ell_k / K\), where \(\ell_k\) is the number of inner iterations at the \(k\)-th outer step. Table~\ref{tab-num-5} shows the overall average and the averages over the first 200 and remaining outer iterations. Inner iterations tend to be more variable in the early phase and more stable later.

\begin{table}[ht]
\caption{Average inner-to-outer iteration ratio of AGILS and its variants for different problem dimensions.}
\label{tab-num-5}
\centering
{\small
	\setlength{\tabcolsep}{8pt}
	\renewcommand{\arraystretch}{1}
	\begin{tabular}{ccccccc}
		\hline
		\multirow{2}{*}{\textbf{Method}} & \multicolumn{3}{c}{Dimension \( n = 200 \)} & \multicolumn{3}{c}{Dimension \( n = 600 \)} \\
		\cline{2-4} \cline{5-7}
		& Overall & First 200 & After 200  & Overall & First 200  & After 200 \\
		\hline
		\textbf{AGILS} & 2.84 & 3.31 & 2.46 &  2.87 & 5.29 & 2.03\\
		\textbf{AGILS\_A} & 3.38 & 4.45 & 2.46 &  2.88 & 5.29 & 2.03\\
		\textbf{AGILS\_R} & 6.28 & 7.48 & 4.20 & 43.39 & 64.59 & 3.00\\
		\hline
	\end{tabular}
}
\end{table}

Next, we evaluate the effectiveness of the proposed inexact criterion \eqref{inexacty1} and \eqref{inexacty2} by comparing AGILS with two extreme variants that use different inexact criteria on the toy example with $n = 200$. One variant, denoted as AGILS-E, solves the proximal lower-level problem almost exactly, using a very small tolerance $10^{-6}$ in \eqref{inexacty1} and \eqref{inexacty2}, while the other, denoted as AGILS-S, solves the proximal lower-level problem using only a single proximal gradient step. The parameters and tolerances used are consistent with those described above.
The results in Figure \ref{fig-1} show that AGILS-E achieves the fastest error reduction in terms of iterations, but is slow in runtime due to the high computational cost of solving the proximal lower-level problem exactly. AGILS-S fails to converge due to its excessive inexactness in solving the proximal lower-level problem. AGILS achieves the fastest error reduction in terms of runtime.

\begin{figure}[!htb]
\centering
\begin{minipage}[t]{.42\linewidth}
	\includegraphics[width=1\textwidth]{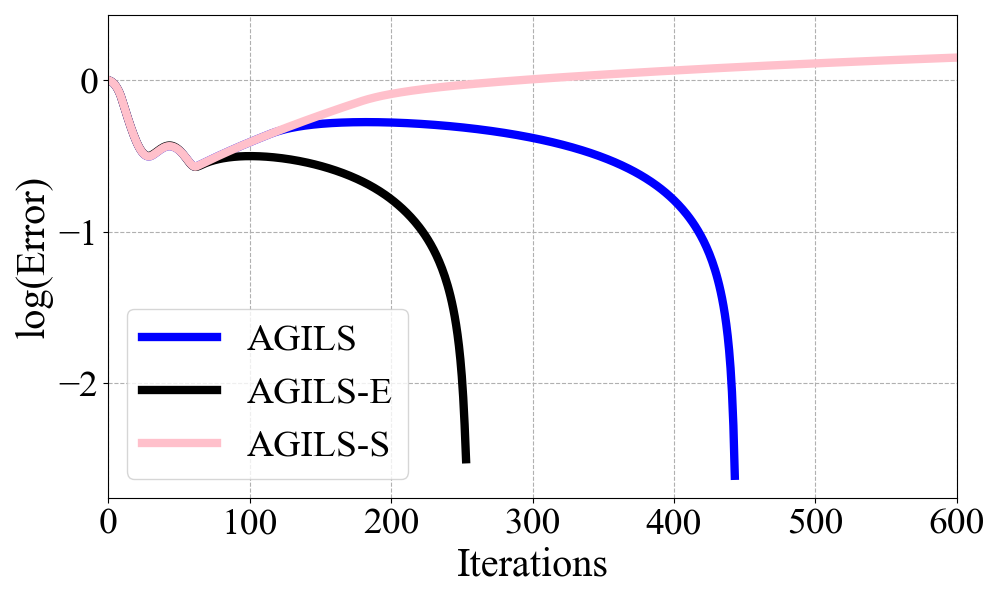}
\end{minipage}
\hspace{10pt}
\begin{minipage}[t]{.42\linewidth}
	\includegraphics[width=1\textwidth]{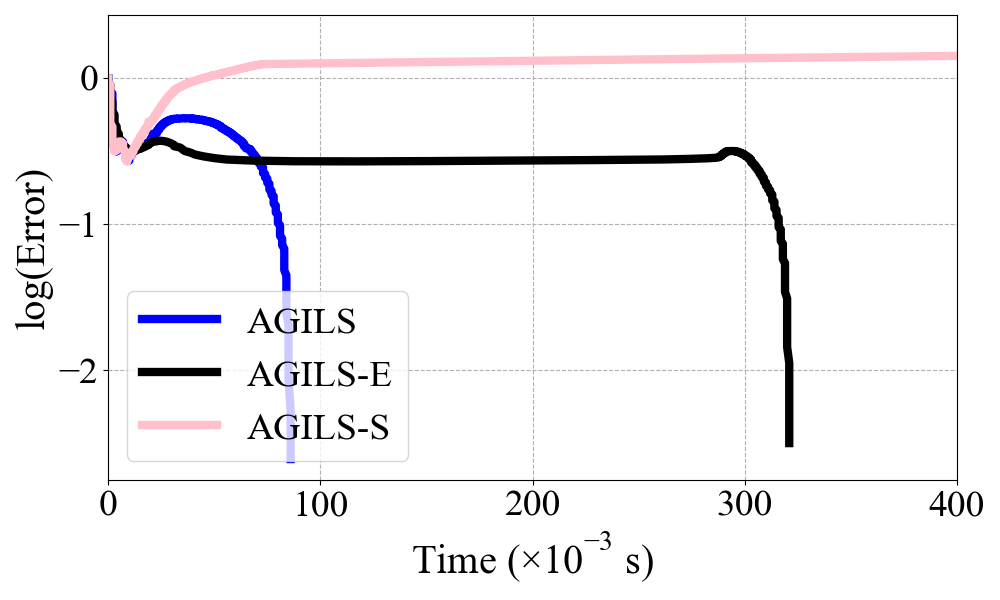}
\end{minipage}
\caption{Effectiveness of the inexact criterion in AGILS: comparison with two extreme variants}\label{fig-1}
\end{figure}

Finally, we evaluate the scalability of AGILS by testing it on toy examples with varying problem dimensions. The parameters of AGILS are as specified earlier, except that \( p_0 \) is set as 1. AGILS terminates when $\text{Error} < 1/n$. The results in Figure \ref{fig-2} show that computational time increases steadily as problem dimension grows. This demonstrates that AGILS remains both efficient and stable as the problem dimension increases, highlighting its ability to handle large-scale problems effectively. 

Furthermore, in all repeated runs of the toy example, the feasibility correction procedure, designed to enforce constraint satisfaction, was never triggered. 

\begin{figure}[!htb]
\centering
\begin{minipage}[t]{.42\linewidth}
\includegraphics[width=1\textwidth]{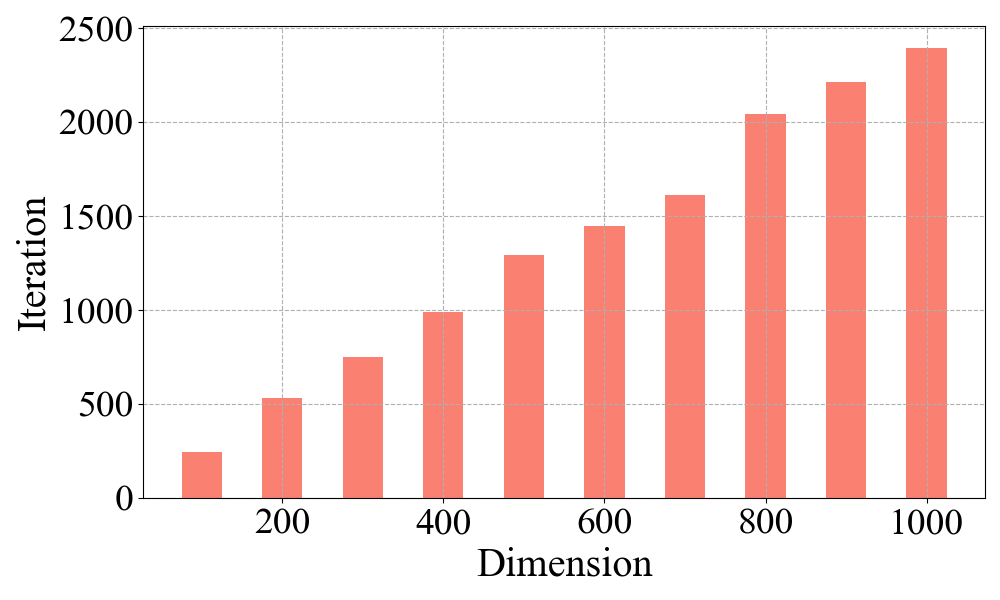}
\end{minipage}
\hspace{10pt}
\begin{minipage}[t]{.42\linewidth}
\includegraphics[width=1\textwidth]{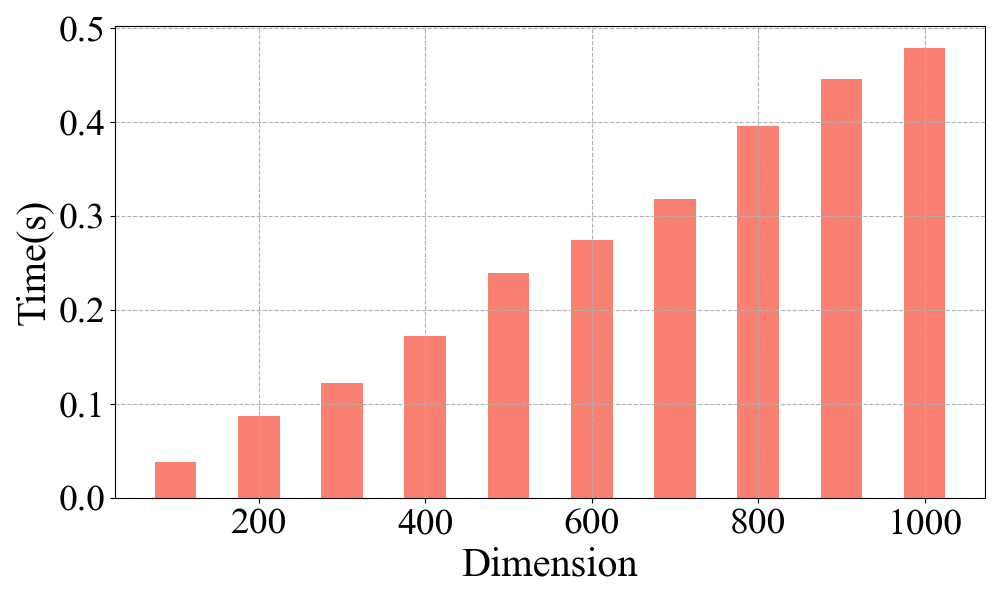}
\end{minipage}
\caption{Iteration and computational time of AGILS on the toy example with varying dimensions}\label{fig-2}
\end{figure}

\subsection{Sparse group Lasso}
The sparse group Lasso \cite{simon2013sparse} is a regularized regression model designed to achieve both individual feature sparsity and group-level sparsity. Given \( m \) features divided into \( J \) groups, let \( y^{(j)} \) and \( x_j \) denote the coefficients and regularization parameter of the \( j \)-th group. The corresponding bilevel hyperparameter selection problem is formulated as follows:     
\begin{equation*}
\begin{aligned}
\min_{ x \in {\Re}^{J+1}_{+}, y \in {\Re}^{m}} \, & \frac{1}{2|I_{\text{val}}|} \sum_{i \in I_{\text{val}}} \lvert b_{i} - y^{T} a_{i} \rvert ^{2} \\
{\rm s.t.}\quad\ &y \in \underset{\hat{y} \in {\Re}^{m}}{\mathrm{argmin}} \,  \frac{1}{2 |I_{\text{tr}}|} \sum_{i \in I_{\text{tr}}} \lvert b_{i} - \hat{y}^{T} a_{i} \rvert ^{2}  
+ \sum_{j=1}^{J}x_{j} \lVert \hat{y}^{\left(j\right)} \rVert_{2} + x_{J+1}  \lVert \hat{y} \rVert_{1},
\end{aligned} 
\end{equation*}
where $I_{\text{val}}$ and $I_{\text{tr}}$ are the sets of validation and training sample indices, respectively.
In this part, we compare the performance of our AGILS and its variants with several hyperparameter optimization methods, including grid search, random search, TPE, IGJO, VF-iDCA, and MEHA. The MPCC approach is excluded from the comparison due to its prohibitively high computational runtime. Following \cite{feng2018gradient}, the data generation process is as follows:
\( a_i \in \mathbb{R}^m \) is sampled from \( N(0, I) \), and the response \( b_i \) is computed as \( b_i = y^T a_i + \sigma \epsilon_i \), where noise \( \epsilon \) is sampled from \( N(0, I) \). The coefficient vector \( y \) is divided equally into five groups, with the $i$th group having the first $2i$ features set to value $2i$ and the remaining features set to zero, for $i = 1, \ldots, 5$. The noise level
\( \sigma \) is chosen to ensure \( \text{SNR} = \|Ay\|/\|b - Ay\| = 3 \).  The dataset is then randomly split into training, validation, and test sets of sizes \( n_{\text{tr}}  \), \( n_{\text{val}}  \), and \( n_{\text{test}}\), respectively. 

For grid search, random search and TPE, the search space is parameterized by \( \rho_i = \log_{10}(x_i) \) over the range \([-9, 2]^{J+1}\). Grid Search uses a \( 20 \times 20 \) grid with the constraint $\rho_1 = \rho_2 = \dots = \rho_M$; Random Search samples 400 points uniformly; and TPE models the search space as a uniform distribution.
All iterative methods (VF-iDCA, MEHA, IGJO, AGILS and its variants) share the initial point \(x_0 = [1, \dots, 1]\). For MEHA and the AGILS family, we also set \(y_{0} = \theta_0 = [1,\dots,1]\). VF-iDCA is configured with parameters \( \beta_0 = 5 \), \( \rho = 0.01 \), \( c = 1 \), and \( \delta = 0.1 \), and it terminates when \( \max\{ \lVert {z}^{k+1} - {z}^{k}\rVert/\sqrt{1 + \lVert {z}^k \rVert^2}, \tilde{t}^{k+1}/m \} < 0.1 \), where \( \tilde{t}^{k+1} \) is defined in \cite{gao2022value}. The IGJO algorithm is run with a maximum of 50 iterations and a minimum step size of  \( 10^{-6} \).
For MEHA, the parameters are searched over \( \tilde{c}_0 \in \{1, 10, 20\} \) and \( p \in \{0.1, 0.2, 0.3, 0.4\} \), while \( \tilde{\alpha} \), \( \tilde{\beta} \), \( \tilde{\eta} \), and \( \tilde{\gamma} \) scaled proportionally from the AGILS parameters \( {\alpha} \), \( {\beta} \), \( {\eta} \), and \( {\gamma} \) using scales \(\{1, 1/2, 1/4, 1/8\}\). 
The optimal parameters that minimize the validation error are selected as \( \tilde{c}_0 = 20 \), \( p = 0.1 \), \( \tilde{\gamma} = \gamma \), \( \tilde{\alpha} = \alpha \), \( \tilde{\beta} = \beta/8 \), and \( \tilde{\eta} = \eta/4 \). MEHA terminates when 
$\lVert \mathbf{z}^{k+1} - \mathbf{z}^{k}\rVert/\sqrt{1 + \lVert \mathbf{z}^k \rVert^2} < 0.005/m$.
AGILS and its variants use the same parameter settings as in the toy example, except for the penalty parameters and the following constants such as \( L_{F_x} = L_{f_x} = \rho_{f_1} = \rho_{f_2} = L_{g_1} = 0 \), \( L_{F_y} = \lambda_{\text{max}}(A_{\text{val}}^\top A_{\text{val}}) / n_{\text{val}} \), \( L_{f_y} = \lambda_{\text{max}}(A_{\text{tr}}^\top A_{\text{tr}}) / n_{\text{tr}} \), \( \rho_{g_1} = 1 \), \( \rho_{g_2} = m \). The penalty parameters are set as \( p_0 = 6 \), \( \varrho_p = 0.01 \). AGILS and its two variants terminate when $\lVert z^{k+1} - z^{k} \rVert/\sqrt{1 + \lVert z^{k}  \rVert^2} < 0.005/m$,  $z^k := (x^k,y^k)$ and $t^{k+1} < 0.1$. As in the toy example, $\theta$, $\tilde{y}$ and $\hat{\theta}$ are computed by the proximal gradient method so that the inexactness criteria \eqref{inexacty1}, \eqref{inexacty2}, \eqref{inexacty_y}, and \eqref{inexacty3} are satisfied. The inexactness parameters: \( s_k = 5/(k+1)^{1.05} \) and \( \tau_k = 10/(k+1)^{0.2} \).

The experiments are repeated 20 times with \( n_{\text{tr}} = 200 \), \( n_{\text{val}} = 200 \), \( n_{\text{test}} = 200 \), and \( m = 300 \). Results are reported in Table \ref{tab-num-2} correspond to the mean values computed over 20 independent runs, with standard deviations shown in parentheses. We report the validation error (``Val. Err.") and test error (``Test. Err."), computed as
$\sum_{i \in I_{\text{val}}} \lvert b_{i} - y^{T} a_{i} \rvert ^{2}/|I_{\text{val}}| $ and $\sum_{i \in I_{\text{test}}} \lvert b_{i} - y^{T} a_{i} \rvert ^{2}/|I_{\text{test}}| $, respectively. Both errors are evaluated using the sparse group Lasso estimator, which is obtained by solving the lower-level problem with the hyperparameter value $x^k$ produced by the algorithms. Additionally, we report a second type of test error, referred to as the infeasible test error (``Test. Err. Infeas.”) $ \sum_{i \in I_{\text{test}}} \left|b_i - (y^k)^\top a_i \right|^2/|I_{\text{test}}|$, where the lower-level variable values are taken directly from the iterates generated by  VF-iDCA, MEHA, AGILS and its variants. ``Feasibility"  is a scaled measure of the value function constraint violation, given by $(\varphi(x^k,y^k)-v(x^k))/|I_{\text{val}}|$ for VF-iDCA, and $(\varphi(x^k,y^k)-v_{\gamma}(x^k,y^k))/|I_{\text{val}}|$ for  MEHA, AGILS and its two variants. Note that this post-processing step is not part of the algorithm and is excluded from runtime. 

\begin{table}[!htb]
\caption{Comparison of different methods on the sparse group Lasso bilevel hyperparameter selection problem with $n_{\text{tr}} = n_{\text{val}} =n_{\text{test}} = 200, m = 300$.}
\label{tab-num-2}
\centering
\setlength{\tabcolsep}{10pt} 
\renewcommand{\arraystretch}{1.2} 
\resizebox{1.0\textwidth}{!}{
\begin{tabular}{cccccc}
			\hline
			\textbf{Method} & \textbf{Time(s)} & \textbf{Val. Err.} & \textbf{Test. Err.} & \textbf{Test. Err. Infeas.} & \textbf{Feasibility} \\
			\hline                     
			Grid & 50.73(3.65) & 168.26(23.91) & 170.10(21.28) &  - & -\\
			Random & 79.71(4.94) & 185.87(27.24) & 193.93(26.39) &  - & -\\
			TPE & 135.66(7.94) & 171.42(21.93) & 182.19(25.32) &  - & -\\
			IGJO & 173.73(90.01) & 139.43(12.32) & 169.62(21.11) &  - & -\\
			VF-iDCA & 84.71(47.81) & 136.16(32.39) & 148.85(31.06) & 129.46(15.47) & 0.08(0.04)\\
			MEHA & 13.15(1.87) & 98.41(8.83) & 158.63(13.52) & 158.39(13.76) & 0.00(0.00) \\
			AGILS & 12.25(2.11) & 95.93(9.71) & 153.75(12.72) & 153.14(12.24) & 0.00(0.00)\\
			AGILS\_A & 12.82(2.34) & 96.06(10.53) & 155.25(12.48) & 155.39(12.05) & 0.00(0.00)\\
			AGILS\_R & 12.37(2.22) & 95.93(9.71) & 153.75(12.72) & 153.13(12.24) & 0.00(0.00)\\
			\hline
		\end{tabular}
	}
\end{table}

As shown in Table \ref{tab-num-2}, our proposed method, AGILS, outperforms the others in terms of validation error while requiring the least computation time. The validation error, corresponding to the value of the upper-level objective at a feasible point derived from the hyperparameter value of the iterates, can be seen as an indicator of optimization performance. AGILS also achieves competitive test errors. Notably, both variants, AGILS\_A and AGILS\_R, exhibit similar performance to AGILS in terms of both validation and test errors. Interestingly, although VF-iDCA finds a parameter $y$ that achieves the lowest test error in the infeasible setting, the corresponding constraint violation is large. This suggests that VF-iDCA may possess good generalization performance, even when feasibility is not strictly enforced. MEHA performs similarly to AGILS, but slightly less effective. Moreover, MEHA's performance heavily relies on proper selection of algorithmic parameters, and selecting inappropriate parameters can significantly degrade its effectiveness. The average inner-to-outer iteration ratios for AGILS and its variants are approximately 2. 

Next, we investigate the impact of different solvers for solving the proximal lower-level problem in AGILS. We compare three algorithms: the Proximal Gradient Method (PGM), the Fast Iterative Shrinkage-Thresholding Algorithm (FISTA), and the Alternating Direction Method of Multipliers (ADMM). Each experiment is repeated 20 times, with all other parameters of AGILS held consistent with those used above. The results, presented in Table \ref{tab-num-4} as the mean and the standard deviation over the 20 runs, show that the AGILS framework is robust to different solvers, supporting flexibility and adaptability of the AGILS framework.

\begin{table}[htb!]
\caption{Comparison of AGILS with different proximal lower-level problem solvers on the sparse group Lasso bilevel hyperparameter selection problem with $n_{\text{tr}} = n_{\text{val}} =n_{\text{test}} = 200, m = 300$.}
\label{tab-num-4}
\centering
\setlength{\tabcolsep}{10pt} 
\renewcommand{\arraystretch}{1.2}
\resizebox{1.0\textwidth}{!}{
	\begin{tabular}{cccccc}
				\hline
				\textbf{Method} & \textbf{Time(s)} & \textbf{Val. Err.} & \textbf{Test. Err.} & \textbf{Test. Err. Infeas.} & \textbf{Feasibility} \\
				\hline
				AGILS(PGM) & 12.25(2.11) & 95.93(9.71) & 153.75(12.72) & 153.14(12.24) & 0.00(0.00)\\
				AGILS(FISTA) & 11.85(2.04) & 95.93(9.71) & 153.75(12.72) & 153.14(12.24) & 0.00(0.00) \\
				AGILS(ADMM) & 12.13(2.06) & 95.60(10.11) & 154.27(12.30) & 154.58(12.11) & 0.00(0.00)  \\
				\hline
			\end{tabular}
		}
	\end{table}

	Finally, we test the performance of AGILS on larger-scale problems with varying dimensions. The experiments, each repeated five times, are reported in Table \ref{tab-num-3}. AGILS is terminated when  \( \lVert z^{k+1} - z^{k} \rVert/\sqrt{1 + \lVert z^{k}  \rVert^2} < 0.1/ m\) and \(t^{k+1}< 0.1\). All other parameters of AGILS remain consistent with the settings described above. The results demonstrate AGILS's ability to efficiently solve larger-scale problems.
	
	During all repeated runs of the sparse group Lasso experiments, the feasibility correction procedure, introduced to enforce feasibility, 
	was never triggered.
	\begin{table}[htb]
		\centering
		\caption{AGILS on large-scale sparse group Lasso bilevel hyperparameter selection problems.}
		\label{tab-num-3}
		\setlength{\tabcolsep}{10pt}
		\renewcommand{\arraystretch}{1.2} 
		\resizebox{1.0\textwidth}{!}{
			\begin{tabular}{cccccc}
						\hline
						\textbf{Dimension} & \textbf{Time(s)} & \textbf{Val. Err.} & \textbf{Test. Err.} & \textbf{Test. Err. Infeas.} & \textbf{Feasibility}\\
						\hline		
						
						\makecell{$ n_{\text{tr}} = 1000,\ n_{\text{val}} =1000$\\ \Xhline{0.3pt}$\ n_{\text{test}} = 1000,\ m = 1500$} & 24.48(3.21) & 112.94(5.85) & 116.72(9.49) & 114.23(9.05) & 0.00(0.00)\\
						\hline
						\makecell{$ n_{\text{tr}} = 3000,\ n_{\text{val}} =3000$\\ \Xhline{0.3pt}$\ n_{\text{test}} = 3000,\ m = 4500$} & 133.71(11.18), & 106.89(2.91) & 112.11(3.06) & 109.93(2.69) & 0.00(0.00)\\
						\hline
						\makecell{$ n_{\text{tr}} = 5000,\ n_{\text{val}} =5000$\\ \Xhline{0.3pt}$\ n_{\text{test}} = 5000,\ m = 7500$} &   445.44(42.99) & 107.94(2.45) & 106.64(1.50) & 105.28(1.51) & 0.00(0.00)\\
						\hline
						\makecell{$ n_{\text{tr}} = 7000,\ n_{\text{val}} =7000$\\ \Xhline{0.3pt}$\ n_{\text{test}} = 7000,\ m = 10500$} &  787.01(74.59) & 105.71(1.80) & 105.91(2.15) & 104.72(2.14) & 0.00(0.00) \\
						\hline
					\end{tabular}
				}
			\end{table}

\bibliographystyle{plain}
\bibliography{references}

@article{ye2021variational,
  title={Variational analysis perspective on linear convergence of some first order methods for nonsmooth convex optimization problems},
  author={Ye, Jane J and Yuan, Xiaoming and Zeng, Shangzhi and Zhang, Jin},
  journal={Set-Valued and Variational Analysis},
  volume={29},
  number={4},
  pages={803--837},
  year={2021},
  publisher={Springer}
}

@article{alcantara2025unified,
	title={Unified smoothing approach for best hyperparameter selection problem using a bilevel optimization strategy},
	author={Alcantara, Jan Harold and Nguyen, Chieu Thanh and Okuno, Takayuki and Takeda, Akiko and Chen, Jein-Shan},
	journal={Mathematical Programming},
	volume={212},
	number={1},
	pages={479--518},
	year={2025},
	publisher={Springer}
}

@article{allende2013solving,
	title={Solving bilevel programs with the {KKT}-approach},
	author={Allende, Gemayqzel Bouza and Still, Georg},
	journal={Mathematical Programming},
	volume={138},
	pages={309--332},
	year={2013},
	publisher={Springer}
}

@article{attouch2010proximal,
	title={Proximal alternating minimization and projection methods for nonconvex problems: {A}n approach based on the {K}urdyka-{{\L}}ojasiewicz inequality},
	author={Attouch, H{\'e}dy and Bolte, J{\'e}r{\^o}me and Redont, Patrick and Soubeyran, Antoine},
	journal={Mathematics of Operations Research},
	volume={35},
	number={2},
	pages={438--457},
	year={2010},
	publisher={INFORMS}
}

@article{attouch2013convergence,
	title={Convergence of descent methods for semi-algebraic and tame problems: {P}roximal algorithms, forward--backward splitting, and regularized {G}auss--{S}eidel methods},
	author={Attouch, Hedy and Bolte, J{\'e}r{\^o}me and Svaiter, Benar Fux},
	journal={Mathematical Programming},
	volume={137},
	number={1},
	pages={91--129},
	year={2013},
	publisher={Springer}
}

@article{bai2025optimality,
author = {Kuang Bai and Jane J. Ye and Shangzhi Zeng},
title = {Optimality conditions for bilevel programmes via {M}oreau envelope reformulation},
journal = {Optimization},
volume = {74},
number = {12},
pages = {2685--2719},
year = {2025},
publisher = {Taylor \& Francis}
}

@book{bard1998practical,
title={Practical bilevel optimization: {A}lgorithms and applications},
author={Bard, Jonathan F},
year={1998},
publisher={Springer},
address={New York}
}

@book{beck2017first,
title={First-order methods in optimization},
author={Beck, Amir},
year={2017},
publisher={SIAM},
address = {Philadelphia}
}

@book{bennett2008bilevel,
title={Bilevel optimization and machine learning},
author={Bennett, Kristin P and Kunapuli, Gautam and Hu, Jing and Pang, Jong-Shi},
year={2008},
publisher={Springer},
address={Berlin}
}

@inproceedings{bergstra2013making,
title={Making a science of model search: {H}yperparameter optimization in hundreds of dimensions for vision architectures},
author={Bergstra, James and Yamins, Daniel and Cox, David},
booktitle={International {C}onference on {M}achine {L}earning},
year={2013}
}

@inproceedings{bertrand2020implicit,
	title={Implicit differentiation of {L}asso-type models for hyperparameter optimization},
	author={Bertrand, Quentin and Klopfenstein, Quentin and Blondel, Mathieu and Vaiter, Samuel and Gramfort, Alexandre and Salmon, Joseph},
	booktitle={International {C}onference on {M}achine {L}earning},
	year={2020}
}

@article{bolte2007lojasiewicz,
	title={The {{\L}}ojasiewicz inequality for nonsmooth subanalytic functions with applications to subgradient dynamical systems},
	author={Bolte, J{\'e}r{\^o}me and Daniilidis, Aris and Lewis, Adrian},
	journal={SIAM Journal on Optimization},
	volume={17},
	number={4},
	pages={1205--1223},
	year={2007},
	publisher={SIAM}
}

@article{bolte2007clarke,
	title={Clarke subgradients of stratifiable functions},
	author={Bolte, J{\'e}r{\^o}me and Daniilidis, Aris and Lewis, Adrian and Shiota, Masahiro},
	journal={SIAM Journal on Optimization},
	volume={18},
	number={2},
	pages={556--572},
	year={2007},
	publisher={SIAM}
}

@article{bolte2010characterizations,
title={Characterizations of {{\L}}ojasiewicz inequalities: {S}ubgradient flows, talweg, convexity},
author={Bolte, J{\'e}r{\^o}me and Daniilidis, Aris and Ley, Olivier and Mazet, Laurent},
journal={Transactions of the American Mathematical Society},
volume={362},
number={6},
pages={3319--3363},
year={2010}
}

@article{bolte2014proximal,
	title={Proximal alternating linearized minimization for nonconvex and nonsmooth problems},
	author={Bolte, J{\'e}r{\^o}me and Sabach, Shoham and Teboulle, Marc},
	journal={Mathematical Programming},
	volume={146},
	number={1},
	pages={459--494},
	year={2014},
	publisher={Springer}
}

@book{bonnans2013perturbation,
title={Perturbation analysis of optimization problems},
author={Bonnans, J Fr{\'e}d{\'e}ric and Shapiro, Alexander},
year={2013},
publisher={Springer},
address={New York}
}

@article{cecchini2013solving,
title={Solving nonlinear principal-agent problems using bilevel programming},
author={Cecchini, Mark and Ecker, Joseph and Kupferschmid, Michael and Leitch, Robert},
journal={European Journal of Operational Research},
volume={230},
number={2},
pages={364--373},
year={2013},
publisher={Elsevier}
}

@inproceedings{chen2021closing,
title={Closing the gap: {T}ighter analysis of alternating stochastic gradient methods for bilevel problems},
author={Chen, Tianyi and Sun, Yuejiao and Yin, Wotao},
booktitle={Advances in {N}eural {I}nformation {P}rocessing Systems},
year={2021}
}

@article{colson2007overview,
title={An overview of bilevel optimization},
author={Colson, Beno{\^\i}t and Marcotte, Patrice and Savard, Gilles},
journal={Annals of Operations Research},
volume={153},
pages={235--256},
year={2007},
publisher={Springer}
}

@article{constantin1995optimizing,
title={Optimizing frequencies in a transit network: {A} nonlinear bi-level programming approach},
author={Constantin, Isabelle and Florian, Michael},
journal={International Transactions in Operational Research},
volume={2},
number={2},
pages={149--164},
year={1995},
publisher={Elsevier}
}

@book{dempe2002foundations,
title={Foundations of bilevel programming},
author={Dempe, Stephan},
year={2002},
publisher={Springer},
address={New York}
}

@book{dempe2020bilevel,
title={Bilevel optimization: {A}dvances and next challenges},
author={Dempe, Stephan},
year={2020},
publisher={Springer},
address={Cham}
}

@article{dempe2013bilevel,
title={The bilevel programming problem: {R}eformulations, constraint qualifications and optimality conditions},
author={Dempe, Stephan and Zemkoho, Alain B},
journal={Mathematical Programming},
volume={138},
pages={447--473},
year={2013},
publisher={Springer}
}

@article{drusvyatskiy2018error,
title={Error bounds, quadratic growth, and linear convergence of proximal methods},
author={Drusvyatskiy, Dmitriy and Lewis, Adrian S},
journal={Mathematics of Operations Research},
volume={43},
number={3},
pages={919--948},
year={2018},
publisher={INFORMS}
}

@book{facchinei2007finite,
title={Finite-dimensional variational inequalities and complementadrity problems},
author={Facchinei, Francisco and Pang, Jong-Shi},
year={2007},
publisher={Springer},
address={New York}
}

@article{feng2018gradient,
title={Gradient-based regularization parameter selection for problems with nonsmooth penalty functions},
author={Feng, Jean and Simon, Noah},
journal={Journal of Computational and Graphical Statistics},
volume={27},
number={2},
pages={426--435},
year={2018},
publisher={Taylor \& Francis}
}

@article{fischer2022semismooth,
title={Semismooth {N}ewton-type method for bilevel optimization: {G}lobal convergence and extensive numerical experiments},
author={Fischer, Andreas and Zemkoho, Alain B and Zhou, Shenglong},
journal={Optimization Methods \& Software},
volume={37},
number={5},
pages={1770--1804},
year={2022},
publisher={Taylor \& Francis}
}

@article{fliege2021gauss,
title={Gauss--{N}ewton-type methods for bilevel optimization},
author={Fliege, J{\"o}rg and Tin, Andrey and Zemkoho, Alain B},
journal={Computational Optimization and Applications},
volume={78},
number={3},
pages={793--824},
year={2021},
publisher={Springer}
}

@inproceedings{franceschi2018bilevel,
title={Bilevel programming for hyperparameter optimization and meta-learning},
author={Franceschi, Luca and Frasconi, Paolo and Salzo, Saverio and Grazzi, Riccardo and Pontil, Massimiliano},
booktitle={International {C}onference on {M}achine {L}earning},
year={2018}
}

@inproceedings{gao2022value,
title={Value function based difference-of-convex algorithm for bilevel hyperparameter selection problems},
author={Gao, Lucy L and Ye, Jane J and Yin, Haian and Zeng, Shangzhi and Zhang, Jin},
booktitle={International {C}onference on {M}achine {L}earning},
year={2022}
}

@Article{gao2024,
title={Moreau Envelope Based Difference-of-\\weakly-Convex Reformulation and Algorithm for Bilevel Programs}, 
author={Lucy L. Gao and Jane J. Ye and Haian Yin and Shangzhi Zeng and Jin Zhang},
year={2024},
journal={preprint, ar{X}iv:2306.16761}
}

@article{garen1994executive,
title={Executive compensation and principal-agent theory},
author={Garen, John E},
journal={Journal of Political Economy},
volume={102},
number={6},
pages={1175--1199},
year={1994},
publisher={The University of Chicago Press}
}

@inproceedings{grazzi2020iteration,
title={On the iteration complexity of hypergradient computation},
author={Grazzi, Riccardo and Franceschi, Luca and Pontil, Massimiliano and Salzo, Saverio},
booktitle={International {C}onference on {M}achine {L}earning},
year={2020}
}

@article{hong2023two,
title={A two-timescale stochastic algorithm framework for bilevel optimization: {C}omplexity analysis and application to actor-critic},
author={Hong, Mingyi and Wai, Hoi-To and Wang, Zhaoran and Yang, Zhuoran},
journal={SIAM Journal on Optimization},
volume={33},
number={1},
pages={147--180},
year={2023},
publisher={SIAM}
}

@inproceedings{ji2021bilevel,
title={Bilevel optimization: {C}onvergence analysis and enhanced design},
author={Ji, Kaiyi and Yang, Junjie and Liang, Yingbin},
booktitle={International {C}onference on {M}achine {L}earning},
year={2021}
}

@article{jolaoso2025fresh,
title={A fresh look at nonsmooth {L}evenberg--{M}arquardt methods with applications to bilevel optimization},
author={Jolaoso, Lateef O and Mehlitz, Patrick and Zemkoho, Alain B},
journal={Optimization},
volume = {74},
number = {12},
pages = {2745--2792},
year = {2025},
publisher={Taylor \& Francis}
}

@article{kunapuli2008classification,
title={Classification model selection via bilevel programming},
author={Kunapuli, Gautam and Bennett, Kristin P and Hu, Jing and Pang, Jong-Shi},
journal={Optimization Methods \& Software},
volume={23},
number={4},
pages={475--489},
year={2008},
publisher={Taylor \& Francis}
}

@inproceedings{kwon2023fully,
title={A fully first-order method for stochastic bilevel optimization},
author={Kwon, Jeongyeol and Kwon, Dohyun and Wright, Stephen and Nowak, Robert D},
booktitle={International {C}onference on {M}achine {L}earning},
year={2023}
}

@inproceedings{liu2022bome,
title={Bome! bilevel optimization made easy: {A} simple first-order approach},
author={Liu, Bo and Ye, Mao and Wright, Stephen and Stone, Peter and Liu, Qiang},
booktitle={Advances in {N}eural {I}nformation {P}rocessing Systems},
year={2022}
}

@article{liu2021investigating,
title={Investigating bi-level optimization for learning and vision from a unified perspective: {A} survey and beyond},
author={Liu, Risheng and Gao, Jiaxin and Zhang, Jin and Meng, Deyu and Lin, Zhouchen},
journal={IEEE Transactions on Pattern Analysis and Machine Intelligence},
volume={44},
number={12},
pages={10045--10067},
year={2021},
publisher={IEEE}
}

@inproceedings{liumoreau,
title={Moreau Envelope for Nonconvex Bi-Level Optimization: {A} Single-Loop and {H}essian-Free Solution Strategy},
author={Liu, Risheng and Liu, Zhu and Yao, Wei and Zeng, Shangzhi and Zhang, Jin},
booktitle={International {C}onference on {M}achine {L}earning},
year={2024}
}

@inproceedings{lu2023slm,
title={{SLM}: {A} smoothed first-order {L}agrangian method for structured constrained nonconvex optimization},
author={Lu, Songtao},
booktitle={Advances in {N}eural {I}nformation {P}rocessing Systems},
year={2023}
}

@article{lu2024first,
title={First-order penalty methods for bilevel optimization},
author={Lu, Zhaosong and Mei, Sanyou},
journal={SIAM Journal on Optimization},
volume={34},
number={2},
pages={1937--1969},
year={2024},
publisher={SIAM}
}

@book{luo1996mathematical,
title={Mathematical programs with equilibrium constraints},
author={Luo, Zhi-Quan and Pang, Jong-Shi and Ralph, Daniel},
year={1996},
publisher={Cambridge University Press},
address={Cambridge}
}

@article{migdalas1995bilevel,
title={Bilevel programming in traffic planning: {M}odels, methods and challenge},
author={Migdalas, Athanasios},
journal={Journal of Global Optimization},
volume={7},
pages={381--405},
year={1995},
publisher={Springer}
}

@article{29,
title={The theory of moral hazard and unobservable behaviour: {P}art {I}},
author={Mirrlees, James A},
journal={Review of Economic Studies},
volume={66},
number={1},
pages={3--21},
year={1999},
publisher={Wiley-Blackwell}
}

@article{mordukhovich2023globally,
title={A globally convergent proximal {N}ewton-type method in nonsmooth convex optimization},
author={Mordukhovich, Boris S and Yuan, Xiaoming and Zeng, Shangzhi and Zhang, Jin},
journal={Mathematical Programming},
volume={198},
number={1},
pages={899--936},
year={2023},
publisher={Springer}
}

@article{ngai2000approximate,
title={Approximate convex functions},
author={Ngai, Huynh V and Luc, Dinh T and Th{\'e}ra, M},
journal={Journal of Nonlinear and Convex Analysis},
volume={1},
number={2},
pages={155--176},
year={2000}
}

@article{nurminskii1973quasigradient,
title={The quasigradient method for the solving of the nonlinear programming problems},
author={Nurminskii, Evgeni Alekseevich},
journal={Cybernetics and Systems Analysis},
volume={9},
number={1},
pages={145--150},
year={1973},
publisher={Springer}
}

@article{okuno2021lp,
title={On $\ell_p$-hyperparameter learning via bilevel nonsmooth optimization},
author={Okuno, Takayuki and Takeda, Akiko and Kawana, Akihiro and Watanabe, Motokazu},
journal={Journal of Machine Learning Research},
volume={22},
number={245},
pages={1--47},
year={2021}
}

@article{outrata1990numerical,
title={On the numerical solution of a class of {S}tackelberg problems},
author={Outrata, Ji{\v{r}}{\'\i}},
journal={Zeitschrift f{\"u}r Operations Research},
volume={34},
pages={255--277},
year={1990},
publisher={Springer}
}

@book{outrata1998nonsmooth,
title={Nonsmooth approach to optimization problems with equilibrium constraints: {T}heory, applications and numerical results},
author={Outrata, Jiri and Kocvara, Michal and Zowe, Jochem},
year={1998},
publisher={Springer},
address={New York}
}

@inproceedings{shen2023penalty,
title={On penalty-based bilevel gradient descent method},
author={Shen, Han and Chen, Tianyi},
booktitle={International {C}onference on {M}achine {L}earning},
year={2023}
}

@article{simon2013sparse,
title={A sparse-group {L}asso},
author={Simon, Noah and Friedman, Jerome and Hastie, Trevor and Tibshirani, Robert},
journal={Journal of Computational and Graphical Statistics},
volume={22},
number={2},
pages={231--245},
year={2013},
publisher={Taylor \& Francis}
}

@book{von2010market,
title={Market structure and equilibrium},
author={Von Stackelberg, Heinrich},
year={2010},
publisher={Springer},
address={Berlin}
}

@article{wachter2006implementation,
	title={On the implementation of an interior-point filter line-search algorithm for large-scale nonlinear programming},
	author={W{\"a}chter, Andreas and Biegler, Lorenz T},
	journal={Mathematical programming},
	volume={106},
	number={1},
	pages={25--57},
	year={2006},
	publisher={Springer}
}

@article{wang2023calculus,
	title={Calculus Rules of the Generalized Concave {K}urdyka--{{\L}}ojasiewicz Property},
	author={Wang, Xianfu and Wang, Ziyuan},
	journal={Journal of Optimization Theory and Applications},
	volume={197},
	number={3},
	pages={839--854},
	year={2023},
	publisher={Springer}
}

@inproceedings{yao2024overcoming,
title={Overcoming Lower-Level Constraints in Bilevel Optimization: {A} Novel Approach with Regularized Gap Functions},
author={Yao, Wei and Yin, Haian and Zeng, Shangzhi and Zhang, Jin},
booktitle={International {C}onference on {L}earning {R}epresentations},
year={2025}
}

@inproceedings{yaoconstrained,
title={Constrained Bi-Level Optimization: {P}roximal {L}agrangian Value Function Approach and {H}essian-free Algorithm},
author={Yao, Wei and Yu, Chengming and Zeng, Shangzhi and Zhang, Jin},
booktitle={International {C}onference on {L}earning {R}epresentations},
year={2024}
}

@Article{ye2023,
title={Difference of convex algorithms for bilevel programs with applications in hyperparameter selection},
author={Ye, Jane J and Yuan, Xiaoming and Zeng, Shangzhi and Zhang, Jin},
journal={Mathematical Programming},
volume={198},
number={2},
pages={1583--1616},
year={2023},
publisher={Springer}
}

@article{ye1995optimality,
title={Optimality conditions for bilevel programming problems},
author={Ye, Jane J and Zhu, DL},
journal={Optimization},
volume={33},
number={1},
pages={9--27},
year={1995},
publisher={Taylor \& Francis}
}

@article{zhang2024introduction,
title={An Introduction to Bilevel Optimization: {F}oundations and applications in signal processing and machine learning},
author={Zhang, Yihua and Khanduri, Prashant and Tsaknakis, Ioannis and Yao, Yuguang and Hong, Mingyi and Liu, Sijia},
journal={IEEE Signal Processing Magazine},
volume={41},
number={1},
pages={38--59},
year={2024},
publisher={IEEE}
}

@book{mordukhovich2018variational,
	title={Variational analysis and applications},
	author={Mordukhovich, Boris Sholimovich},
	year={2018},
	publisher={Springer},
	address={Cham}
}
\end{document}